\documentclass[12pt,letterpaper]{article}

\usepackage{amssymb,amsfonts,amscd,amsthm}
\usepackage[all,arc]{xy}
\usepackage{enumerate, bbm}
\usepackage{mathrsfs}
\usepackage{tikz}
\usepackage[left=0.9in,top=0.9in,right=0.9in,bottom=0.9in]{geometry}
\usepackage{mathtools}
\usepackage{hyperref}
\usepackage{color}
\usepackage{graphicx}
\usepackage{cleveref}
\usepackage{xfrac}
\usepackage{xcolor}
\usepackage{soul}
\usepackage{svg}
\usepackage{float}
\usepackage{bookmark}
\usepackage[style=trad-abbrv,doi=false,isbn=false,url=false]{biblatex}
\AtEveryBibitem{
	\clearlist{address}
	\clearfield{date}
	\clearfield{eprint}
	\clearfield{isbn}
	\clearfield{issn}
	\clearlist{location}
	\clearfield{month}
	
	\ifentrytype{book}{}{
		\clearlist{publisher}
		\clearname{editor}
	}
}

\usepackage{enumitem} 
\usepackage[font={small,it},labelfont=bf]{caption}
\graphicspath{ {TexImages/} }

\newtheorem{thm}{Theorem}[section]
\newtheorem{cor}[thm]{Corollary}
\newtheorem{prop}[thm]{Proposition}
\newtheorem{lem}[thm]{Lemma}
\newtheorem{claim}[thm]{Claim}

\newtheorem{quest}[thm]{Question}

\newtheorem{conj}[thm]{Conjecture}

\theoremstyle{definition}
\newtheorem{defn}{Definition}
\crefname{defn}{definition}{definitions}
\crefname{claim}{claim}{claims}

\hypersetup{
	colorlinks,
	citecolor=blue,
	filecolor=blue,
	linkcolor=blue,
	urlcolor=blue,
	linktocpage
}

\setenumerate[2]{label*=\arabic*.} 
\setlist[enumerate]{itemsep=2ex, topsep=2ex} 
\setlist[itemize]{itemsep=2ex, topsep=2ex}

\def\D{\mathcal{D}}
\def\HH{\mathcal{H}}
\def\G{\mathcal{G}}
\def\A{\mathcal{A}}
\def\I{\mathcal{I}}
\def\K{\mathcal{K}}
\def\S{\mathcal{S}}
\def\C{\mathcal{C}}
\def\a{\mathbf{a}}
\def\k{\mathbf{k}}
\def\RR{\mathbb{R}}

\newcommand{\cH}{\mathcal{H}}

\newcommand{\eps}{\varepsilon}
\newcommand{\ex}{\mathrm{ex}}

\newcommand{\R}{\mathbb{R}}
\newcommand{\N}{\mathbb{N}}

\newcommand{\Z}{\mathbb{Z}}

\newcommand{\E}{\mathbb{E}}

\newcommand{\al}{\alpha}

\newcommand{\gam}{\gamma}

\newcommand{\sig}{\sigma}
\newcommand{\ep}{\varepsilon}

\newcommand{\Lam}{\Lambda}

\newcommand{\Om}{\Omega}

\newcommand{\del}{\delta}

\renewcommand{\l}{\left}
\renewcommand{\r}{\right}

\newcommand{\half}{\frac{1}{2}}
\newcommand{\quart}{\frac{1}{4}}

\newcommand{\sm}{\setminus}
\newcommand{\sub}{\subseteq}

\renewcommand{\c}[1]{\mathcal{#1}}

\newcommand{\tr}[1]{\textrm{#1}}
\newcommand{\rec}[1]{\frac{1}{#1}}
\newcommand{\f}[2]{\frac{#1}{#2}}
\newcommand{\floor}[1]{\l\lfloor #1\r\rfloor}
\newcommand{\ceil}[1]{\l\lceil #1\r\rceil}

\DeclareMathOperator{\forest}{forest}

\title{The Random Tur\'an Problem for Theta Graphs}
\author{Gwen McKinley \thanks{Department of Mathematics,
		University of California, San Diego,
		La Jolla, CA. Email \url{gmckinley@ucsd.edu}.}\and 
	Sam Spiro\thanks{Department~of Mathematics, Rutgers University, Piscataway, NJ. Email: \url{sas703@scarletmail.rutgers.edu}. Research supported by the NSF Mathematical Sciences Postdoctoral Research Fellowships Program under Grant DMS-2202730.}}
\date{\today}
\begin{document}
	\maketitle
	\begin{abstract}
		Given a graph $F$, we define $\mathrm{ex}(G_{n,p},F)$ to be the maximum number of edges in an $F$-free subgraph of the random graph $G_{n,p}$.  Very little is known about $\mathrm{ex}(G_{n,p},F)$ when $F$ is bipartite, with essentially tight bounds known only when $F$ is either $C_4,\ C_6,\ C_{10}$, or $K_{s,t}$ with $t$ sufficiently large in terms of $s$, due to work of F\"uredi and of Morris and Saxton.  We extend this work by establishing essentially tight bounds when $F$ is a theta graph with sufficiently many paths.  
		Our main innovation is in proving a balanced supersaturation result for vertices, which differs from the standard approach of proving balanced supersaturation for edges.
	\end{abstract}
	
	\section{Introduction}
	
	Given a graph $F$, we define the \textit{Tur\'an number} or \textit{extremal number} $\ex(n,F)$ to be the maximum number of edges that an $n$-vertex $F$-free graph can have.  If $F$ is not bipartite, then the asymptotic behavior of $\ex(n,F)$ is determined by the Erd\H{o}s-Stone theorem~\cite{erdos1946structure}.  Only sporadic results for $\ex(n,F)$ are known when $F$ is bipartite, and in most cases these bounds are not tight.  
	
	For example, The K{\H{o}}v{\'a}ri-S\'os-Tur\'an theorem~\cite{kHovari1954problem} implies  $\ex(n,K_{s,t})=O(n^{2-1/s})$, and this bound is only known to be tight when $t$ is sufficiently large in terms of $s$; see for example \cite{bukh2021extremal}.   The bound $\ex(n,C_{2b})=O(n^{1+1/b})$ was first proven by Bondy and Simonovits~\cite{bondy1974cycles}.  It was shown by Faudree and Simonovits~\cite{faudree1983class} that this same upper bound continues to hold for theta graphs with paths of length $b$, and it was later shown by Conlon~\cite{conlon2019graphs} that this upper bound is tight for theta graphs which have sufficiently many paths.  A more detailed treatment on Tur\'an numbers of bipartite graphs can be found in the survey by F\"uredi and Simonovits~\cite{furedi2013history}.
	
	In this paper, we study a probabilistic analog of the Tur\'an number.  We define the random graph $G_{n,p}$ to be the $n$-vertex graph obtained by including each possible edge independently and with probability $p$, and we let $\ex(G_{n,p},F)$ denote the maximum number of edges in an $F$-free subgraph of $G_{n,p}$.  
	
	Observe that $\ex(G_{n,1},F)=\ex(n,F)$.  With this in mind, it is natural to ask if the classical bounds on $\ex(n,F)$ mentioned above can be extended to bounds on $\ex(G_{n,p},F)$ for all $p$.  For example, just as in the classical case, we have a complete asymptotic understanding of $\ex(G_{n,p},F)$ when $F$ is not bipartite due to  breakthrough work done independently by Conlon and Gowers~\cite{conlon2016combinatorial} and Schacht~\cite{schacht2016extremal}.  To formally state their result, we define the 2-density of a graph $F$ by \[m_2(F)=\max\left\{\f{e(F')-1}{v(F')-2}:F'\sub F,\ e(F')\ge 2\right\},\] and we write $f(n)\gg g(n)$ to mean $f(n)/g(n)$ tends to infinity as $n$ tends to infinity.  We also recall that a sequence of events $A_n$ occurs \textit{with high probability} or w.h.p.\ if $\Pr[A_n]$ tends to 1 as $n$ tends to infinity.
	\begin{thm}[\cite{conlon2016combinatorial,schacht2016extremal}]\label{thm:nonbipartite}
		For any graph $F$, w.h.p.\
		\[\ex(G_{n,p},F)= \begin{cases}\left(1-\f{1}{\chi(F)-1}+o(1)\right)p{n\choose 2} & p\gg n^{-1/m_2(F)},\\ (1+o(1))p {n\choose 2} & n^{-1/m_2(F)}\gg p\gg n^{-2}.\end{cases}\]
	\end{thm}
	Theorem~\ref{thm:nonbipartite} to a large extent solves the random Tur\'an problem for non-bipartie graphs, though many questions still remain in this setting; see the survey by R\"odl and Schacht~\cite{rodl2013extremal} for more on this.
	
	For the rest of this paper we focus on the random Tur\'an problem when $F$ is bipartite, a setting where much less is known.  This lack of information is partially due to the fact that the classical Tur\'an numbers $\ex(n,F)$ are unknown for almost all bipartite graphs.  An additional obstacle is that even if it is known that $\ex(n,F)=\Theta(n^\al)$ for some bipartite graph $F$, it is  typically not the case that $\ex(G_{n,p},F)=\Theta(p n^\al)$ for large $p$; see for example \Cref{fig:C4} which plots $\ex(G_{n,p},C_4)$.  That is, unlike in the non-bipartite case, the extremal constructions for $\ex(G_{n,p},F)$ with $F$ bipartite are far from the intersection of $G_{n,p}$ with an extremal $F$-free graph.

	\begin{figure}[htb]
		\centering
		\includegraphics[width=0.6\textwidth]{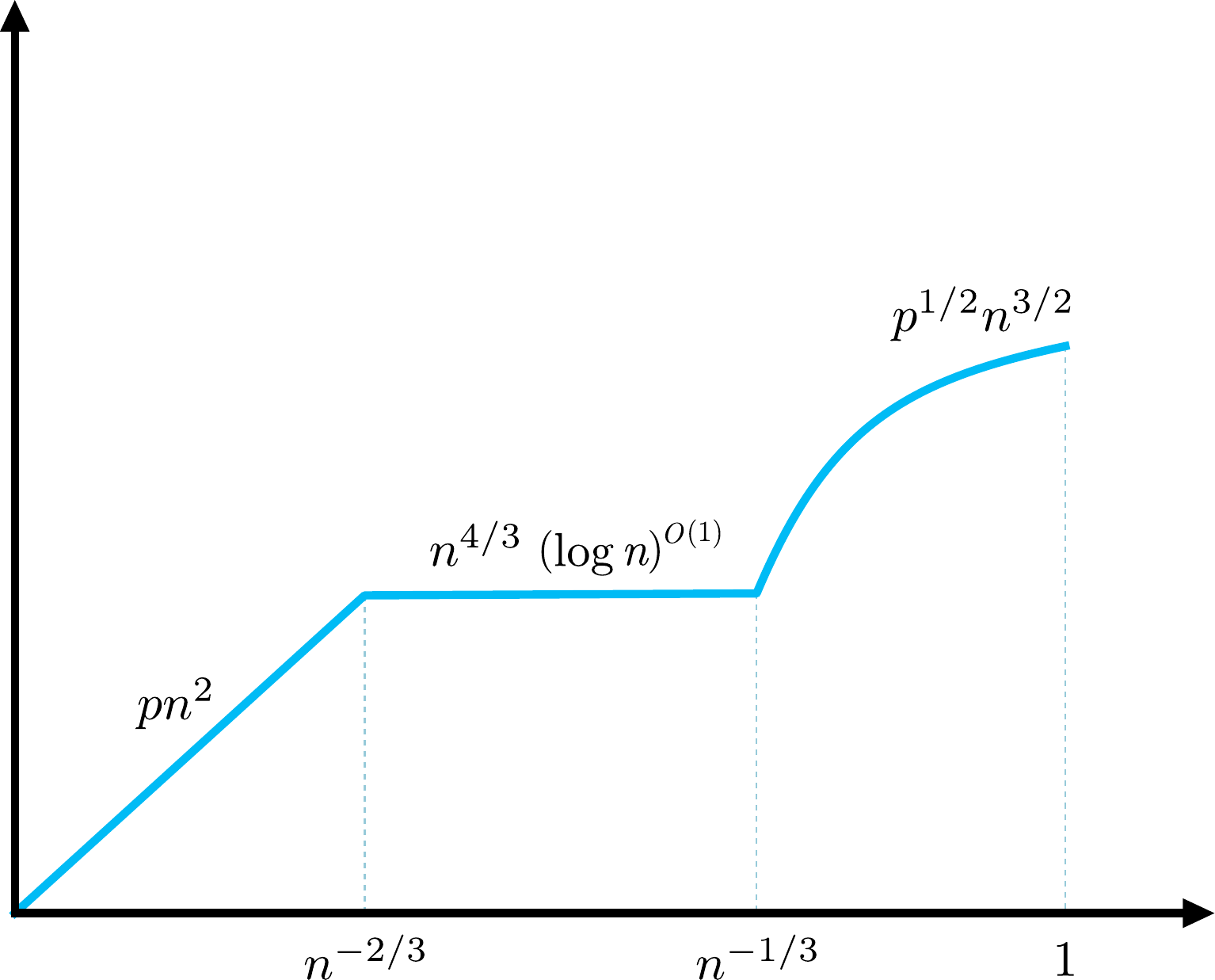}
		\caption{A plot of the value of $\ex(G_{n,p},C_4)$ as a function of $p$, with these bounds holding with high probability  and up to constant factors.}
		\label{fig:C4}
	\end{figure}
	
	Despite these obstacles, what is currently known about the random Tur\'an problem for bipartite graphs suggests that the following analog of \Cref{thm:nonbipartite} could be true; notice that in contrast to the non-bipartite case, where $\ex(G_{n,p},F)$ exhibits a single ``phase transition" around $p = n^{-1/m_2(F)}$, if \Cref{conj:main} is correct, we will see \textit{two} phase transitions in the bipartite case.
	
	\begin{conj}\label{conj:main}
		If $F$ is a graph with $\ex(n,F)=\Theta(n^\al)$ for some $\al\in (1,2]$, then w.h.p.\
		\[\ex(G_{n,p},F)= \begin{cases}\max\{\Theta(p^{\al-1}n^\al),n^{2-1/m_2(F)}(\log n)^{O(1)}\} & p\ge n^{-1/m_2(F)},\\ (1+o(1))p {n\choose 2} & n^{-1/m_2(F)}\gg p\gg n^{-2}.\end{cases}\]
		Equivalently, 
		\[\ex(G_{n,p},F)= \begin{cases}\Theta(p^{\al-1}n^\al) & p\ge n^{\f{2-\al-1/m_2(F)}{\al-1}} (\log n)^{O(1)}\\n^{2-1/m_2(F)}(\log n)^{O(1)} & n^{\f{2-\al-1/m_2(F)}{\al-1}}(\log n)^{O(1)} \ge p\ge n^{-1/m_2(F)},\\ (1+o(1))p {n\choose 2} & n^{-1/m_2(F)}\gg p\gg n^{-2}.\end{cases}\]
	\end{conj}
	
	The bound for $n^{-1/m_2(F)}\gg p\gg n^{-2}$ is easy to prove with a deletion argument, as is the lower bound of $\Om(n^{2-1/m_2(F)})$ when $p\ge n^{-1/m_2(F)}$.  Thus the hard part of \Cref{conj:main} is in showing that $\ex(G_{n,p},F)=\Theta(p^{\al-1}n^\al)$ whenever $p$ is large.  
	
	In terms of evidence supporting \Cref{conj:main}, F\"uredi~\cite{F} showed that this conjecture holds when $F=C_4$.  This work was substantially generalized by Morris and Saxton~\cite{morris2016number} who proved $\ex(G_{n,p},K_{s,t})=O(p^{1-1/s}n^{2-1/s})$ w.h.p.\ when $p$ is large, and that $\ex(G_{n,p},C_{2b})=O(p^{1/b}n^{1+1/b})$ w.h.p.\ when $p$ is large.  Moreover, they showed that these bounds are tight provided $\ex(n,K_{s,t})=\Om(n^{2-1/s})$ and $\ex(n,\{C_3,C_4,\ldots,C_{2b}\})=\Om(n^{1+1/b})$, respectively.  The lower bound of Conjecture~\ref{conj:main} was shown to hold for large powers of balanced trees by Spiro~\cite{spiro2022random}.  Under some mild conditions, Jiang and Longbreak~\cite{jiang2022balanced} proved a general upper bound of the form
	\[\ex(G_{n,p},F)=O\big(p^{1-m_2^*(F)(2-\al)}n^\al\big)\textrm{ w.h.p.\ when }p\tr{ is large},\]
	where $m_2^*(F)=\max\left\{\f{e(F')-1}{v(F')-2}:F'\subsetneq F,\ e(F')\ge 2\right\}$.  This bound matches Conjecture~\ref{conj:main} precisely when $m_2^*(F)=1$, which happens when $F=C_{2b}$ (giving a simpler proof of \cite{morris2016number}), but otherwise is strictly weaker than the bound proposed in Conjecture~\ref{conj:main}. As far as we are aware, these are the only known bounds for the random Tur\'an problem for bipartite graphs, though there have been a number of recent results regarding the analogous problem for degenerate hypergraphs, see for example \cite{MY,nie2023random,nie2023tur,NSV,SV-K22}.
	
	In this paper, we contribute to this growing body of literature by studying the random Tur\'an problem for theta graphs.  Recall that a theta graph $\theta_{a,b}$ is a graph which consists of two vertices $u,v$ together with $a$ internally disjoint paths from $u$ to $v$ of length $b$.  For example, $\theta_{3,4}$ is depicted in \Cref{fig:theta}.
	
	\begin{figure}[htb]
		\centering
		\includegraphics[width=0.2\textwidth]{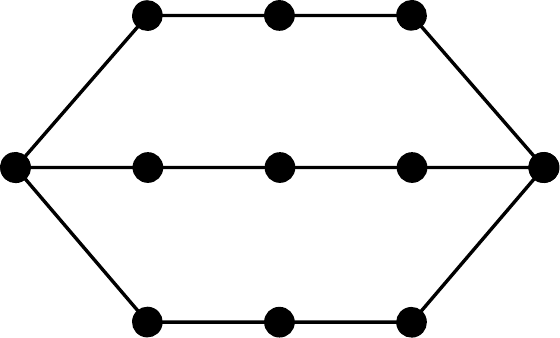}
		\caption{A depiction of the theta graph $\theta_{3,4}$.}
		\label{fig:theta}
	\end{figure}
	
	Observe that $\theta_{2,b}=C_{2b}$ and $\theta_{a,2}=K_{2,a}$.  Given that Morris and Saxton essentially solved the random Tur\'an problem for cycles and complete bipartite graphs, one might hope that their methods can be extended to give bounds for $\ex(G_{n,p},\theta_{a,b})$ in general.
	
	And indeed, by using similar ideas as in \cite{morris2016number}, Corsten and Tran~\cite{corsten2021balanced} implicitly proved
	\begin{equation}\ex(G_{n,p},\theta_{a,b})=O\left(p^{\f{2}{ab}} n^{1+\rec{b}}\right)\textrm{ w.h.p.\ when }p\textrm{ is large},\label{eq:oldBound}\end{equation}
	which matches the general upper bound given by Jiang and Longbreak in this case.  However, Faudree and Simonovits~\cite{faudree1983class} proved $\ex(n,\theta_{a,b})=O(n^{1+1/b})$, so Conjecture~\ref{conj:main} predicts that we should have $\ex(G_{n,p},\theta_{a,b})=O(p^{1/b}n^{1+1/b})$ when $p$ is large, which differs substantially from \eqref{eq:oldBound} when $a$ is large.

	By adding several new ideas (see \Cref{subsec:new}) to the approach of Morris and Saxton, we significantly improve upon the bounds of \eqref{eq:oldBound} and establish essentially tight (and unconditional) bounds for $\ex(G_{n,p},\theta_{a,b})$ when $a$ is large, which agree with what \Cref{conj:main} predicts.
	\begin{thm}\label{thm:main}
		For all $b\ge 2$, there exists $a_0=a_0(b)$ such that for any fixed $a\ge a_0$, w.h.p.
		\[\ex(G_{n,p},\theta_{a,b})=\begin{cases}
			\Theta\left(p^{\rec{b}}n^{1+\rec{b}}\right) & p\ge n^{-\f{b-1}{ab-1}}(\log n)^{2b},\\ 
			n^{2-\f{a(b-1)}{ab-1}}(\log n)^{O(1)} & n^{-\f{b-1}{ab-1}}(\log n)^{2b}\ge p\ge n^{-\f{a(b-1)}{ab-1}},\\ 
			(1+o(1))p{n\choose 2} &  n^{-\f{a(b-1)}{ab-1}}\gg p\gg n^{-2}.
		\end{cases}\]
	\end{thm}
	
	As far as we are aware, Theorem~\ref{thm:main} is the first result since Morris and Saxton~\cite{morris2016number} which gives essentially tight bounds on $\ex(G_{n,p},F)$ for any bipartite graph $F$.

	\subsection{Proof Outline}\label{subsec:outline}
	The vast majority of our paper is focused on proving the upper bound of \Cref{thm:main} when $p$ is large, with the other bounds following from previous results together with the monotonicity of $\ex(G_{n,p},\theta_{a,b})$.  Before going into our new ideas, we first recall the proof ideas from Morris and Saxton~\cite{morris2016number} for bounding $\ex(G_{n,p},C_{2b})$, as well as the adaptation of these methods by Corsten and Tran~\cite{corsten2021balanced} to theta graphs.
	
	\subsubsection{Previous Ideas}
	The main approach to upper bounding $\ex(G_{n,p},C_{2b})$ when $p$ is large is to show that $C_{2b}$ exhibits  ``balanced supersaturation'',  which roughly means that if $e(G)\gg \ex(n,C_{2b})$, then one can find a large collection of copies of $C_{2b}$ in $G$ which are ``spread out."  Given such a result, one can derive upper bounds on $\ex(G_{n,p},C_{2b})$ by using what is by now a somewhat standard argument involving hypergraph containers; see \Cref{thm:containers} below.
	
	To prove this balanced supersaturation result, let $G$ be an $n$-vertex graph with $kn^{1+1/b}$ edges.  Our goal is to build a hypergraph $\c{H}$ whose vertex set is $E(G)$, hyperedges are copies of  $C_{2b}$ in $G$, and is such that
	\[\deg_{\c{H}}(\sig)\le \f{k^{2b}n^2}{\del kn^{1+1/b}(\del k^{b/(b-1)})^{|\sig|-1}},\]
	for all $\sig\sub E(G)$, where $\del$ is some suitably small constant.  If we can construct such a collection with $|\c{H}|\approx k^{2b}n^2$, then this together with \Cref{thm:containers} will give our desired result.
	
	To construct such an $\c{H}$, we iteratively build copies of $C_{2b}$ to add to $\c{H}$ as follows.  Given our current collection $\c{H}$, we ``clean up'' $G$ by deleting edges $e$ with $d_{\c{H}}(e)\ge \f{k^{2b}n^2}{\del kn^{1+1/b}}$ (since we will not be able to use these edges in any copy of $C_{2b}$ to add to $\c{H}$), and by iteratively deleting vertices with low degree.  We then pick a vertex $x$ that will have some number $t(x)\le b$ associated to it (which roughly measures how well the graph expands near $x$) to use in our cycle.  We then run an algorithm which starts with a set $\chi=\{x\}$, and then iteratively adds new vertices to $\chi$ until we build a cycle $C$, with the exact algorithm we use depending on the value of $t(x)$.  A key point is that at each step of our algorithm, we always  have significantly more than $\del k^{b/(b-1)}$ choices for each new vertex to add.  Because we have so many choices, one of the cycles $C$ that we can create will be such that $C\notin \c{H}$ and such that $\{C\}\cup \c{H}$ continues to satisfy our desired codegree condtions.  By applying this algorithm repeatedly, we end up constructing a large collection $\c{H}$ of cycles which satisfies our codegree conditions, proving our desired balanced superaturation result.
	
	For theta graphs, Corsten and Tran~\cite{corsten2021balanced} used an argument similar to the one outlined above, but their approach only gave effective bounds on the codegrees $\deg_{\c{H}}(\sig)$ when $\sig$ is a set of edges inducing a forest\footnote{The general result of Jiang and Longbreak~\cite{jiang2022balanced} similarly only gives effective bounds for forests, and as such it seems like moving beyond the forest case is the main difficulty in general for upper bounding $\ex(G_{n,p},F)$.}.  At a high level, the fundamental issue with the approach outlined above is that the algorithm proceeds by selecting \textit{vertices} one at a time, but the codegree bounds of \Cref{thm:containers} are a function of the number of \textit{edges} $|\sig|$.  For $\sig$ which are forests, this distinction turns out to be irrelevant because $\sig$ has at least as many vertices as edges, but this property fails substantially for general subgraphs of theta graphs.   
	
	\subsubsection{New Ideas}\label{subsec:new}  As in previous works, our proof relies on showing that theta graphs exhibit balanced supersaturation.  In particular, we use the following three main ideas in order to get around the fundamental issues outlined above:
	
	\textbf{Vertex Supersaturation}.  The first key idea is that instead of viewing $\c{H}$ as a hypergraph on $E(G)$, we essentially view it as a hypergraph on $V(G)$.  To be more precise, we consider hypergraphs $\c{H}$ on $V(\theta_{a,b})\times V(G)$ in such a way that hyperedges of $\c{H}$ correspond to a unique labelled theta graph in $G$.  Balanced supersaturation for this hypergraph $\c{H}$ then easily translates to balanced supersaturation for the corresponding hypergraph on $E(G)$, which we ultimately need in order to use hypergraph containers.
	
	\textbf{Asymmetric Codegree Bounds}.  The second idea is that the codegree bounds we enforce on $\chi\sub V(\theta_{a,b})\times V(G)$ will not depend solely on the number of vertices $|\chi|$, but also on the vertices of $\theta_{a,b}$ which the set $\chi$ corresponds to.  For example, if $u,v\in V(\theta_{a,b})$ are the two high-degree vertices of $\theta_{a,b}$ (i.e.\ the vertices of degree at least 3), then the codegree bounds we enforce on the set $\chi=\{(u,x),(v,y)\}$ will be higher than those we would put on $\chi=\{(w,x),(w',y)\}$ for some other $w,w'\in V(\theta_{a,b})$ since, in some sense, $u,v$ are the most important vertices of $\theta_{a,b}$.  A similar idea was implicitly used by Morris and Saxton when proving balanced supersaturation results for $K_{s,t}$ (where the vertices in the smaller part are ``more important'' than those in the larger part).
	
	\textbf{Multiple Collections}. The final idea is that we do not build a single collection $\c{H}$, but instead build multiple collections $\c{H}_1,\c{H}_2,\ldots$ and impose different codegree bounds on each of these.  
	
	As a somewhat concrete example of why we do this, say we knew our graph $G$ was a random graph with $kn^{1+1/b}$ edges.  In this case, it would be possible to build an $\c{H}$ such that for every $x,y\in V(G)$, there are at most roughly $k^{ab}$ theta graphs in our collection which use $x,y$ as the two high degree vertices; and this is a strong bound on the codegree of this pair.  In contrast, if $G$ was a clique with $kn^{1+1/b}$ edges together with some isolated vertices, then it would be impossible to impose such a codegree bound for all $x,y$.  Thus if we only worked with a single collection $\c{H}$, we would have to pessimistically use the weaker codegree bounds that work for a clique when $x,y$ correspond to the high degree vertices, and similarly we would have to consider the worst possible choice of $G$ when determining the codegree bounds for any given set $\chi\sub V(G)$.  Doing this would ultimately give bounds that are too weak.  By building multiple collections, we can impose the ``correct'' codegree bounds regardless of the structure of $G$.
	
	\textbf{Other Ideas}. In addition to these three main ideas, we use a slightly different algorithm to construct our theta graphs compared to those of \cite{corsten2021balanced,morris2016number}.  Previous algorithms worked (roughly) by first specifying a vertex $x\in V(G)$ to play the role of one of the high degree vertices of $\theta_{a,b}$, then choosing a path in $G$ of length $b$ (which specifies the other vertex $y$ playing the role of a high degree vertex in $\theta_{a,b}$), and from there choosing the remaining $a-1$ paths from $x$ to $y$ one at a time.  Instead, our algorithm chooses the two high degree vertices $x,y$ at the start and then builds all of the $a$ paths from $x$ to $y$ one at a time.  This somewhat more symmetric argument allows us to overcome various technical issues that arose with previous approaches, and is crucial for our present argument to go through.
	
	\subsection{Organization and Notation}
	The rest of this paper is organized as follows. In \Cref{sec:expansion} we prove several auxiliary results that will be used in our main proof, and in \Cref{sec:prelim} we establish the main definitions used in this paper.  In \Cref{sec:main_proof}, which is the real heart of our argument, we establish our balanced supersaturation result for vertices.  We then translate this result into balanced supersaturation for edges in \Cref{sec:edgeSupersaturation} before completing the proof of \Cref{thm:main} in \Cref{sec:finish} by invoking a result that follows from a standard containers type argument.   A few open problems are given in \Cref{sec:conclusion}.
	
	Throughout the paper we adopt the following conventions.  We always use $u,v,w$ to denote vertices of $\theta_{a,b}$ and $x,y,z$ to denote vertices of a (larger) graph $G$.  Further, we will almost always let $u,v$ denote the two vertices of $\theta_{a,b}$ with degree at least 3 (when $a\ge 3$), and we will informally call these the ``vertices of high degree" in $\theta_{a,b}$.  We write $v(G)=|V(G)|$.  Whenever we write asymptotic notation such as $O(f)$, our implicit constants will always depend on $a,b$, and we will occasionally emphasize this point by writing, for example, $O_{a,b}(f)$.  For a hypergraph $\c{H}$ and a set of vertices $\sig\sub V(\c{H})$, its \textit{degree} or \textit{codegree} $\deg_{\c{H}}(\sig)$ is the number of hyperedges of $\c{H}$ containing $\sig$.
	
	\section{Auxiliary Results}\label{sec:expansion}
	Here we establish two results that will be crucial for our proof.  
	\subsection{Expansion}
	One of the key technical lemmas of Morris and Saxtion says that in a graph with sufficiently large minimum degree, there exists a vertex $x$ which is the endpoint of many ``nice'' paths of some length $t$.  Analogously, we will rely heavily on the following.
	
	\begin{prop}\label{prop:expansion}
		For all integers $b\ge 2$, there exists some $\ep>0$ such that the following holds.  If $G$ is an $m$-vertex graph with minimum degree $\ell m^{1/b}$ and $\ell\ge \ep^{-1}$, and if $\c{F}$ is a set of forests, then there exists an integer $2\le t\le b$ and a set of vertices $X$ such that the following properties hold:
		\begin{itemize}
			\item[(a)] For each $x\in X$, there exists a pair $(\c{B},\c{Q})$ such that $\c{B}=(B_0,\ldots,B_t)$ is a tuple of (not necessarily disjoint) vertex sets of $G$ with $B_0=\{x\}$, and $\c{Q}$ is a set of paths $x z_1\cdots z_t$ with $z_i\in B_i$ for all $i$.
			\item[(b)] We have $|B_{t-1}|,|B_t|\ge \ep  \ell^{(b-t+1)/(b-1)}m^{(t-1)/b}$.
			\item[(c)] We have $|\c{Q}|\ge \ep \ell^tm^{t/b}$.
			\item[(d)] For every $1\le i\le t$ and $z\in B_{i-1}$, we have $|N(z)\cap B_{i}|\ge \ep \ell m^{1/b}$, and for every $z\in B_t$, we have $|N(z)\cap B_{t-1}|\ge \ep \ell^{b/(b-1)}$.
			\item[(e)] For every $y\in B_t$, we have $|\c{Q}[x\to y]|\ge \ep \ell^{(t-1)b/(b-1)}$, where $\c{Q}[x\to y]$ denotes the set of paths of $\c{Q}$ starting at $x$ and ending at $y$. 
			\item[(f)] For any  $y\in B_t$ and non-empty set of vertices $S\sub V(G)\sm \{x,y\}$, there are at most $\ep^{-1}\ell^{(t-1-|S|)b/(b-1)}$ paths in $\c{Q}[x\to y]$ containing $S$.
			\item[(g)] If $\c{F}$ is such that for every path $x_1\cdots x_r$ of $G$ with $r\le b$ which does not contain an element of $\c{F}$ as a subgraph, there are at most $\ep \ell m^{1/b}$ vertices $x_{r+1}\in N_G(x_r)$ such that the path $x_1\cdots x_{r+1}$ contains an element of $\c{F}$ as a subgraph;  then no path of $\c{Q}$ contains an element of $\c{F}$ as a subgraph. \label{part:g}
			\item[(h)] We have  $|X|\ge \ep \ell^{(b-t)/b}m^{t/b}$.
		\end{itemize}
	\end{prop}
	
	Morris and Saxton essentially proved this same result but with the last condition replaced by $|X|>0$ as opposed to $|X|\ge \ep \ell^{(b-t)/b} m^{t/b}$.  Our two proofs will be essentially identical outside of this improved quantitative bound\footnote{It's possible that \Cref{prop:expansion} holds with the even stronger quantitative bound $|X|\ge \ep m$. 
		If true, this would significantly simplify our proof of \Cref{thm:main}; see \Cref{quest:largeX} for more on this.}, and as such we defer many of the redundant details of the proof to \Cref{append:Expansion}. 
	
	For our proof, we fix a sequence of rapidly decreasing constants \[1\ge \ep_b\ge \cdots \ge \ep_2\ge \ep_1>0\] which depend only on $b$.  The exact values of these constants are not particularly important, other than that they are sufficiently small with respect to 1 and with respect to each other.  In particular, we demand $\ep_t\ge 16(b+1)\ep_{t-1}$ for all $t$.  For the rest of the subsection we will fix some $m$-vertex graph $G$ with minimum degree $\ell m^{1/b}$ with $\ell$ (and hence $m$) sufficiently large in terms of the $\ep_t$ constants.
	
	\begin{defn}
		For $x\in V(G)$, we say that a tuple $\c{A}=(A_0,A_1,\ldots,A_t)$ of (not necessarily disjoint) subsets of $V(G)$ is a \textit{concentrated $t$-neighborhood of $x$} if $A_0=\{x\}$, $|A_t|\le \ell^{(b-t)/(b-1)}m^{t/b}$, and $|N(y)\cap A_i|\ge \ep_t \ell m^{1/b}$ for all $y\in A_{i-1}$.
		
		We define $t(x)$ to be the minimal $t\ge 2$ such that there exists a concentrated $t$-neighborhood of $x$ in $G$.  Note that $t(x)\le b$ for all $x$ since we can iterativly take $A_i=\bigcup_{y\in A_{i-1}} N(y)$.
	\end{defn}

	Morris and Saxton implicitly proved that for any vertex $x$ with $t(x)=\min_{y\in V(G)} t(y)$, there exist sets $(\c{B},\c{Q})$ as in Proposition~\ref{prop:expansion}, and in particular, at least one such vertex exists.  The only place where $t(x)=\min_{y\in V(G)} t(y)$ is used in their argument is in showing that there exists a tuple $\c{A}=(A_0,\ldots,A_t)$ with $t=t(x)$, $A_0=\{x\}$, $|A_t|\le \ell^{(b-t)/(b-1)}m^{t/b},\ |N(y)\cap A_i|\ge \ep_t \ell m^{1/b}$ for all $y\in A_{i-1}$, and (crucially) every $y\in \bigcup_{i=0}^t A_i$ has $t(y)\ge t$; in other words, for their argument to go through we only need that $t(x)$ achieves a \textit{local} minimum value among vertices $y$ near $x$, and it is not strictly necessary for $t(x)$ to achieve a \textit{global} minimum value.  
	
	Motivated by this, our main goal is to show that tuples with essentially these same properties noted above exist for many vertices $x$.  Specifically, we prove the following.
	
	\begin{lem}\label{lem:XSet}
		There exists some integer $2\le t\le b$ and some set $X\sub V(G)$ of size at least $\half (4b)^{t-b}\ell^{(b-t)/(b-1)}m^{t/b}$ such that $t(x)=t$ for every $x\in X$, and such that for every $x\in X$, there exists a tuple of sets $\c{A}=(A_0,\ldots,A_t)$ such that $A_0=\{x\}$, $|A_t|\le \ell^{(b-t)/(b-1)}m^{t/b}$, $|N(y)\cap A_i|\ge \half \ep_t\ell m^{1/b}$ for all $y\in A_{i-1}$,  and every $y\in \bigcup_{i=0}^t A_i$ has $t(y)\ge t$. 
	\end{lem}
	
	This differs ever so slightly from the condition that Morris and Saxton worked with since we only guarantee  $|N(y)\cap A_i|\ge \half \ep_t\ell m^{1/b}$ as opposed to $|N(y)\cap A_i|\ge \ep_t\ell m^{1/b}$.  By slightly adjusting the constants of Morris and Saxton, their same proof still carries over word for word for any vertex $x$ as in Lemma~\ref{lem:XSet}; see \Cref{append:Expansion} for details.  Thus to prove Proposition~\ref{prop:expansion}, it suffices to prove Lemma~\ref{lem:XSet}, which will be our goal for the rest of this subsection.
	
	For any integer $2\le t'\le b$, define
	\[\Lam(t')=(4b)^{t'-b}\ell^{(b-t')/(b-1)}m^{t'/b}.\]
	Note that $\ell m^{1/b}\le m$ (since $\ell m^{1/b}$ is the minimum degree of an $m$-vertex graph), i.e. $\ell^{1/(b-1)}\le m^{1/b}$, and thus $\Lam(t')$ is an increasing function in $t'$. From now on we let $t$ be the smallest integer such that there are at least $\Lam(t)$ vertices with $t(x)\le t$.  Note that $t=b$ satisfies these conditions, so such a (smallest) integer exists.
	
	Let $Y_0$ denote the set of vertices $x$ with $t(x)<t$.  Iteratively define $Y_i$ to be the set of vertices $x\notin \bigcup_{j=0}^{i-1} Y_j$ which have at least $\al\ell m^{1/b}$ neighbors in $Y_{i-1}$, where $\al:=\rec{2(b+1)} \ep_t$.  Note that every $x\in Y_i$ with $i\ge 1$ has $t(x)\ge t$ since $x\notin Y_0$.
	
	To motivate these definitions, we observe that in proving Lemma~\ref{lem:XSet} with $t$ as stated, we can not include any vertex of $Y_0$ in any of the $A_i$ sets.  While we are allowed to include vertices of $Y_1$ in these sets, these vertices are ``dangerous'' since a large number of their neighbors lie in $Y_0$, and similarly it is somewhat dangerous to include $Y_2$ since a large number of their neighbors are in $Y_1$, and so on.  We thus want to show that these $Y_i$ sets are all relatively small, which is accomplished by the following lemma.
	\begin{lem}
		If $t=2$ then $Y_i=\emptyset$ for all $i\ge 0$, and otherwise $|Y_i|\le \Lam(t-1)$ for all $i\ge 0$.
	\end{lem}
	For this proof, we note that by choosing $\ep$ sufficiently small in Proposition~\ref{prop:expansion}, we may assume $m\ge \ell^{b/(b-1)}\ge \ep^{-b/(b-1)}$ is sufficiently large compared to all of the constants $\ep_{t'}$.
	\begin{proof}
		If $t=2$ then $Y_0=\emptyset$, and hence inductively we have $Y_i=\emptyset$ for all $i$.  From now on we assume $t>2$. We prove the result by induction on $i$, the base case $|Y_0|\le \Lam(t-1)$ being immediate from the definition of $t$ and $Y_0$.  Say we have proven $|Y_i|\le \Lam(t-1)$ for some $i\ge 0$. The key technical observation we need is the following.
		\begin{claim}
			There exists a non-empty bipartite graph $G'\sub G$ with bipartition $S\cup T$ such that $S\sub Y_i$, $T\sub Y_{i+1}$, and such that $d_{G'}(y)\ge \rec{4} \al\ell m^{1/b}$ for $y\in T$ and $d_{G'}(y)\ge \rec{4} \al\ell m^{1/b} |Y_i|^{-1}|Y_{i+1}|$ for $y\in S$.
		\end{claim}
		\begin{proof}
			Let $G^*\sub G$ be the graph on $Y_i\cup Y_{i+1}$ obtained after deleting every edge within $Y_i$ and within $Y_{i+1}$.  Note that by definition, each vertex of $Y_{i+1}$ has at least $\al\ell m^{1/b}$ neighbors in $Y_i$ (which is disjoint from $Y_{i+1}$), so $e(G^*)\ge \al\ell m^{1/b}|Y_{i+1}|$.  Define $G'$ by iteratively deleting every vertex which violates the degree conditions of the claim.  Note that the number of edges deleted in this process is at most
			\[ \rec{4} \al\ell m^{1/b}\cdot |Y_{i+1}|+ \rec{4} \al\ell m^{1/b} |Y_i|^{-1}|Y_{i+1}|\cdot |Y_i|=\half \al\ell m^{1/b}|Y_{i+1}|<e(G^*).\]
			In particular, $G'$ is non-empty, and it satisfies all of the other properties by construction.
		\end{proof}
		Returning to our induction, we wish to show that $|Y_{i+1}|\leq \Lambda(t-1)$; our inductive hypothesis gives $|Y_i|\leq \Lambda(t-1)$, so it suffices to prove that $|Y_{i+1}|\leq |Y_i|$. Assume for contradiction that $|Y_{i+1}|> |Y_i|$. Let $x$ be any vertex of $T$ (which exists since $G'$ is non-empty), and let $A_0,\ldots,A_{t-1}$ be defined by $A_0=\{x\}$ and $A_j=\bigcup_{y\in A_{j-1}}N_{G'}(y)$.  Note that $A_j\sub S$ if and only if $j$ is odd since $G'$ is bipartite. Also note that for all $y\in A_{j-1}$ we have \[|N_G(y)\cap A_j|\ge |N_{G'}(y)\cap A_j|\ge \rec{4}\al\ell m^{1/b}\ge \ep_{t-1} \ell m^{1/b},\]
		where this last step used $\al=\rec{2(b+1)}\ep_t$ and that $\ep_{t-1}$ is sufficiently small compared to $\ep_{t}$.
		In particular, if $t>2$ is even, then $(A_0,\ldots,A_{t-1})$ is a concentrated $(t-1)$-neighborhood of $x$ since \[|A_{t-1}|\le |S|\le |Y_i|\le \Lam(t-1)\le \ell^{(b-t+1)/(b-1)}m^{(t-1)/b}.\]  This implies $t(x)<t$, a contradiction to $x\in T\sub Y_{i+1}$ since $Y_{i+1}$ is disjoint from $Y_0$.
		
		Thus we may assume $t>2$ is odd.  Define the random set $A'_{t-1}\sub Y_{i+1}$ by including each vertex of $Y_{i+1}$ independently and with probability $p=|Y_i||Y_{i+1}|^{-1}$, which is well defined since we assumed $|Y_{i+1}|> |Y_i|$. Observe that $|A'_{t-1}|$ is a binomial random variable with $|Y_{i+1}|$ trials and probability of success $p$.  Since $\E[|A'_{t-1}|]=p|Y_{i+1}|=|Y_i|\le \Lam(t-1)$, by Markov's inequality we have $\Pr[A'_{t-1}\ge 2\Lam(t-1)]\le 1/2$.  Thus for $m$ sufficiently large, we conclude that the event $|A'_{t-1}|< 2\Lam(t-1)\le \ell^{(b-t+1)/(b-1)}m^{(t-1)/b}$ occurs with probability at least 1/2.

		Similarly for each $y\in A_{t-2}\sub S$, the random variable $|N_G(y)\cap A_{t-1}'|$ is binomial with success probability $p$ and number of trials $d_{G'}(y)\ge \rec{4} \al\ell m^{1/b} p^{-1}$.  Thus by the multiplicative Chernoff inequality, we have for any $y\in A_{t-2}$,
		\[\Pr[|N_G(y)\cap A_{t-1}'|\le \rec{8}\al\ell m^{1/b}]\le e^{-\al\ell m^{1/b}/32},\]
		and for $m$ sufficiently large this probability is at most $.1 m^{-1}$. By a union bound over $y\in A_{t-2}$, we see that with probability at least $.9$, every vertex $y\in A_{t-2}$ satisfies $|N_G(y)\cap A_{t-1}'|\ge \rec{8} \al\ell m^{1/b}\ge \ep_{t-1} \ell m^{1/b}$.
		
		In total we conclude that there exists some choice of $A'_{t-1}\sub Y_{i+1}$ such that both $|A'_{t-1}|\le \ell^{(b-t+1)/(b-1)}m^{(t-1)/b}$ and $|N_G(y)\cap A_{t-1}'|\ge \ep_{t-1} \ell m^{1/b}$ for all $y\in A_{t-2}$ (since in particular, this holds with positive probability for a random subset of $Y_{i+1}$).  Thus $(A_0,\ldots,A_{t-2},A_{t-1}')$ is a concentrated $(t-1)$-neighborhood of $x$.  This implies $t(x)<t$, which again is a contradiction.  We conclude $|Y_{i+1}|\le |Y_i|$, and hence $|Y_{i+1}|\le \Lam(t-1)$ by the inductive hypothesis.
	\end{proof}
	We are now ready to prove Lemma~\ref{lem:XSet}.  
	
	\begin{proof}[Proof of Lemma~\ref{lem:XSet}]
		Let $X$ be the set of vertices $x$ with $t(x)=t$ and $x\notin \bigcup_{i=1}^{b} Y_i$.
		
		We claim that $|X|\ge \half \Lam(t)$.  Indeed, by definition of $t$,  there are at least $\Lam(t)$ vertices with $t(x)\le t$.  Every vertex with $t(x)\le t$ is either in $X$ or $\bigcup_{i=0}^b Y_i$, so by the previous lemma,
		\[|X|\ge \Lam(t)-|\bigcup_{i=0}^b Y_i|\ge \Lam(t)-(b+1)\Lam(t-1)\ge \half \Lam(t).\]
		
		It remains to find the tuple of sets $\c{A}$ guaranteed by Lemma~\ref{lem:XSet} for each $x\in X$.  For each $x\in X$, by definition of $t(x)=t$, there exists a tuple $(A'_0,A'_1,\ldots,A'_t)$ with $A'_0=\{x\}$,  $|A'_t|\le \ell^{(b-t)/(b-1)}m^{t/b}$, and $|N(y)\cap A_i'|\ge \ep_t\ell m^{1/b}$ for all $y\in A_{i-1}'$.  Define $A_i=A'_i\sm \bigcup_{j=0}^{b-i} Y_j$.  Note that with this we have $A_0=\{x\}$, $|A_t|\le \ell^{(b-t)/(b-1)}m^{t/b}$, and no $y\in A_i$ has $t(y)<t$ because we removed $Y_{0}$ from each $A_i$.  
		
		It remains to show that each $y\in A_{i-1}$ has many neighbors in $A_i$.  Since each $y\in A_{i-1}$ does not belong to any $Y_j$ with $1\le j\le b-i+1$, by definition $y$ has at most $(b+1)\al\ell m^{1/b}$ neighbors in $\bigcup_{j=0}^{b-i} Y_j$.   This implies \[|N(y)\cap A_i|=|N(y)\cap (A_i'\sm\bigcup_{j=0}^{b-i} Y_j)| \ge |N(y)\cap A_i'|-(b+1)\al\ell m^{1/b}\ge (\ep_t-(b+1)\al)\ell m^{1/b}= \half \ep_t \ell m^{1/b},\]
		proving the result.
	\end{proof}
	
	\subsection{Minimum Degrees}
	A very minor step in the proof of Morris and Saxton calls for deleting vertices of low degree in $G$.  In their setting this is fine, as this does not significantly decrease the number of edges in $G$.  However, because the focus of our approach is on balanced supersaturation for vertices rather than for edges, we will need to be more careful with this step.
	
	Towards this end, we use the reduction lemma stated below, which guarantees a subgraph $G'\sub G$ of large minimum degree, where the degree condition is stronger the more vertices are removed from $G$.  In particular, the tradeoff is roughly what one would expect if $G$ was a clique $G'$ together with some number of isolated vertices.
	
	As a small technical convenience, we will prove this lemma in the more general setting of multigraphs with loops.  Here, the degree of a vertex $v$ is the number of edges incident to $v$ (so each loop contributes 1 to its degree).
	
	\begin{lem}\label{lem:reduction}
		Let $G$ be an $n$-vertex multigraph with loops. For all $b\ge 1$, there exists a subgraph $G'\sub G$ with $v(G')>0$ and minimum degree at least \[2^{-b}\left(\frac{v(G')}{n}\right)^{1/b}\f{e(G)}{ v(G')}.\]
	\end{lem}
	\begin{proof}
		Write $e = e(G)$. The result is trivial if $e=0$, so assume $e>0$.  Assume for contradiction that no such subgraph $G'$ exists.  We claim that for all non-negative integers $r$ there exist $G_r\sub G$ with at most $2^{-br}n$ vertices and at least $2^{-r}e$ edges.  The result holds with $G_0=G$, so inductively assume the result has been proven through some $r$.   Let $G_{r+1}\sub G_r$ be the graph obtained after iteratively deleting vertices of degree less than $2^{(b-1)r-1}(e/n)$ from $G_r$.  Note that \[e(G_{r+1})\ge e(G_r)-2^{(b-1)r-1}(e/n)\cdot v(G_r)\ge 2^{-r}e-2^{(b-1)r-1}(e/n)\cdot 2^{-br}n\ge 2^{-r-1}e.\]  If $v(G_{r+1})\ge 2^{-b(r+1)}n>0$, then $2^{r}\ge \half (\f{n}{v(G_{r+1})})^{1/b}$.  This implies that  $G_{r+1}$ has minimum degree at least \[2^{(b-1)r-1}(e/n)\ge \left(\f{n}{v(G_{r+1})}\right)^{\f{b-1}{b}}2^{-b}(e/n)=2^{-b}\left(\frac{v(G_{r+1})}{n}\right)^{1/b}\f{e(G)}{ v(G_{r+1})},\]
		where this first inequality implicitly uses $b-1\ge 0$.  This contradicts our assumption that no such subgraph of $G$ exists, so we must have $v(G_{r+1})< 2^{-b(r+1)}n$.
		Thus $G_{r+1}$ satisfies the desired conditions, proving the claim.  Taking $r=\log_2(n)$, the claim implies there exists a subgraph on less than 1 vertex with at least $e/n>0$ edges, which is impossible. 
	\end{proof}

	The only reason we proved \Cref{lem:reduction} in the more general setting of multigraphs with loops is to prove the following technical result.
	
	\begin{lem}\label{lem:reductionSet}
		Let $B$ be a set, and let $f$ be any function from $B$ to $\Z_{\ge 0}$. For all $b>1$, there exists a subset $B'\sub B$ such that \[\min_{y'\in B'} f(y')\ge 2^{-b}\left(\frac{|B'|}{|B|}\right)^{1/b}\f{\sum_{y\in B} f(y)}{ |B'|}.\]
	\end{lem}
	\begin{proof}
		Define an auxiliary graph $G$ on $B$ where each vertex $y$ has $f(y)$ loops (and these are the only edges in $G$).  Applying Lemma~\ref{lem:reduction} gives the result.
	\end{proof}
	Roughly speaking, this lemma will be applied with $B$ the set of vertices that are at distance $b$ from some vertex $x$ and with $f(y)$ the number of paths of length $b$ from $x$ to $y$. This will allow us to choose vertices $x$ and $y$ connected by many paths of length $b$, which we can use to construct copies of $\theta_{a,b}$ where $x$ and $y$ are the two high-degree vertices.

	\section{Preliminaries}\label{sec:prelim}
	
	\subsection{Key Definitions}
	
	As noted in the proof outline  Subsection~\ref{subsec:outline}, we wish to consider hypergraphs on $V(\theta_{a,b})\times V(G)$.  To aid with this, we make use of the following definitions throughout the paper; see \Cref{fig:valid} for an example.
	
	\begin{defn}
		Given a graph $G$ and a set $\chi\sub V(\theta_{a,b})\times V(G)$, we define the \textit{projection sets} $\chi_\theta=\{w:\exists z,(w,z)\in \chi\}$ and $\chi_G=\{z:\exists w,(w,z)\in \chi\}$.  We say that a set $\chi\sub V(\theta_{a,b})\times V(G)$ is \textit{valid} if
		\begin{enumerate}
			\item $|\chi|=|\chi_\theta|=|\chi_G|$ (equivalently, each vertex of $\theta_{a,b}$ and $G$ appears at most once in a pair of $\chi$), and 
			\item  If $(w,z),(w',z')\in \chi$ with $ww'\in E(\theta_{a,b})$, then $zz'\in E(G)$.  
		\end{enumerate}
		We let $\mathbb{V}$ denote the set of valid subsets of $V(\theta_{a,b})\times V(G)$.
	\end{defn}
	
	\begin{figure}[h]
		\centering
		\mbox{}\\\vspace{1ex}
		\includegraphics[width=0.68\textwidth]{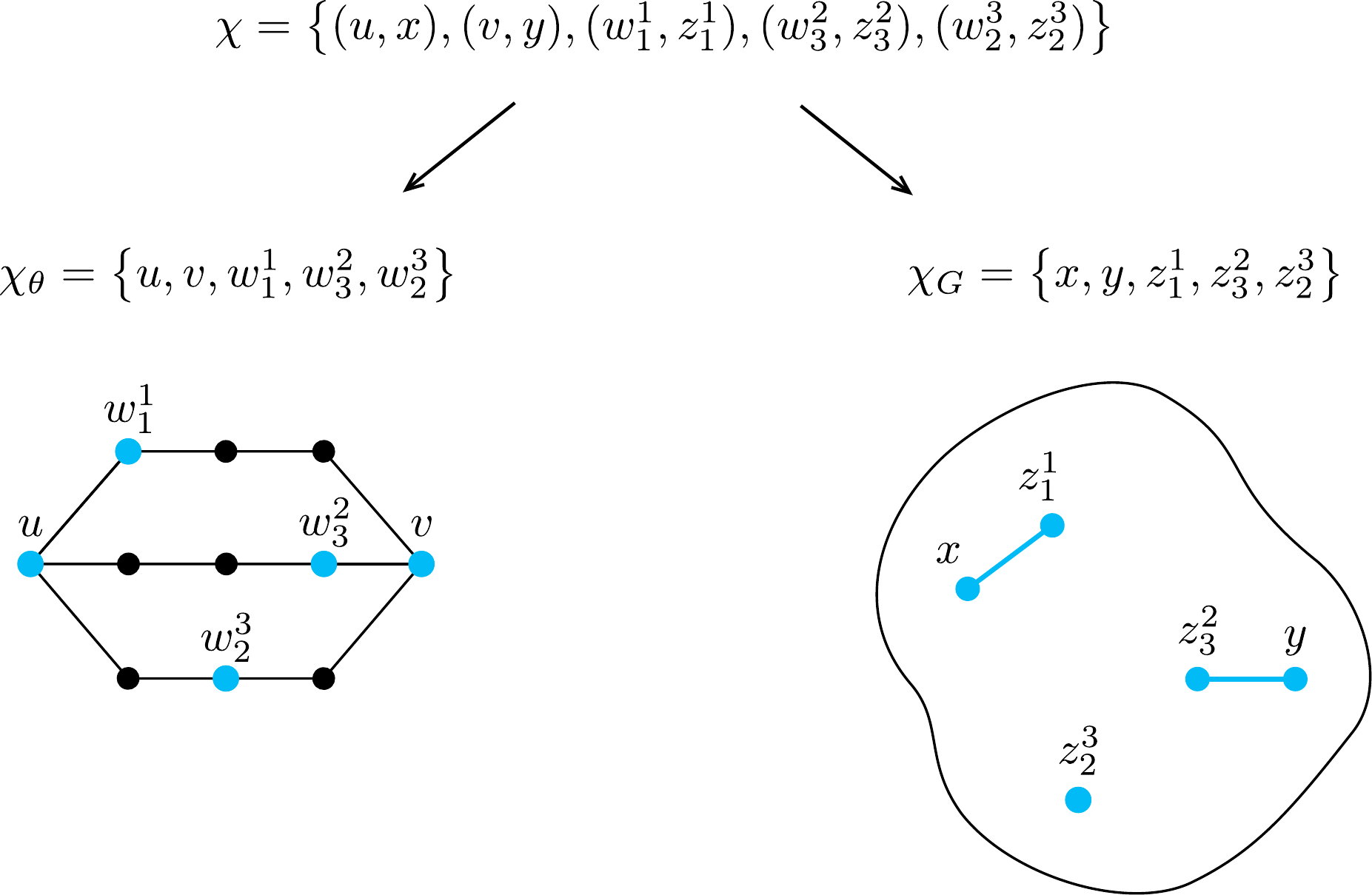}
		\caption{A set $\chi\sub V(\theta_{a,b})\times V(G)$ together with its projections $\chi_\theta\sub V(\theta_{a,b})$ and $\chi_G\sub V(G)$.  The set $\chi$ is valid if and only if all of the vertices depicted are distinct and $xz_1^1,\ z_3^2y\in E(G)$.}
		\label{fig:valid}
	\end{figure}
	
	Note that if $\chi$ is a valid set with $|\chi|=v(\theta_{a,b})$, then by definition this means the map $\phi_\chi:V(\theta_{a,b})\to V(G)$ which sends $w\in V(\theta_{a,b})$ to the unique vertex $z\in V(G)$ with $(w,z)\in \chi$ is an injective homomorphism.  In particular, the vertices $\chi_G$ induce a graph containing a copy of $\theta_{a,b}$ as a subgraph.  Since our ultimate goal is to find a large colleciton of such subgraphs which are spread out, we make the following definitions.
	\begin{defn}
		We say that a hypergraph $\c{H}$ is a \textit{$G$-hypergraph} if its vertex set is $V(\theta_{a,b})\times V(G)$ and all of its hyperedges are valid sets $h$ with $|h|=v(\theta_{a,b})$.  We say that functions of the form $D:2^{V(\theta_{a,b})}\to \N\cup \{\infty\}$ are \textit{codegree functions}, and for such a $D$ we say that $\c{H}$ is \textit{$D$-good} if $\deg_{\c{H}}(\chi)\le D(\chi_\theta)$ for all $\chi\in \mathbb{V}$.
	\end{defn}
	
	Note that being $D$-good means that no valid set $\chi$ is contained in too many hyperedges, with the exact degree condition depending only on $D$ and the projection $\chi_\theta\sub V(\theta_{a,b})$ (which allows us, for example, to impose stronger conditions if $\chi$ contains vertices corresponding to the high degree vertices of $V(\theta_{a,b})$).

	The main technical work of this paper is in constructing $G$-hypergraphs which have many hyperedges and which are $D$-good for some $D$ which is sufficiently small to apply \Cref{thm:containers}.  To do this, we will consider several $D$ functions simultaneously (which will ultimately be combined into a unified codegree bound in the next two sections). The first and simplest function we consider is the following.

	\begin{defn}\label{def:Dforest}
		Let $\del>0$, and let $k,n$ be positive integers. We define a codegree function $D_{\forest}$ as follows: if $\nu\sub V(\theta_{a,b})$ induces a forest on $e\ge 1$ edges, then
		\[D_{\forest}(\nu)=\ceil{\f{k^{ab}n^2}{\del kn^{1+1/b}\cdot (\del k^{b/(b-1)})^{e-1}}},\]
		and otherwise $D_{\forest}(\nu)=\infty$.
	\end{defn}
	Having $D_{\forest}(\nu)$ evaluate to infinity is not necessary, and we  do this only to emphasize that this function essentially ignores sets $\nu$ which do not induce a forest with at least one edge.
	
	Essentially, the main technical result of Morris and Saxton and Corsten and Tran says that one can construct large $G$-hypergraphs which are $D_{\forest}$-good. To go beyond this, we will show that one can construct collections which are $D$-good for functions $D$ which are finite on sets $\nu\sub V(\theta_{a,b})$ that contain cycles (and more precisely, sets that contain the two vertices of $\theta_{a,b}$ of large degree).
	
	The specific functions $D$ we need are somewhat complex.  In all of these functions, the denominator of $D(\nu)$ roughly corresponds to the number of choices our algorithm has to build $\nu$, with the terms to the left of the ``\,$\cdot$\," typically counting the number of choices for the two high degree vertices of $\theta_{a,b}$.  The parameter $s$ will be chosen roughly such that the graph $G'$ obtained from \Cref{lem:reduction} has $2^{-s}n$ vertices.
	
	With all of this established, we define the remainder of our codegree functions. 
	
	\begin{defn}\label{def:Dsb}
		Let $\del>0$, and let $k,n$ be positive integers. For $a\ge 3$, let $u,v$ denote the two vertices of $\theta_{a,b}$ of degree larger than 2.  For each integer $0\le s\le 3\log n$, we define a codegree function $D_{s,b}$  as follows:  if $u,v\in \nu$, then 
		\[D_{s,b}(\nu)=\ceil{\f{k^{ab}n^2}{2^{-s}n^2\cdot \delta^{|\nu|}(2^{2s/3} k^b)^{(|\nu|-2)/(b-1)}}},\]
		and otherwise $D_{s,b}(\nu)=\infty$.
	\end{defn}

	The definition above will be used when the $t$ value from Proposition~\ref{prop:expansion} equals $b$.  The $t<b$ case is somewhat more complicated.  Again this is because the denominator roughly represents the number of choices we have for our algorithm at any step, and just as in \cite{corsten2021balanced,morris2016number}, the $t<b$ case of the algorithm is somewhat more complicated.
	
	\begin{defn}\label{def:Dst}
		Let $\del>0$, and let $k,n$ be positive integers.  For $0\le s\le 3\log n$ and $2\le t<b$, we define a codegree function $D_{s,t}$ as follows:  write the paths of $\theta_{a,b}$ as $uw_1^j\cdots w_{b-1}^j v$ for $1\le j\le a$, and define
		\[F_t=\{w_i^j: t\le i<b,\ i-t\tr{ even}\},\hspace{2em} f=|\nu\cap F_t|.\]
		If $u,v\in \nu$, then
		\[D_{s,t}(\nu)=\ceil{\f{k^{ab}n^2 }{2^{-2s} k^{(2b-2t+1)/(b-1)}n^{(2t-1)/b}\cdot \delta^{|\nu|}(2^{2s/3}kn^{1/b})^f(2^{2s/3}k^{b/(b-1)})^{|\nu|-f-2}}},\]
		and otherwise $D_{s,t}(\nu)=\infty$.
	\end{defn}
	
	The intuition for this codegree function is as follows: in the $t<b$ case with $s=0$, our algorithm first selects $u,v$, which it will be able to do in about $k^{(2b-2t+1)/(b-1)}n^{(2t-1)/b}$ ways (which is essentially the product of the bounds from \Cref{prop:expansion}(c) and (h)).  When choosing $w_i^j$, the number of choices will turn out to be about $kn^{1/b}$ if $w_i^j\in F_t$ and about $k^{b/(b-1)}$ if $w_i^j\notin F_t$.  Thus the denominator represents the number of choices our algorithm has for building $\nu$.
	
	One last codegree function is needed for the $t<b$ case.
	
	\begin{defn}\label{def:Dt}
		Let $\del>0$, and let $k,n$ be positive integers.  For $2\le t<b$, we define a codegree function $D_{t}$ as follows: define $F_t$ as above and let $g=|\nu\cap F_t|$ if $b-t$ is even and $g=|\nu\cap F_t|-1$ otherwise.  If $u,v,w_{b-1}^j\in \nu$ for some $j$, then
		\[D_{t}(\nu)=\ceil{\f{k^{ab}n^2 }{kn^{1+1/b}\cdot (kn^{1/b})^g(k^{b/(b-1)})^{|\nu|-g-3}\delta^{|\nu|}}},\]
		and otherwise $D_{t}(\nu)=\infty$.  For notational convenience, we also define $D_b$ by $D_b(\nu)=\infty$ for all $\nu\sub V(\theta_{a,b})$.
	\end{defn}
	The motivation for this definition is that the number of choices for $u,v,w_{b-1}^j$ is at least the number of choices for just $v,w_{b-1}^j$, which is about $kn^{1+1/b}$, i.e.\ the number of total edges in the graph $G$.  As before, the number of choices for every other $w_{i}^j$ vertex will depend on whether it is in $F_t$ or not.  The definition of $g$ reflects the fact that if $b-t$ is odd, then $w_{b-1}^j\in F_t$, but the number of choices for the first vertex of this form is already accounted for by the $kn^{1+1/b}$ term, so we can not include an extra factor of $kn^{1/b}$ for this vertex.   Note that with this codegree function, we omit counting the number of choices for $u$ in the denominator.  As such, $D_t$ will typically be much weaker (i.e.\ larger) than $D_{s,t}$, though it will do better when $\nu$ has few vertices and $k$ is small.
	
	\subsection{Saturated Sets}
	When building our $G$-hypergraph $\c{H}$, we need to be careful to avoid constructing theta graphs which contain a subset that has very large codegree in $\c{H}$.  To aid with this, we introduce the following.
	\begin{defn}
		Let $\c{H}$ be a $G$-hypergraph and $D$ a codegree function $D:2^{V(\theta_{a,b})}\to \N\cup \infty$.  We define the set of \textit{saturated sets}
		\[\c{F}(\c{H},D)=\{\chi\in \mathbb{V}: \deg(\chi)\ge D(\chi_\theta)\}.\]
		
		Given a valid set $\chi$ and $\nu\sub V(\theta_{a,b})\sm \chi_\theta$, define the \textit{link set} $\c{J}_{\c{H},D}(\chi;\nu)$ to be the set of $\gam\in \mathbb{V}$ with $\gam_\theta=\nu$ such that $\chi\cup \gam\in\c{F}(\c{H},D)$. If $\nu=\{w\}$ we will sometimes denote this set simply by $\c{J}_{\c{H},D}(\chi;w)$
	\end{defn}
	
	The intuition for the link set comes from our goal of  algorithmically trying to iteratively add a new theta graph to some $\c{H}$ such that $\c{H}$ continues to be $D$-good even after adding the theta graph.  If during the algorithm we have already designated some $\chi\in \mathbb{V}$ to be used in our new theta graph, and if our algorithm is about to choose some $\gam$ to add to $\chi$ such that $\gam_\theta=\nu$, then the algorithm can not choose any $\gam\in \c{J}_{\c{H},D}(\chi';\nu)$ for any $\chi'\sub \chi$, as otherwise the degree of $\chi'\cup \gam$ would be strictly larger than what $D$ dictates.  As an aside, our definition of link sets differs slightly from Morris and Saxton, who essentially defined the links to be  $\bigcup_{\chi'\sub \chi} \c{J}_{\c{H},D}(\chi';\nu)$.
	
	Because the link sets represent the number of ``bad'' choices our algorithm has, we will want to show that these sets are relatively small.  This is accomplished by the following lemma.
	
	\begin{lem}\label{lem:link}
		Let $\c{H}$ be a $G$-hypergraph which is $D$-good for some $D:2^{V(\theta_{a,b})}\to \N\cup \infty$.  If $D(\chi_\theta\cup \nu)=\infty$ then $\c{J}_{\c{H},D}(\chi;\nu)=\emptyset$, and otherwise
		\[|\c{J}_{\c{H},D}(\chi;\nu)|\le 2^{v(\theta_{a,b})}\frac{D(\chi_\theta)}{D(\chi_\theta\cup \nu)}.\]
	\end{lem}
	\begin{proof}
		If $D(\chi_\theta\cup \nu)=\infty$, then every $\gam\in \mathbb{V}$ with $\gam_\theta=\nu$ trivially has $\deg(\chi\cup \gamma)<D(\chi_\theta\cup \gamma_\theta)=\infty$, so no such $\gam$ satisfies $\chi\cup \gam\in \c{F}(\c{H},D)$ and we conclude $\c{J}_{\c{H},D}(\chi;\nu)=\emptyset$.  From now on we assume $D(\chi_\theta\cup \nu)<\infty$.  Note that	\[  \sum_{\gam \in \c{J}_{\c{H},D}(\chi;\nu)} \deg(\chi\cup \gam)\le \sum_{\gam:\ |\gam|=|\nu|} \deg(\chi\cup \gam)\le 2^{v(\theta_{a,b})}\deg(\chi)\le 2^{v(\theta_{a,b})} D(\chi_\theta),\]
		where the second inequality used that each hyperedge $h$ containing $\chi$ is counted at most $2^{v(\theta_{a,b})}$ times by the sum over $\gam$, and the last inequality used that $\cH$ is $D$-good.  On the other hand,
		\[\sum_{\gam\in \c{J}_{\c{H},D}(\chi;\nu)} \deg(\chi\cup \gam)\ge \sum_{\gam\in \c{J}_{\c{H},D}(\chi;\nu)} D(\chi_\theta\cup \gam_\theta)=|\c{J}_{\c{H},D}(\chi;\nu)|D(\chi_\theta\cup \nu).\]
		Rearranging these two inequalities gives
		\[|\c{J}_{\c{H},D}(\chi;\nu)|\le 2^{v(\theta_{a,b})} \frac{D(\chi_\theta)}{D(\chi_\theta\cup \nu)},\]
		completing the proof.

	\end{proof}

	Because all of our codegree functions involve the ceiling of a real-valued function, the following result, which allows us to ignore the ceilings, will be slightly more convenient to use compared to \Cref{lem:link}.
	
	\begin{cor}\label{cor:link}
		Let $\c{H}$ be a $G$-hypergraph which is $D$-good for some $D:2^{V(\theta_{a,b})}\to \N\cup \infty$, and suppose that $D(\nu)=\ceil{D'(\nu)}$ for some $D':2^{V(\theta_{a,b})}\to \R_{>0}$.  If $D(\chi_\theta\cup \nu)=\infty$ then $\c{J}_{\c{H},D}(\chi;\nu)=\emptyset$.  If $D(\chi_\theta\cup \nu)\ne \infty$ and $\chi\notin \c{F}(\cH,D)$, then
		\[|\c{J}_{\c{H},D}(\chi;\nu)|\le 2^{v(\theta_{a,b})+1}\frac{D'(\chi_\theta)}{D'(\chi_\theta\cup \nu)}.\]
		Moreover, this bound continues to hold for $D=D_{\forest}$ even if $\chi\in \c{F}(\cH,D)$ provided $\del$ is sufficiently small.
	\end{cor}
	\begin{proof}
		Note that trivially $D(\chi_\theta\cup \nu)\ge D'(\chi_\theta\cup \nu)$ and that $D(\chi_\theta)\le 2 D'(\chi_\theta)$ provided $D(\chi_\theta)\ge 2$.  Thus in this case, the result follows immediately from \Cref{lem:link}, and in particular, this situation always holds for $D_{\forest}$ provided $\del$ is sufficiently small. 
		
		It remains to consider the case that $D(\chi_\theta)=1$ and $\chi\notin \c{F}(\cH,D)$.  These two conditions imply $\deg(\chi)<D(\chi_\theta)=1$, so there exists no hyperedge of $\cH$ containing $\chi$.  Thus there exists no $\gam$ such that $\chi\cup \gam\in \c{F}(\c{H},D)$, i.e. such that $\deg(\chi\cup \gam)\ge D(\chi_\theta\cup \gam_\theta)\ge 1$.  We conclude that the link set is empty in this case, and hence the result trivially holds.  
	\end{proof}
	When applying this claim it will always be immediate\footnote{In terms of the notation for the next section, we will only apply \Cref{cor:link} with $D\ne D_{\forest}$ when $\chi$ is a subset of an $(s,t)$-compatible set, which by definition will not be in $\c{F}(\cH,D)$.} that $\chi\notin \c{F}(\cH,D)$, and for simplicity we will omit saying this explicitly.

	\section{Balanced Supersaturation for Vertices}\label{sec:main_proof}
	
	In this section, we prove our main technical theorem: a balanced supersaturation result for vertices.
	
	\begin{thm}\label{thm:balancedVertex}
		For all $a\ge 6$ and $b\ge 3$, there exist constants $\del>0, k_0 > 0$ such that the following holds for all $n\in\mathbb{N}$ and $k\geq k_0$.  If $G$ is an $n$-vertex graph with $kn^{1+1/b}$ edges, then there exists an integer $2\le t\le b$ and a $G$-hypergraph $\c{H}'_t$ with $|\c{H}'_t|\ge b^{-1} \del k^{ab}n^2$ which is $D'_t$-good, where $D'_t$ is defined by
		\[D'_t(\nu):=\begin{cases}
			D_{\forest}(\nu)& \tr{if }\nu \tr{ induces a forest},\\ 
			\min\{D_t(\nu),20(D_{0,t}(\nu)+\ceil{\log n})\} & \tr{otherwise}.
		\end{cases}\]
	\end{thm}
	We note that there is no need to consider $b=2$ since the case $\theta_{a,2}=K_{2,a}$ is already dealt with by Morris and Saxton~\cite{morris2016number}.

	\Cref{thm:balancedVertex} will follow quickly from the following technical result, which roughly says that given a collection of much fewer than $k^{ab} n^2$ copies of $\theta_{a,b}$ satisfying certain codegree conditions, we can find an additional copy of $\theta_{a,b}$ to add to the collection while maintaining the desired codegrees.
	
	\begin{prop}\label{prop:main}
		For all $a\ge 6$ and $b\ge 3$, there exist constants $\del>0, k_0 > 0$ such that the following holds for all $n\in\mathbb{N}$ and $k\geq k_0$. Let $G$ be an $n$-vertex graph on $kn^{1+1/b}$ edges with $k\ge k_0$, and let $\{\c{H}_{s,t}\}_{s,t}$ be a set of $G$-hypergraphs such that $\c{H}_{s,t}$ is $D_{s,t}$-good for each $0\le s\le 3\log n$ and $2\le t\le b$, and $\bigcup_s \c{H}_{s,t}$ is $D_t$-good for each $t<b$, and $\bigcup_{s,t}\c{H}_{s,t}$ is $D_{\forest}$-good.

		If $|\bigcup_{s,t} \c{H}_{s,t}|\le \del k^{ab}n^2$, then there exists some valid set $h\in \mathbb{V}$ of size $v(\theta_{a,b})$  and some $s',t'$ such that $h\notin  \bigcup_{s,t} \c{H}_{s,t}$, and such that if we define $\c{H}'_{s',t'}=\c{H}_{s',t'}\cup \{h\}$ and $\c{H}'_{s,t}=\c{H}_{s,t}$ for all $(s,t)\ne (s',t')$, then $\c{H}_{s,t}'$ is $D_{s,t}$-good for each $0\le s\le 3\log n$ and $2\le t\le b$, and $\bigcup_s \c{H}_{s,t}'$ is $D_t$-good for each $t<b$, and $\bigcup_{s,t}\c{H}_{s,t}'$ is $D_{\forest}$-good.
	\end{prop}

	Before proving \Cref{prop:main}, we show how it may be repeatedly applied to obtain our main supersaturation theorem.
	
	\begin{proof}[Proof of \Cref{thm:balancedVertex}]
		Initially start with collections $\{\c{H}_{s,t}\}_{s,t}$ where $\c{H}_{s,t}=\emptyset$ for all $s,t$.  By repeatedly applying \Cref{prop:main}, we obtain collections satisfying all of the codegree conditions and with $|\bigcup_{s,t} \c{H}_{s,t}|\ge \del k^{ab}n^2$.  In particular, there exists some $2\le t\le b$ such that $\c{H}_t':=\bigcup_{s} \c{H}_{s,t}$ contains at least $b^{-1}\del k^{ab}n^2$ hyperedges.
		
		By Proposition~\ref{prop:main}, we have for all $\chi\in \mathbb{V}$ that
		\[\deg_{\c{H}'_t}(\chi)\le \deg_{\bigcup_{s,t}\c{H}_{s,t}}(\chi)\le D_{\forest}(\chi_\theta),\]
		and similarly
		\[\deg_{\c{H}'_t}(\chi)\le D_t(\chi_\theta).\]
		To complete the proof, we only have to show $\deg_{\c{H}'_t}(\chi)\le 20(D_{0,t}(\nu)+\floor{\log n})$ for all $\chi$ such that $\chi_\theta$ contains a cycle.  We first consider the case $t=b$.  Here Proposition~\ref{prop:main} gives
		\[\deg_{\c{H}'_b}(\chi)\le \sum_{s=0}^{3\log n} \deg_{\c{H}_{s,b}}(\chi)\le \sum_{s=0}^{3\log n} D_{s,b}(\nu)\le 3\ceil{\log n}+ 4 D_{0,b}(\nu)\sum_{s=0}^{3\log n} 2^{\left(1-\f{2(|\chi_\theta|-2)}{3(b-1)}\right)s},\]
		where this last step used that either $D_{s,b}(\nu)=1$, or (by \Cref{def:Dsb}) $D_{s,b}(\nu)$ differs from $D_{0,b}(\nu)$ by at most a multiplicative factor of $4\cdot 2^{\left(1-\f{2(|\chi_\theta|-2)}{3(b-1)}\right)s}$ (where the factor of 4 comes from the two ceiling functions involving $D_{s,b}$ and $D_{0,b}$).  Since $\chi_\theta$ contains a cycle, we have $|\chi_\theta|\ge 2b$, so the sum above is at most $\sum_{s=0}^\infty 2^{-s/3}\le 5$.   We conclude that $\deg_{\c{H}_b'}(\chi)\le D'_b(\chi_\theta)$ for all valid $\chi$.  
		
		When $t<b$, essentially the same reasoning gives that if $\chi_\theta$ contains a cycle then
		\[\deg_{\c{H}'_t}(\chi)\le 3\ceil{\log n}+4 D_{0,t}(\nu)\sum_{s=0}^\infty 2^{\left(2-(2b-2)\f{2}{3}\right)s},\]
		and since $b\ge 3$ this latter sum is at most $\sum_{s=0}^\infty 2^{-2s/3}\le 5$.  This gives the result.
	\end{proof}

	The rest of this section is dedicated to proving \Cref{prop:main}. Let $G$ be the graph in the hypothesis of Proposition~\ref{prop:main} and $\{\c{H}_{s,t}\}_{s,t}$ the corresponding collections.
	
	The basic idea of the proof is to algorithmically construct many copies of $\theta_{a,b}$, and to show that at least one of them is not already contained in $\c{H}_{s',t'}$ for some appropriate $s',t'$, and such that our codegree conditions continue to be satisfied.

	We will begin by pruning the graph $G$ so that all of its remaining edges and vertices are ``well-behaved" (\Cref{sec:setup}). We then give the algorithmic construction of copies of $\theta_{a,b}$ and show that a new copy may be added to some $\c{H}_{s,t}$ (\Cref{sec:alg}), completing the proof.  Throughout the argument we fix some $\del$ depending only on $a,b$ which is sufficiently small for our arguments to go through.
	
	\subsection{Pruning}\label{sec:setup}

	Let $G_0\sub G$ be the graph obtained by deleting edges of $G$ that are already ``saturated" by hyperedges of $\bigcup_{s,t}\c{H}_{s,t}$, that is, those edges $e\in E(G)$ for which there exists $\chi\in \mathbb{V}$ with $\chi_G=e$ and  $$\deg_{\bigcup_{s,t} \c{H}_{s,t}}(\chi)\, \ge\, D_{{\forest}}\big(\chi_\theta\big) = \left\lceil \frac{k^{ab}n^2}{\delta kn^{1+1/b}}\right\rceil.$$ 
	This will ensure that any new theta graph constructed using only edges of $G_0$ will not violate our edge-codegree bounds when added to $\bigcup_{s,t}\c{H}_{s,t}$.  We bound the number of these saturated edges by double-counting elements of $\bigcup_{s,t}\c{H}_{s,t}\,$: \begin{align*}
		\left(\text{\# edges $e$ with $\chi\in \mathbb{V},\ \chi_G=e,\ \deg_{\bigcup_{s,t} \c{H}_{s,t}}(\chi) =\left\lceil \tfrac{k^{ab}n^2}{\delta kn^{1+1/b}}\right\rceil$}\right)\cdot \left\lceil \frac{k^{ab}n^2}{\delta kn^{1+1/b}}\right\rceil
		&\le e(\theta_{a,b})\cdot|\textstyle\bigcup_{s,t} \c{H}_{s,t}|.
	\end{align*}
	Rearranging slightly, the number of such edges is at most
	\[
	e(\theta_{a,b})\,|\textstyle\bigcup_{s,t} \c{H}_{s,t}|\bigg/\left\lceil \dfrac{k^{ab}n^2}{\delta kn^{1+1/b}}\right\rceil
	\le\, {\delta^2 e(\theta_{a,b})}\cdot  kn^{1+1/b},
	\]
	and if $\del$ is sufficiently small this is at most $\half k n^{1+1/b}=\half e(G)$, which implies $e(G_0)\ge \half k n^{1+1/b}$.  Notice that no remaining edges are ``saturated,'' i.e.\ we have \begin{equation}\deg_{\bigcup_{s,t}\c{H}_{s,t}}(\chi)\le \left\lceil \frac{k^{ab}n^2}{\delta kn^{1+1/b}}\right\rceil - 1 \tr{ for all }\chi\in \mathbb{V}\tr{ such that }\chi_G\in E(G_0).\label{eq:G0}\end{equation}
	
	Now we further prune the graph by eliminating low-degree vertices: let $G'\sub G_0$ be the subgraph of high minimum degree guaranteed by Lemma~\ref{lem:reduction}. Although $G'$ may have substantially fewer vertices than $G$ (meaning our algorithm will have fewer choices at various steps), in this case it will compensate by having a substantially larger minimum degree.

	More concretely, let $m = v(G')$, and let $\ell$ be the real number such that $G'$ has minimum degree $\ell m^{1/b}$.
	By Lemma~\ref{lem:reduction} we have
	\[\ell m^{1/b}\ge 2^{-b-1}k n^{1+1/b}m^{-1}(n/m)^{-1/b}\implies \ell\ge 2^{-b-1}(n/m)k.\]
	We let $r$ be the unique integer such that 
	\begin{equation}
		2^{-r}n\le m< 2^{-r+1}n, \label{eq:m}
	\end{equation}
	and we note that the previous inequality implies
	\begin{equation}\label{eq:ell}
		\ell \ge 4^{-b} 2^rk.
	\end{equation}
	In total this implies the minimum degree $\ell m^{1/b}$ of $G'$ is at least $\Om(kn^{1/b})$, which is the average degree of $G$, and that the minimum degree of $G'$ is much larger compared to that of $G$ if $m=v(G')$ is much smaller than $n=v(G)$.  Before moving on, we note
	\begin{equation}\ell^{b/(b-1)}\le \ell m^{1/b}.\label{eq:minDegree}\end{equation}
	Indeed, since $\ell m^{1/b}$ is the minimum degree of the $m$-vertex graph $G'$, we must have $\ell m^{1/b}\le m$, and rearranging shows this is equivalent to \eqref{eq:minDegree}.

	\subsection{The Algorithm}\label{sec:alg}
	
	We are now ready to begin finding copies on $\theta_{a,b}$ in $G'$. Our strategy is roughly as follows: first, we identify which collection $\c{H}_{s,t}$ we wish to add a copy to (based on the expansion properties of $G'$ detailed in \Cref{prop:expansion}).   After this, we carefully choose vertices $x$ and $y$ to serve as the two high-degree vertices for the copies we will add. We then show that $x$ and $y$ are not already contained in too many copies in $\c{H}_{s,t}$, and algorithmically construct a large number of theta graphs in $G'$ that \emph{do} contain $x$ and $y$; this allows us to conclude that we have found at least one new copy not already contained in $\c{H}_{s,t}$. Crucially, along the way, we ensure that at no step of the algorithm are our codegree conditions violated, ensuring that the copy added to $\c{H}_{s,t}$ is ``good." 
	
	Before delving into the meat of the proof, we introduce some notation which is more compact.  Define \[\c{H}=\bigcup_{s,t} \c{H}_{s,t},\hspace{2em} \c{H}_t=\bigcup_s \c{H}_{s,t}.\]  Also define \[\c{F}_{\forest}=\c{F}(\c{H},D_{\forest}),\hspace{2em} \c{F}_t=\c{F}(\c{H}_t,D_t),\hspace{2em} \c{F}_{s,t}=\c{F}(\c{H}_{s,t},D_{s,t}).\]  When applying \Cref{cor:link}, we adopt the shorthand
	\[\c{J}_{\forest}=\c{J}_{\c{H},D_{\forest}},\hspace{2em} \c{J}_t=\c{J}_{\c{H}_t,D_t},\hspace{2em} \c{J}_{s,t}=\c{J}_{\c{H}_{s,t},D_{s,t}}.\]
	
	We say that a set $\chi$ is \textit{$(s,t)$-compatible} if $\chi\in \mathbb{V}$ and if no subset of $\chi$ lies in $\c{F}_{\forest}\cup \c{F}_t\cup \c{F}_{s,t}$.  Crucially, we observe that proving the proposition is equivalent to showing that for some $s,t$, there exists an $(s,t)$-compatible set $h$ with $|h|=v(\theta_{a,b})$ such that $h\notin \c{H}$ (since, for example, no subset of $h$ being in $\c{F}_{\forest}$ implies $\c{H}\cup \{h\}$ is $D_{\forest}$-good).

	Before moving on, we make a small but important observation.
	\begin{claim}\label{cl:oneEdge}
		If $\chi$ is a valid set and $\nu\sub V(\theta_{a,b})\sm \chi_\theta$ is such that $\chi_\theta\cup \nu$ induces at most one edge in $\theta_{a,b}$, then $\c{J}_{\forest}(\chi;\nu)=\emptyset$.
	\end{claim}
	\begin{proof}
		If $\chi_\theta\cup \nu$ induces 0 edges then $D_{\forest}(\chi_\theta\cup \nu)=\infty$ and the result follows from \Cref{cor:link}.  Thus we can assume $\chi_\theta\cup \nu$ induces exactly one edge $ww'$.  If there exists some $\gam\in \c{J}_{\forest}(\chi;\nu)$, then there exist pairs $(w,z),(w',z')\in \chi\cup \gam$ (since $(\chi\cup \gam)_\theta=\chi_\theta\cup\nu$).  In this case we have
		\[\deg_{\cH}(\chi\cup \gam)\le \deg_{\cH}(\{(w,z),(w',z')\})<D_{\forest}(\{w,w'\})=D_{\forest}(\chi_\theta\cup \gam_\theta),\]
		where the first inequality used that we are looking at the codegree of a smaller set, the second inequality follows from \eqref{eq:G0} which says $G'\sub G_0$ does not contain any edges $zz'$ with $\deg_{\cH}(\{(w,z),(w',z')\})\ge D_{\forest}(\{w,w'\})$, and the equality used that $\chi_\theta\cup \gam_\theta=\chi_\theta\cup \nu$ induces exactly one edge. This inequality implies $\chi\cup \gam\notin \c{F}(\cH,D_{\forest})$, contradicting the assumption $\gam\in \c{J}_{\forest}(\chi,\nu)$.  We conclude that this link set is indeed empty.
	\end{proof}

	\subsubsection{The Setup}
	We wish to apply Proposition~\ref{prop:expansion} to the ``pruned" graph $G'$.  For this we need to specify a set of forests to avoid.  Intuitively we wish to use the set $\c{F}_{\forest}$, but this is a collection of subsets of $V(\theta_{a,b})\times V(G)$, not of subgraphs of $G'$.  To get around this minor technically,  for $\chi\in \mathbb{V}$ we define the ``projection graph" $H_\chi$ by
	\[V(H_\chi)=\chi_G,\hspace{2em} E(H_\chi)=\{zz':\exists ww'\in E(\theta_{a,b}),\ (w,z),(w',z')\in \chi\}.\]
	Note that by definition of $\chi$ being valid, $H_\chi$ is a subgraph of $G$ which is isomorphic to the subgraph of $\theta_{a,b}$ induced by $\chi_\theta$.  We let $\c{F}'_{\forest}=\{H_\chi:\chi\in \c{F}_{\forest}\}$.   Since $D_{\forest}(\nu)=\infty$ for $\nu$ which do not induce forests, every element of $\c{F}_{\forest}'$ is a forest.  To apply Proposition~\ref{prop:expansion} with this set, it remains to verify the following, which gives the hypotheses of Proposition~\ref{prop:expansion}(g).
	
	\begin{claim}\label{cl:forest}
		Let $\ep>0$ be as in Proposition~\ref{prop:expansion}.  If $\del>0$ is sufficiently small, then for every path $x_1\cdots x_p$ of $G'$ with $p\le b$ which does not contain an element of $\c{F}_{\forest}'$ as a subgraph, the number of vertices $x_{p+1}\in N_{G'}(x_p)$ such that some subgraph of the path $x_1\cdots x_{p+1}$ is in $\c{F}_{\forest}'$ is at most $\ep \ell m^{1/b}$. 
	\end{claim}
	\begin{proof}
		Fix any path $x_1\cdots x_p$ as above; we wish to bound then. We introduce the following notation which will only be used in the proof of this claim: we say that a pair $(\chi,w)$ with $\chi\in \mathbb{V}$ and $w\in V(\theta_{a,b})$ is \textit{good} if $\chi_G\sub \{x_1,\ldots,x_p\}$ and $w$ is adjacent to at most one vertex of $\chi_\theta$.  We claim that if $x_{p+1}\in N_{G'}(x_p)$ is such that some subgraph of the path $x_1\cdots x_{p+1}$ is in $\c{F}'_{\forest}$, then $\{(w,x_{p+1})\}\in \c{J}_{\forest}(\chi;w)$ for some good pair $(\chi,w)$.
		
		Indeed, say the subgraph of $x_1\cdots x_{p+1}$ in $\c{F}'_{\forest}$ was $H_\gam$ for some $\gam\in \mathbb{V}$.  Because $H_\gam$ is not a subgraph of $x_1\cdots x_p$, we must have $x_{p+1}\in V(H_\gamma)$, and thus we have $(w,x_{p+1})\in \gam$ for some $w\in V(\theta_{a,b})$. If $\chi:=\gam\sm \{(w,x_{p+1})\}$, then $H_\gam$ being a subgraph of $x_1\cdots x_{p+1}$ implies $\chi_G\sub \{x_1,\ldots,x_p\}$ and that $w$ is adjacent to at most one vertex of $\chi_\theta$ (as otherwise $x_{p+1}$ would have degree greater than 1 in $H_{\gam}$, contradicting this being a subgraph of $x_1\cdots x_{p+1}$).  In total we find $\gam=\chi \cup \{(w,x_{p+1})\}$ for some good pair $(\chi,w)$. Moreover, by definition of $H_\gam\in \c{F}'_{\forest}$, we find \[\chi\cup\{(w,x_{p+1})\}=\gam\in \c{F}_{\forest}=\c{F}(\c{H},D_{\forest}),\]
		so by definition $\{(w,x_{p+1})\}\in \c{J}_{\forest}(\chi;w)$ as desired.
		
		With this we see that the number of choices for $x_{p+1}$ is at most the number of elements of $\c{J}_{\forest}(\chi;w)$ for all possible good pairs $(\chi,w)$.  To count the number of such elements, fix some good pair $(\chi,w)$.   If $\chi_\theta$ induces no edges, then since $w$ is adjacent to at most one vertex of $\chi_\theta$ by definition of $(\chi,w)$ being a good pair, $\chi_\theta\cup \{w\}$ induces at most one edge.   \Cref{cl:oneEdge} then implies $\c{J}_{\forest}(\chi;w)=\emptyset$.

		Now assume $\chi_\theta$ induces at least one edge.  By definition of $D_{\forest}$ and \Cref{cor:link}, we have for any good pair $(\chi,w)$ that
		
		\[|\c{J}_{\forest}(\chi;w)|\le 2^{v(\theta_{a,b})+1}\del k^{b/(b-1)}\le 2^{v(\theta_{a,b})+1} \del 16^b \ell m^{1/b},\]
		where the first inequality used that $\chi_\theta\cup \{w\}$ induces at most one more edge than $\chi_\theta$, and the last inequality used \eqref{eq:ell} and \eqref{eq:minDegree}.   As the total number of good pairs is at most $v(\theta_{a,b})2^p=O_{a,b}(1)$, the result follows by taking $\del$ sufficiently small.
	\end{proof}
	
	With this claim and the fact that $\ell$ is at least a sufficiently large constant due to \eqref{eq:ell}, we can apply Proposition~\ref{prop:expansion} to $G'$ and $\c{F}_{\forest}'$, and we let $t,X$ be the integer and set guaranteed by this proposition.  We recall that $u,v\in V(\theta_{a,b})$ are  the high degree vertices of $\theta_{a,b}$.  
	
	Let $x\in X$ be a vertex such that $(u,x)$ is in the fewest hyperedges of $\c{H}$, that is, a vertex with $\deg_{\c{H}}(\{(u,x)\})=\min_{y\in X}\deg_{\c{H}}(\{(u,y)\})$ (this will help us ensure we can find a \emph{new} copy of $\theta_{a,b}$ containing $x$).  Let $(\c{B},\c{Q})$ with $\c{B}=(B_1,\ldots,B_t)$ be the pair for $x$ as guaranteed by Proposition~\ref{prop:expansion}(a).
	
	We now split our analysis into several cases based on the value of $t$. The overarching strategy is the same in both cases, but the details are somewhat simpler in the case $t=b$ (where $G'$ has nice ``random-like" expansion near $x$).

	\subsubsection{Case 1: $t=b$}
	Our strategy is to build copies of $\theta_{a,b}$ in $G'$ which use $x$ as one of the high degree vertices $u$. To choose the vertex $y$ which plays the role of the other high-degree vertex $v$, we would like to ensure there are many paths in $\c{Q}$ connecting $y$ to $x$ (which we will then use to build our theta graphs). And indeed, this holds for many vertices in $B_b$: for $y\in B_b$, let $f(y)$ denote the number of paths of $\c{Q}$ which $y$ is an endpoint of. By \Cref{lem:reductionSet}, there is some $B'\subseteq B_b$ so that
	\begin{equation}\label{eq:fbound_initial}
		\min_{y\in B'} f(y)\ge 2^{-b}\left(\frac{|B'|}{|B_b|}\right)^{1/b}\f{\sum_{y\in B_b} f(y)}{ |B'|}.
	\end{equation}
	Notice that we face a trade-off: we may have a large set of vertices $B'$, each of which is the endpoint of approximately the average number of paths, or a smaller set $B'$ where $f(y)$ is much larger than average. We can obtain a strong balanced supersaturation result either way, but to do so, we must keep track  of this trade-off.  To this end, let $r'$ be the unique integer such that \begin{equation}2^{-r'}m\le |B'|<2^{-r'+1}m.\label{eq:r'}\end{equation}  Since $|B_b|\le m$ and $\sum_{y\in B_b}f(y)=|\c{Q}|\ge \ep \ell^b m$ by \Cref{prop:expansion}(c),  \eqref{eq:fbound_initial} can be relaxed to
	\begin{equation}\label{eq:fbound}		
		\min_{y\in B'} f(y)\ge 2^{-b}\f{\ep \ell^b m}{|B_b|^{1/b}|B'|^{1-1/b}}\ge \ep 4^{-b}\cdot 2^{(1-1/b)r'} \ell^b.
	\end{equation}

	Recall from \eqref{eq:m} that $2^{-r}n\le m< 2^{-r+1}n$, and let \[s=2r+r'.\] Roughly speaking, 
	if $s$ is small, then $m$ and/or $|B'|$ are large, which is what we would expect to happen if $G$ was a random graph (as opposed to $G$ being, e.g.,\ a clique with isolated vertices, wherein both of these quantities would be small). We now aim to show that we can add a new theta graph to the collection $\c{H}_{s,b}$. 	Let $y\in B'$ be such that $(v,y)$ is in the fewest number of hyperedges with $(u,x)$ in $\c{H}$.  As before, this will help ensure we find a \textit{new} hyperedge, since $x$ and $y$ are not already contained in too many elements of $\c{H}$.  
	\begin{claim}\label{cl:case1bd}
		The set $\chi:=\{(u,x),(v,y)\}$ satisfies
		\[\deg_{\c{H}}(\chi)\le \del \ep^{-1}2^s k^{ab},\]
		and is $(s,b)$-compatible provided $\del$ is sufficiently small.
	\end{claim}
	\begin{proof}
		First, recall that $x$ is such that $(u,x)$ is contained in the fewest number of hyperedges in $\c{H}$ among all vertices in the set $X$, which means
		\[
		\deg_{\c{H}}(\{(u,x)\})\le \f{|\c{H}|}{|X|}\le \f{\del k^{ab} n^2}{\ep m},
		\]
		where the second inequality used $|X|\ge \ep m$ by \Cref{prop:expansion}(b) when $t=b$ and the hypothesis  $|\c{H}|\le \del k^{ab}n^2$ of \Cref{prop:main}.
		Similarly, as $y\in B'$ is such that $(v,y)$ is in the fewest number of hyperedges with $(u,x)$ in $\c{H}$, we have
		\[
		\deg_{\c{H}}(\chi)\le \f{\del k^{ab}n^2}{\varepsilon m|B'|}\le  \f{\del k^{ab}n^2}{\ep 2^{-2r-r'}n^2}=\del \ep^{-1}2^s k^{ab},
		\]
		where the second inequality used \eqref{eq:m} and \eqref{eq:r'}.
		
		It remains to show $\chi$ is $(s,b)$-compatible.  First we show $\chi$ is a valid set.  Since $u,v$ are distinct non-adjacent vertices of $\theta_{a,b}$, we only need to check $x\ne y$.  And indeed, we can not have $x\in B_b$, since by Proposition~\ref{prop:expansion}(e), every element of $B_b$ is the endpoint of a positive number of paths of length $b$ from $x$ (and since these are paths, $x$ can not serve as both endpoints).  Since $y\in B'\sub B_b$, we conclude $x\ne y$ and that $\chi$ is valid.
		
		It remains to check that every subset of $\chi$ satisfies the desired codegree conditions.  Note that for any $\chi'\sub \chi$ of size 1, we have $D_{\forest}(\chi'_\theta)=D_{b}(\chi'_\theta)=D_{s,b}(\chi'_\theta)=\infty$, and as such $\chi'$ will not belong to $\c{F}_{\forest}\cup \c{F}_b\cup \c{F}_{s,b}$.  Similarly $D_{\forest}(\chi_\theta)=D_{b}(\chi_\theta)=\infty$, so it only remains to verify that
		\[\deg_{\c{H}_{s,b}}(\chi)<D_{s,b}(\chi_\theta)=\ceil{\f{k^{ab}n^2}{2^{-s}n^2 \del^2}}=\ceil{2^s \del^{-2} k^{ab}}.\]
		Since $\deg_{\c{H}_{s,b}}(\chi)\le \deg_{\c{H}}(\chi)$, this bound follows from the first part of the claim, completing our proof.
	\end{proof}
	
	We now wish to construct many ``good" copies of $\theta_{a,b}$ in $G'$ with $x$ and $y$ as the two high-degree vertices 
	(i.e.\ more than the bound in \Cref{cl:case1bd}). 
	To do this, we iteratively pick paths $P_1,\ldots,P_a$ in $\c{Q}$ that end in $y$, and we take our copy of $\theta_{a,b}$ to be the union of these paths. 
	We must ensure that the paths chosen are such that $P_1\cup \cdots \cup P_a$ is $(s,b)$-compatible, and in particular that they do not intersect each other and that no subset of their vertices is already saturated.  For this claim, we recall that the paths of $\theta_{a,b}$ are denoted by $uw_1^j\cdots w_{b-1}^jv$.
	\begin{claim}\label{cl:pathsCase1}
		Let $1\leq j\leq a$, and let $P_1,\dots, P_{j-1}$ be a collection of paths in $\c{Q}$ ending in $y$, and for each path $P_{j'}$, write $P_{j'} = x z_1^{j'} \cdots z_{b-1}^{j'}y$. Suppose that the set 
		$$\chi := \{(u,x),(v,y)\}\cup \bigcup_{j'<j} \left\{(w_1^{j'},z_1^{j'}),\dots, (w_{b-1}^{j'},z_{b-1}^{j'})\right\}$$
		is $(s,b)$-compatible. Then there are  at least $\ep 4^{-2b^2} \cdot 2^{2s/3}k^b$ choices of a path $P_j = xz_1^j \cdots z_{b-1}^jy$ in $\c{Q}$ so that
		$$\chi_j:=\chi \cup \left\{ (w_1^j,z_1^j),\dots, (w_{b-1}^j,z_{b-1}^j)\right\}$$
		
		is $(s,b)$-compatible. 
	\end{claim}
	\begin{figure}[H]
		\centering
		\includegraphics[width=0.52\textwidth]{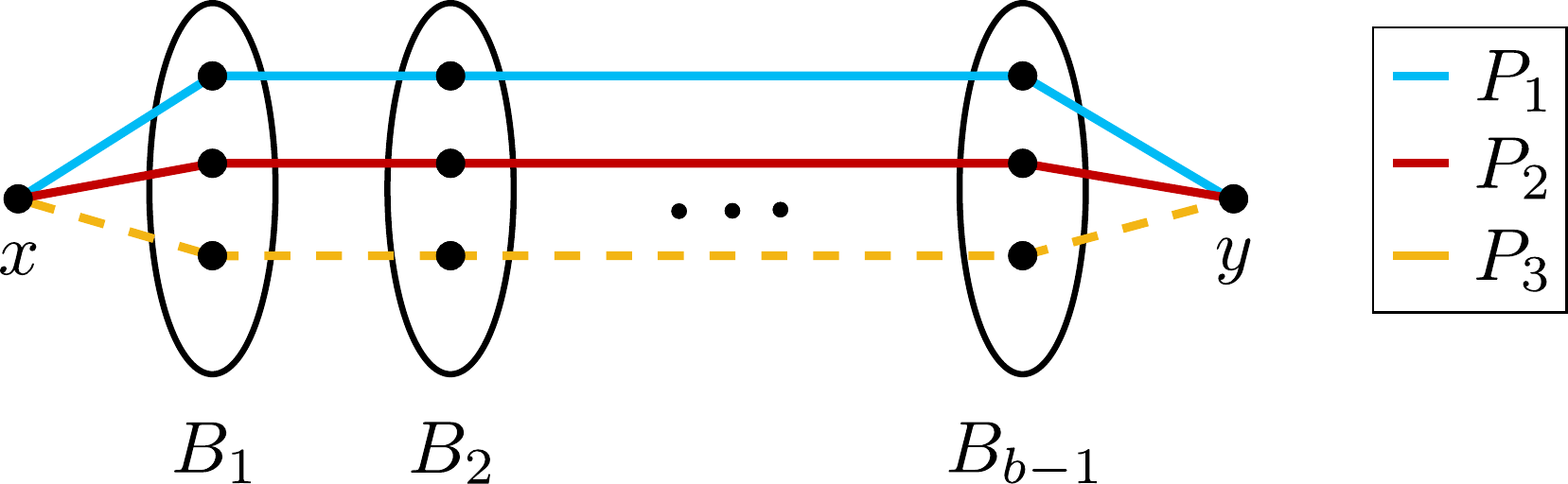}
		\caption{Given paths $P_1$ and $P_2$, the algorithm next picks a path $P_3$ from $y$ to $x$ in $\c{Q}$ while maintaining that the relevant sets are compatible.}
		\label{fig:EqualsB}
	\end{figure}

	\begin{proof}
		Let $\c{Q}(y)$ denote the set of paths in $\c{Q}$ ending in $y$. Since $y\in B'$, we have \begin{equation}|\c{Q}(y)|\,\stackrel{\eqref{eq:fbound}}{\ge}\,\ep 4^{-b}\cdot 2^{(1-1/b)r'} \ell^b\,\stackrel{\eqref{eq:ell}}{\ge}\, \ep 4^{-b}\cdot 4^{-b^2}\cdot 2^{(1-1/b)r'+br}\, k^b\ge 2\ep 4^{-2b^2}\cdot 2^{2s/3} k^b,\label{eq:Qy}\end{equation}
		where this last step used $b\ge 3$ and $s=2r+r'$.

		Our goal now is to show that among these paths, there are few ``bad choices" that must be avoided, i.e.\ few choices so that $\chi_j$ is not $(s,b)$-compatible.  To show this, \Cref{prop:expansion}(f) will be crucial, which we recall says that for any non-empty set $S$ of vertices in $G'$ not containing $x,y$, there are at most $\varepsilon^{-1}\ell^{(b-1-|S|)b/(b-1)}$ paths in $\c{Q}(y)$ which contain $S$. 
		
		We first show that almost all choices of $P_j$ make $\chi_j$ valid.  Because $\chi$ was already valid, this is equivalent to choosing a $P_j$ which contains none of the vertices of $\bigcup_{j'<j} P_{j'}$ other than $x$ and $y$.  By \Cref{prop:expansion}(f) with $|S| = 1$,  the number of paths in $\c{Q}$ containing a given vertex from $\bigcup_{j'<j} P_{j'}$ is at most $\varepsilon^{-1}\ell^{(b-2)b/(b-1)}$. Therefore, the number of $P_j\in \c{Q}(y)$ containing any of the vertices in $\bigcup_{j'<j} P_{j'}$ (other than $x$ and $y$) is at most a constant times $\ell^{(b-2)b/(b-1)}$, which 
		for $k_0$ sufficiently large (which makes $\ell$ sufficiently large) will be at most $\frac{1}{4} |\c{Q}(y)| =  \Omega(\ell^b)$. Thus at least three quarters of the paths $P_j\in \c{Q}(y)$ will make $\chi_j$ valid.

		To show that $\chi_j$ is $(s,b)$-compatible for most choices of paths in $\c{Q}(y)$, it remains to bound the number of ``bad" sets $\gam$ that must be avoided when choosing $P_j$.  To this end, for each integer $1\le p\le b-1$ define

		\[         \widetilde{\c{J}}_p:=\bigcup_{\substack{\chi'\subseteq {\chi},\\ \nu\subseteq \{w_1^j,\dots, w_{b-1}^j\}:\ |\nu| = p}}\c{J}_{s,b}(\chi';\nu)
		\ \ \cup\ 
		\bigcup_{\substack{\chi'\subseteq {\chi},\\ \nu\subseteq \{w_1^j,\dots, w_{b-1}^j\}:\ |\nu| = p}}\c{J}_{{\forest}}(\chi';\nu)
		\]
		Note that $\chi_j$ is $(s,b)$-compatible if and only if it is valid and does not contain any set $\gam\in \widetilde{\c{J}}_p$ for any value of $p$ (here we implicitly use that $\c{F}(\c{H}_b,D_b)=\emptyset$ since $D_b$ is always equal to $\infty$, so we can ignore $\c{J}_b(\chi';\nu)$ when checking for compatibility).

		We may use \Cref{cor:link} to bound the size of each link set above: consider $\chi'\subseteq \chi$ and $\nu\subseteq \{w_1^j,\dots, w_{b-1}^{j}\}$ with $|\nu|=p$. If $(u,x)\not\in \chi'$ or $(v,y)\not\in \chi'$, then $\c{J}_{s,b}(\chi';\nu) = \emptyset$ by \Cref{cor:link}. Otherwise, \Cref{cor:link} gives
		
		\begin{equation}
			\left|\c{J}_{s,b}(\chi';\nu)\right|		
			\le  2^{v(\theta_{a,b})+1}\delta^{p}\left(2^{2s/3}k^b\right)^{p/(b-1)}.\label{eq:Jsb}
		\end{equation}
		Similarly, if $\chi_\theta'\cup \nu$ does not induce a forest on at least one edge, then $\c{J}_{{\forest}}(\chi';\nu)=\emptyset$. If both $\chi_\theta'\cup \nu$ and $\chi_\theta'$ induce a forest on at least one edge, then \Cref{cor:link} gives
		\begin{equation}
			|\c{J}_{{\forest}}(\chi';\nu)|\le 
			2^{v(\theta_{a,b})+1}\delta^p k^{pb/(b-1)}.\label{eq:Jforest}
		\end{equation}
		It remains to deal with the case that $\chi_\theta'$ does not induce a forest on at least one edge but $\chi_\theta'\cup \nu$ does.  Analogous to the proof of \Cref{cl:oneEdge}, in this setting \Cref{cor:link} gives no meaningful bound, but we can show that  for any choice of $P_j$ in $\c{Q}(y)$, no element of the link set $\c{J}_{{\forest}}(\chi';\nu)$ can appear in $\chi_j$ (and therefore there are no ``bad options" that must be avoided in choosing $P_j$ from $\c{Q}$). To this end, consider any set $\gam\in \c{J}_{{\forest}}(\chi';\nu)$. By definition, this means $\gam_\theta = \nu$ and 
		\begin{equation}\label{eq:annoyingBound}
			\deg_{\c{H}}(\chi'\cup \gam) \ge D_{\forest}(\chi_\theta'\cup \nu).
		\end{equation}
		
		Since $\chi_\theta'$ does not induce any edges but $\chi'_\theta\cup \nu$ does, all of these induced edges of $\theta_{a,b}$ must be contained in the path $W_j:=uw_1^j\cdots w_{b-1}^j v$.  Since $D_{\forest}$ only depends on the number of edges induced, we have
		\begin{equation}\label{eq:annoyingBound2}
			D_{\forest}(\chi_\theta'\cup \nu) = D_{\forest}\left((\chi_\theta'\cup \nu)\cap W_j\right)..
		\end{equation}
		On the other hand, if for each $P_j\in \c{Q}(y)$ we let \[(W_j, P_j):= \left\{(u,x), (w_1^j,z_1^j),\dots, (w_{b-1}^j,z_{b-1}^j), (v,y)\right\},\] 
		then we have
		\begin{equation}\label{eq:annoyingBound3}
			\deg_\c{H}\big((\chi'\cup \gam)\cap (W_j,P_j)\big) \ge 
			\deg_\c{H}(\chi'\cup \gam), 
		\end{equation}
		since taking smaller sets can only cause $\deg_{\cH}$ to increase. Putting this all together, we obtain
		\begin{equation}
			\label{eq:annoyingBound4}
			\deg_\c{H}\big((\chi'\cup \gam)\cap (W_j,P_j)\big)
			\,\stackrel{\eqref{eq:annoyingBound3},\eqref{eq:annoyingBound},\eqref{eq:annoyingBound2}}{\ge}\,
			D_{\forest}\left((\chi_\theta'\cup \nu)\cap W_j\right).
		\end{equation}
		Notice that $(\chi_\theta'\cup \nu)\cap W_i=((\chi'\cup \gam)\cap (W_j,P_j))_\theta$. Thus \eqref{eq:annoyingBound4} says that $(\chi'\cup \gam)\cap (W_j,P_j)\in \c{F}_{\forest}$.  Using the notation introduced just before \Cref{cl:forest}, this means that the subgraph $H':=H_{(\chi'\cup \gam)\cap (W_j,P_j)}\sub P_j$ is an element of $\c{F}'_{\forest}$.  Recall that by \Cref{prop:expansion}, no path in $\c{Q}$ contains an element of $\c{F}'_{\forest}$ as a subgraph. Since $H'$ is a subgraph of $P_j\in \c{Q}$, we conclude that there is no choice of $P_j\in \c{Q}(y)$ such that $\chi_j$ contains $\gam\in \c{J}_{{\forest}}(\chi';\nu)$ in this case.

		Putting it all together, and writing $\c{P}_\text{possible} := \{ \gam\subseteq (W_j,P_j) : P_j\in \c{Q}(y)\}$, we obtain

		{\footnotesize \begin{align}
				|\widetilde{\c{J}}_p\cap \c{P}_\text{possible}| 
				&\stackrel{\eqref{eq:Jsb},\eqref{eq:Jforest}}{\le} 2^{v(\theta_{a,b})+1}\delta^pk^{pb/(b-1)}\left[ \left(2^{2s/3}\right)^{p/(b-1)} + 1\right]\cdot \text{\footnotesize$\left|\{(\chi',\nu)\ :\ \chi'\subseteq\chi, \nu\subseteq \{w_1^j,\dots, w_{b-1}^{j}\}, |\nu| = p\}\right|$}
				\notag\\				
				&\le  2^{2v(\theta_{a,b})+2b}\cdot \delta^p\cdot \left(2^{2s/3} k^b\right)^{p/(b-1)}.\label{eq:aBound}
		\end{align}}

		By \Cref{prop:expansion}(f), the number of $P_j$ which contain the projection $\gam_{G}$ of a given $\gam$ (of size $p$) is at most $\varepsilon^{-1}\ell^{(b-1-p)b/(b-1)}$. So combining this with \eqref{eq:aBound}, the number of $P_j$ which contain $\gam_{G}$ for any $\gam\in \widetilde{\c{J}}_p\cap \c{P}_\text{possible}$ is at most
		\[\left(\varepsilon^{-1}\ell^{(b-1-p)b/(b-1)}\right)\cdot 2^{2v(\theta_{a,b})+2b}\cdot \delta^p\cdot \left(2^{2s/3} k^b\right)^{p/(b-1)}.\]
		
		Summing over all values of $p$ from 1 to $b-1$ and simplifying slightly, the number of $P_j\in \c{Q}(y)$ that contain a ``bad" set $\gam$ of any size is at most
		
		\begin{equation}\label{eq:badBoundy}
			\max_{1\leq p\leq b-1}C\ell^{b}\c\cdot \left(2^{2s/3} (k/\ell)^b\right)^{p/(b-1)},
		\end{equation}
		where $C = \left(\varepsilon^{-1}\delta(b-1)2^{2v(\theta_{a,b})+2b}\right).$
		By \eqref{eq:ell}, and since $s = 2r+r'$ and $b\geq 3$, we have
		\[
		2^{2s/3} (k/\ell)^b \le 2^{2s/3} (4^b2^{-r})^b \le 4^{b^2}2^{2r'/3}.
		\]
		This gives 
		\[
		\eqref{eq:badBoundy} \le C'2^{2r'/3} \ell^b,
		\]
		where $C' = \left(\varepsilon^{-1}\delta(b-1)2^{2v(\theta_{a,b})+2b+2b^2}\right)$. By taking $\delta$ sufficiently small, we can assume $C'\le \frac{1}{4} \ep 4^{-b}$.  So, after taking into account that at most one quarter of the choices $P_j\in \c{Q}(y)$ have $\chi_j$ not valid, we find that the number of choices for $P_j\in \c{Q}(y)$ such that $\chi_j$ is $(s,b)$-compatible is at least
		\[\frac{3}{4}|\c{Q}(y)| - C'2^{2r'/3} \ell^b \ge \frac{1}{2}|\c{Q}(y)|\ge  \ep 4^{-2b^2}\cdot 2^{2s/3}k^b,\]
		where both inequalities used \eqref{eq:Qy}.  This gives the desired result.
	\end{proof}
	
	With \Cref{cl:pathsCase1} established, we are now nearly ready to finish Case 1.
	
	\begin{claim}\label{cor:manyGoodCase1}
		The number of $(s,b)$-compatible sets of size $v(\theta_{a,b})$ containing $(u,x)$ and $(v,y)$  is at least $$(\ep 4^{-2b^2})^a2^{s}k^{ab}.$$
	\end{claim}
	
	\begin{proof}
		This result will follow directly by an iterative application of \Cref{cl:pathsCase1}. As a base step, we take $\chi_0=\{(u,x),(v,y)\}$, which is $(s,b)$-compatible by \Cref{cl:case1bd}. With this we may apply \Cref{cl:pathsCase1} to obtain at least $\ep 4^{-2b^2} \cdot 2^{2s/3}k^b$ choices of a path $P_1$ in $G'$ such that the corresponding set $\chi_1$ is $(s,b)$-compatible. Iterating up to $j=a$, we obtain at least $$\left(\ep 4^{-2b^2} \cdot 2^{2s/3}k^b\right)^a \ge (\ep 4^{-2b^2})^a 2^sk^{ab}$$ distinct collections $P_1,\dots, P_a$ such that the corresponding sets $\chi_a$ are $(s,b)$-compatible. This completes the proof.
	\end{proof}
	
	Now we are ready to finish Case 1.  By \Cref{cl:case1bd}, the number of hyperedges in $\c{H}_{s,b}$ containing $(u,x)$ and $(v,y)$ is at most $${\del \ep^{-1}2^s k^{ab}}.$$
	By \Cref{cor:manyGoodCase1}, the number of $(s,b)$-compatible sets of size $v(\theta_{a,b})$ containing $(u,x)$ and $(v,y)$ is at least $$(\ep 4^{-2b^2})^a2^{s}k^{ab}.$$
	Therefore, provided $\delta$ is sufficiently small, there must be at least one $(s,b)$-compatible set $h$ of size $v(\theta_{a,b})$ that is not already in $\c{H}_{s,b}$. This $h$ may be added to $\c{H}_{s,b}$, completing the proof of \Cref{prop:main} when $t=b$.
	
	\subsubsection{Case 2: $t<b$ and $b-t$ Even}
	Parts of this proof are nearly identical to the previous case, and as such we omit some of the redundant details.
	
	Recall that $r$ is the unique integer such that $2^{-r}n\le m<2^{-r+1}n$.  Our goal in this case is to show that we can add a new theta graph to $\c{H}_{r,t}$. 	Let $y\in B_{t}$ be such that $(v,y)$ is in as few hyperedges in $\c{H}$ with $(u,x)$ as possible.  
	\begin{claim}\label{cl:tStart}
		The set $\chi=\{(u,x),(v,y)\}$ satisfies
		\[\deg_{\c{H}}(\chi)\le\f{\del \ep^{-2} 4^{2b}  k^{ab}n^2}{2^{-2r}k^{(2b-2t+1)/(b-1)}n^{(2t-1)/b}}\]
		and is $(r,t)$-compatible if $\del$ is sufficiently small.
	\end{claim}
	\begin{proof}
		Recall that $x$ is such that $(u,x)$ is contained in the fewest number of hyperedges in $\c{H}$ among all vertices in the set $X$.  This together with the definition of $y$ implies that the number of hyperedges containing both $(u,x)$ and $(v,y)$ is at most
		\[
		\frac{|\c{H}|}{|X|\cdot |B_t|}\le \f{\del k^{ab} n^2 }{\ep^2 \ell^{(2b-2t+1)/(b-1)}m^{(2t-1)/b}},
		\]
		where this last step used $|\c{H}|\le \del k^{ab} n^2$ and that $|B_{t}|\ge \ep \ell^{(b-t+1)/(b-1)}m^{(t-1)/b}$ and $|X|\ge \ep \ell^{(b-t)/(b-1)}m^{t/b}$  by Proposition~\ref{prop:expansion}(b) and (h).  Using $m\ge 2^{-r}n$ and $\ell\ge 4^{-b}2^r k$ from \eqref{eq:m} and \eqref{eq:ell} gives the first result.
		
		As in the $t=b$ case, we have $x\notin B_t$ by Proposition~\ref{prop:expansion}(e), so $y\ne x$ and the set $\chi$ is valid.  Any $\chi'\subsetneq \chi$ trivially fails to be in $\c{F}_{\forest}\cup \c{F}_t\cup \c{F}_{r,t}$, and to show $\chi$ is not in this set it suffices to show
		\[\deg_{\c{H}_{r,t}}(\chi)<D_{r,t}(\chi_\theta)=\ceil{\f{\del^{-2} k^{ab}n^2}{2^{-2r} k^{(2b-2t+1)/(b-1)}n^{(2t-1)/b}}},\]
		and this follows by the first result.  We conclude that $\chi$ is $(r,t)$-compatible.
		
	\end{proof}
	Now that we have selected our two high degree vertices $x,y$ of our theta graph, we build the rest of the theta graph as follows. First, we work our way out from $y$ by selecting neighbors $z_{b-1}^j\in B_{t-1}$ of $y$, then neighbors $z_{b-2}^j\in B_t$ of each $z_{b-1}^j$, and so on, until we have chosen vertices $z_t^j\in B_t$. Then, once we have chosen the vertices $z_t^j$, we select paths from the set $\c{Q}$ connecting the vertices $z_t^j$ to $x$.  
	
	To do the first part, we use the following claim. Here we recall that the paths of $\theta_{a,b}$  are denoted $uw_1^j\cdots w_{b-1}^jv$, and for this claim we adopt the convention that $w_b^j:=v$ and $z_b^j:=y$.  We also recall that $F_t\sub V(\theta_{a,b})$ is defined to be the set of $w_i^j$ with $t\le i< b$ and $i-t$ even.  In particular, $w_{b-1}^j\notin F_t$ when $b-t$ is even.
	\begin{claim}\label{cl:tMiddleLayer}
		Let $t\le i\le b-1$ and $1\le j\le a$  be integers, and let $\chi$ be an $(r,t)$-compatible set consisting of the pairs $(u,x),(v,y)$, and $(w_{i'}^{j'},z_{i'}^{j'})$ for all $i',j'$ with either $i'>i$ or with $i'=i$ and $j'<j$. 
		\begin{itemize}
			\item  If $i-t$ is odd and $z_{i+1}^j\in B_{t}$, then there exist at least $\half \ep \ell^{b/(b-1)}$ choices $z_{i}^j\in B_{t-1}\cap N_{G'}(z_{i+1}^j)$ such that $\chi\cup \{(w_i^j,z_i^j)\}$ is $(r,t)$-compatible.   
			\item If $i-t$ is even and  $z_{i+1}^j\in B_{t-1}$, then there exist at least $\half \ep \ell m^{1/b}$ choices $z_i^j\in B_t\cap N_{G'}(z_{i+1}^j)$ such that $\chi\cup \{(w_i^j,z_i^j)\}$ is $(r,t)$-compatible. 
		\end{itemize}
	\end{claim}
	\begin{figure}[H]
		\centering
		\includegraphics[width=0.27\textwidth]{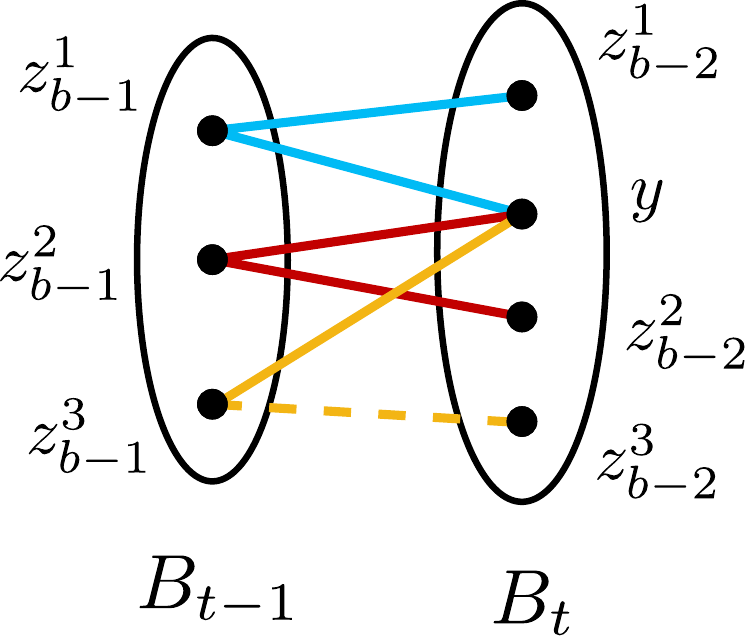}
		\vspace*{1.3ex}\caption{When $a=3$, after picking $z_{b-1}^1,z_{b-1}^2,z_{b-1}^3,z_{b-2}^1,z_{b-2}^2$ (in that order), the algorithm next selects $z_{b-2}^3$.}
		\label{fig:Bt}
	\end{figure}
	
	\begin{proof}
		Observe that if $z_{i}^{j}\in N_{G'}(z_{i+1}^j)$ is a vertex such that $\chi\cup \{(w_{i}^j,z_{i}^{j})\}$ is not $(r,t)$-compatible, then either $z_{i}^{j}\in \chi_G$ (which can only hold for $O(1)$ vertices), or there exists some $\chi'\sub \chi$ with
		\[\{(w_{i}^{j},z_{i}^{j})\}\in  \c{J}_{\forest}(\chi';w_{i}^{j})\cup  \c{J}_{r,t}(\chi';w_{i}^{j})\cup \c{J}_{t}(\chi';w_{i}^{j}).\]
		Thus it suffices to show that each of these sets are small for each $\chi'\subseteq \chi$. 
		
		First consider $\c{J}_{\forest}(\chi';w_{i}^{j})$.  If $\chi'_\theta\cup w_{i}^{j}$ induces at most one edge, then this link set is empty by \Cref{cl:oneEdge}.  If this is not the case, then $\chi'_\theta$ must induce at least one edge since $w_{i}^{j}$ only has one edge incident to $\chi'_\theta$ (this implicitly uses $i\ge t\ge 2$, as otherwise $w_i^j$ would also be adjacent to $u$).  By \Cref{cor:link}, we find \begin{equation}|\c{J}_{\forest}(\chi';w_{i}^{j})|\le 2^{v(\theta_{a,b})+1}\del k^{b/(b-1)}.\label{eq:forestT}\end{equation}
		
		Next consider $\c{J}_{r,t}(\chi';w_{i}^{j})$, which we recall is based off of the codegree function defined in \Cref{def:Dst}.  If $\{u,v\}\not\sub \chi'_\theta$, then this link set is empty by \Cref{cor:link}, so we may assume $\{u,v\}\sub \chi'_\theta$.  Then \Cref{cor:link} gives \[|\c{J}_{r,t}(\chi';w_{i}^{j})|\le 2^{v(\theta_{a,b})+1}\del 2^{2r/3} k^{b/(b-1)}\hspace{.4em}\tr{ if }i-t\tr{ is odd},\] since if $i-t$ is odd, adding $w_{i}^j\notin F_t$ to $\chi'_\theta$ keeps the parameter $f=|\nu\cap F_t|$ in \Cref{def:Dst} the same while increasing $|\nu|$.  Similarly, \[|\c{J}_{r,t}(\chi';w_{i}^{j})|\le 2^{v(\theta_{a,b})+1}\del 2^{2r/3}kn^{1/b}\hspace{.4em}\tr{ if }i-t\tr{ is even},\] since $f$ and $\nu$ both increase by 1.
		
		Finally consider $\c{J}_{t}(\chi';w_{i}^{j})$, which we recall is based off of the codegree function defined in \Cref{def:Dt}. Again we may assume $\{u,v\}\sub \chi'_\theta$.  If $w_{b-1}^{j'}\in \chi'_\theta$ for some $j'$, then the argument and final bound is exactly the same as in the case for $\c{J}_{r,t}$ (with $g$ taking the role of $f$ in exactly the same way as before).  We next consider the subcase $w_{b-1}^{j'}\notin \chi'_\theta$ for all $j'<j$.  If $i\ne b-1$, then $\chi'_\theta\cup \{w_i^j\}$ contains no vertex of the form $w_{b-1}^{j'}$, so $D_t(\chi'_\theta\cup \{w_i^j\})=\infty$, and hence the link set is empty by \Cref{lem:link}.  If $i=b-1$, then $\chi'_\theta\sub\{u,v,w_{b-1}^1,\ldots,w_{b-1}^{a}\}$ by the hypothesis of the claim, so our assumption $w_{b-1}^{j'}\notin \chi'_\theta$ implies $\chi'=\{(u,x),(v,y)\}$.  Thus \[\deg_{\c{H}_t}(\chi'\cup \{(w_{b-1}^{j},z_{b-1}^{j})\})\le \deg_{\c{H}}(\{(v,y),(w_{b-1}^{j},z_{b-1}^{j})\})<D_{\forest}(\{v,w_{b-1}^{j}\})=D_t(\chi'_\theta\cup \{w_{b-1}^{j}\}),\]
		where the first inequality used that we are looking at the codegree of a smaller set in a larger hypergraph, the second inequality used \eqref{eq:G0} (i.e.\ that every edge in $G'$ has codegree smaller than that given by $D_{\forest}$), and the equality used $|\chi'_\theta\cup \{w_{b-1}^{j}\}|=3$ and $g=|(\chi'_\theta\cup \{w_{b-1}^{j}\})\cap F_t|=0$ in the definition of $D_t$ for $b-t$ even.  This implies $\c{J}_{t}(\chi';w_{b-1}^{j})=\emptyset$.
		
		By summing up the sizes of all of these sets over all possible choices of $\chi'\sub \chi$ (as well as the number of choices $z_{i}^{j}\in \chi_G$), we find  when $i-t$ is odd that the number of $z_{i}^{j}$ which can not be selected is at most \[O_{a,b}(\del 2^{2r/3} k^{b/(b-1)})=O_{a,b}(\del \ell^{b/(b-1)}),\] with the last step using $\ell\ge 4^{-b}2^r k$. By Proposition~\ref{prop:expansion}(d),  $z_{i+1}^j\in B_{t}$ has at least $\ep \ell^{b/(b-1)}$ neighbors in $B_{t-1}$, and  for $\del$ sufficiently small this is at least twice the number of forbidden choices.  Essentially the same reasoning holds for the $i-t$ even case after noting $k^{b/(b-1)}\le k n^{1/b}$ when applying \eqref{eq:forestT}. We conclude the result.
		
	\end{proof}
	
	By starting with the two high-degree vertices $(u,x)$ and $(v,y)$, and iteratively applying \Cref{cl:tMiddleLayer}, we can find many $(r,t)$-compatible sets $\chi$ with $\chi_\theta=\{u,v\}\cup \bigcup_{i\ge t}\{w_i^1,\ldots,w_i^a\}$.  To get the remaining vertices corresponding to $w_i^j$ with $i<t$, we use the same strategy as in the $t=b$ case of choosing paths from $\c{Q}$.

	\begin{claim}\label{cl:tFinish}
		Let $P_1,\dots, P_{j-1}$ be a collection of paths in $\c{Q}$ ending in $B_t$, and for each path $P_{j'}$, write $P_{j'} = x z_1^{j'} \cdots z_{t}^{j'}$. Suppose  that the set 
		$$\chi := \{(u,x),(v,y)\}\cup \bigcup_{i\ge t}\{(w_i^1,z_i^1),\ldots,(w_i^a,z_i^a)\} \cup  \bigcup_{j'<j} \left\{(w_1^{j'},z_1^{j'}),\dots, (w_{t}^{j'},z_{t}^{j'})\right\}$$
		is $(r,t)$-compatible. 
		Then there are  at least $\half \ep \ell^{(t-1)b/(b-1)}$ choices of a path $P_j = xz_1^j \cdots z_{t}^j$ in $\c{Q}$ so that 
		$$\chi_j:=\chi \cup \left\{ (w_1^j,z_1^j),\dots, (w_{t}^j,z_{t}^j)\right\}$$
		
		is $(r,t)$-compatible. 
	\end{claim}
	\begin{figure}[H]
		\centering
		\includegraphics[width=0.46\textwidth]{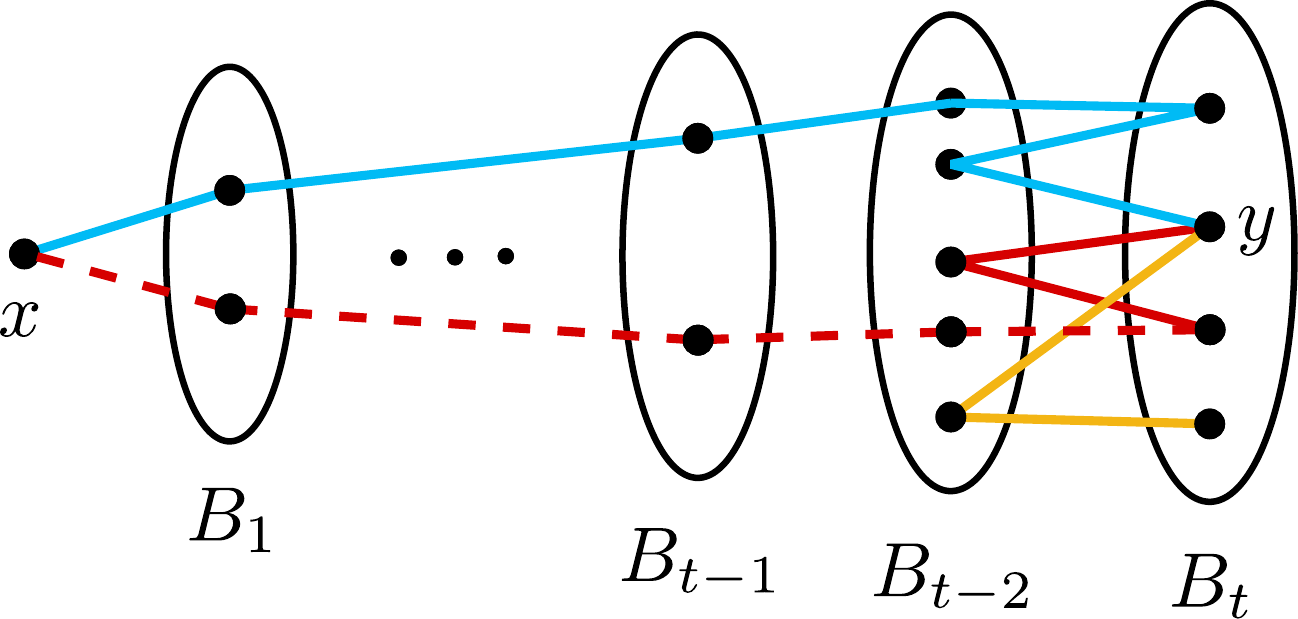}
		\caption{After picking all of the vertices $z_i^j$ with $i\ge t$, the algorithm next picks a path $P_1$ from $z_t^1$ to $x$, then a path $P_2$ from $z_t^2$ to $x$, and so on.}
		\label{fig:LessB}
	\end{figure}
	\begin{proof}[Sketch of Proof]
		The argument is almost identical to that of Claim~\ref{cl:pathsCase1} so we only sketch the details (with our notation defined analogously as before).  By Proposition~\ref{prop:expansion}(c) we have that there are at least $\ep \ell^{(t-1)b/(b-1)}$ paths in $\c{Q}$ from $z_t^{j}$ to $x$.  Using \Cref{prop:expansion}(f) we find that very few of these paths contain any of the other vertices of $\chi_G$ besides $x$ and $y$.  
		
		By using \Cref{cor:link}, we find that each of the sets $\c{J}_{r,t}(\chi';\nu),\ \c{J}_t(\chi';\nu)$, and $\c{J}_{\forest}(\chi';\nu)$ after intersecting with $\c{P}_{\text{possible}}$ are all of size $O((\del 2^{2r/3} k^{b/(b-1)})^p)$ whenever $\chi'\sub \chi$ and $\nu\sub \{w_1^j,\ldots,w_{t-1}^j\}$ with $|\nu|=p$ (here we use that $\nu\cap F_t=\emptyset$ for any such $\nu$, so $f,g$ in the definitions of $D_{r,t},D_t$ do not change when going from $\chi'_\theta$ to $\chi'_\theta\cup \nu$).  From here essentially the same computations as before go through.
	\end{proof}
	Combining the previous three claims, we find that our algorithm produces a large number of theta graphs. 
	\begin{claim}\label{manyCase2}
		The number of $(r,t)$-compatible sets of size $v(\theta_{a,b})$ containing $(u,x)$ and $(v,y)$  is at least $$\Om\left(2^{2r}k^{ab-\f{2b-2t+1}{b-1}}n^{2-\f{2t-1}{b}}\right).$$
	\end{claim}
	\begin{proof}
		The result follows by iteratively applying \Cref{cl:tMiddleLayer,cl:tFinish}. Starting with $\chi_0:=\{(u,x),(v,y)\}$, which is $(r,t)$-compatible by \Cref{cl:tStart}, we repeatedly apply \Cref{cl:tMiddleLayer} to build paths $z_t^j\cdots z_{b-1}^j y$ (for $1\leq j\leq a$); we then finish by repeatedly applying \Cref{cl:tFinish} to select paths $x z_1^j\cdots z_{t}^j y$. In total, we find that the number of $(r,t)$-compatible sets $h$ of size $v(\theta_{a,b})$ with $(u,x),(v,y)\in h$ is at least
		\begin{equation}\left( \left(\half \ep \ell m^{1/b}\right)^{(b-t)/2}\cdot \left(\half \ep \ell^{b/(b-1)}\right)^{(b-t)/2}\cdot \left(\half \ep \ell^{(t-1)b/(b-1)}\right)\right)^a,\label{eq:choices}\end{equation}
		where the first two terms use that for each path $uz_1^j\cdots z_b^jv$ we get a factor of $\half \ep \ell m^{1/b}$ for each vertex in position $i\in \{t,t+2,\ldots,b-2\}$ and a factor of $\half \ep \ell^{b/(b-1)}$ for each $i\in \{t+1,t+3,\ldots,b-1\}$ by Claim~\ref{cl:tMiddleLayer}, and the last term uses Claim~\ref{cl:tFinish}.  The expression above is equal to some positive constant depending only on $a,b,\ep$ times
		\begin{align*}\ell^{ab-\f{a(b-t)}{2(b-1)}}m^{\f{a(b-t)}{2b}}&= \ell^{ab-\f{a(b-t)}{2(b-1)}} m^{\f{(a-4)(b-t)-2}{2b}}\cdot m^{\f{4b-4t+2}{2b}}\\&\ge \ell^{ab-\f{a(b-t)}{2(b-1)}+\f{(a-4)(b-t)-2}{2(b-1)}}\cdot   m^{\f{2b-2t+1}{b}}=\ell^{ab-\f{2b-2t+1}{b-1}}\cdot m^{2-\f{2t-1}{b}},\end{align*}
		where the inequality used \eqref{eq:minDegree}, i.e.\ $m^{1/b}\ge \ell^{1/(b-1)}$, and implicitly that $a\ge 6$ so that the exponent of $m$ is positive.  Finally, using $\ell=\Om(2^r k)$ and $m=\Om(2^{-r} n)$ crudely gives
		\[\ell^{ab-\f{2b-2t+1}{b-1}}\cdot m^{2-\f{2t-1}{b}}=\Om\left(2^{(ab-2)r}k^{ab-\f{2b-2t+1}{b-1}}\cdot 2^{-(2-\f{2t-1}{b})r}n^{2-\f{2t-1}{b}}\right)=\Om\left(2^{2r}k^{ab-\f{2b-2t+1}{b-1}}n^{2-\f{2t-1}{b}}\right),\]
		where this last step used $ab\ge 6$.  
	\end{proof}
	We are now ready to finish Case 2. If $\delta$ is sufficiently small in terms of $a,b,\ep$, the number of theta graphs guaranteed by \Cref{manyCase2} exceeds the codegree bound in \Cref{cl:tStart}; thus there exists some $(r,t)$-valid set $h$ obtained through our algorithm which is not already a hyperedge of $\c{H}$.  Adding such an $h$ to $\c{H}_{r,t}$ gives the result in this case.
	
	\subsubsection{Case 3: $t<b$ and $b-t$ Odd}
	This case is nearly identical to the previous one, and as such we only sketch the proof.
	
	Again our goal is to add a new hyperedge to $\c{H}_{r,t}$.  To start, we pick $y\in B_{t-1}\setminus \{x\}$ such that $(v,y)$ is in as few hyperedges with $(u,x)$ as possible.  Here we emphasize that, in the previous case, we picked $y\in B_t$ and hence immediately obtained $y\ne x$ (since each element of $B_t$ is the endpoint of a path with $x$), but here we have to be slightly more careful and explicitly enforce $y\ne x$. However, since no hyperedge of $\cH$ contains both $(u,x)$ and $(v,x)$ (since every hyperedge is a valid set), and since $|B_{t-1}\sm\{x\}|\ge \half \ep \ell^{(b-t+1)/(b-1)}m^{(t-1)/b}$ by Proposition~\ref{prop:expansion}(b), we find that $\deg_{\c{H}}(\{(u,x),(v,y)\})$ is at most twice the bound from 
	Claim~\ref{cl:tStart}, and the rest of the proof showing that this set is $(r,t)$-compatible goes through in exactly the same way as in Claim~\ref{cl:tStart}.
	
	From here we apply \Cref{cl:tMiddleLayer} exactly as written (since $w_i^j\in F_t$ depends only on the parity of $i-t$ and not of $b-t$); the proof of \Cref{cl:tMiddleLayer} also remains word for word the same, with the only minor exception being that we have $g = g(\nu):=|\nu\cap F_t|-1$ (which again implies $g=0$ when $i=b-1$ and $\chi'_\theta=\{u,v\}$).
	
	Finally, we choose paths in $\c{Q}$ going from each of the $z_t^j$ vertices to $x$, and again the statement and proof of Claim~\ref{cl:tFinish} remain exactly the same.  With this, the total number of choices for the algorithm to produce an $(r,t)$-compatible set is 
	\[\left( \left(\half \ep \ell m^{1/b}\right)^{(b-t+1)/2}\cdot \left(\half \ep \ell^{b/(b-1)}\right)^{(b-t-1)/2}\cdot \left(\half \ep \ell^{(t-1)b/(b-1)}\right)\right)^a,\]
	since in this setting we get a factor of $\half \ep \ell m^{1/b}$ for each vertex in position $i\in \{t,t+2,\ldots,b-1\}$, of which there are $(b-t+1)/2$.  This quantity is at least as large as \eqref{eq:choices}, so we conclude that for $\del$ sufficiently small the number of choices is more than the  number of hyperedges containing $(u,x),(v,y)$ in $\c{H}$. With this we conclude the result.

	\section{Balanced Supersaturation for Edges}\label{sec:edgeSupersaturation}
	
	In the previous section we showed that $\theta_{a,b}$ exhibits balanced supersaturation for vertices in terms of the (complicated) codegree function $D'_t$.  We begin by simplifying this function.
	
	\begin{prop}\label{prop:codegrees}
		For all $a\ge 100$ and $b\ge 3$, let $\del>0$ and $D'_t$ be as in \Cref{thm:balancedVertex}.  There exist constants $C',k_0>0$ such that if $n^{1-1/b}\ge k\ge k_0$ and $\nu\sub V(\theta_{a,b})$ induces $e$ edges, where $1\le e\le e(\theta_{a,b})-1$, then

		\[D'_t(\nu)\le  \f{C'k^{ab} n^{2}}{kn^{1+1/b} \big(\min\big\{k^{b/(b-1)},kn^{\frac{b-1}{b(ab-1)}}\big\}\big)^{e-1}}. \]

	\end{prop}
	Note that $n^{1-1/b}\ge k$ always holds if we are considering $n$-vertex graphs $G$ with $kn^{1+1/b}$ edges.  We defer the proof of  \Cref{prop:codegrees} for the moment and show that together with \Cref{thm:balancedVertex}, it  implies a balanced supersaturation result for edges which we will use to complete the proof of \Cref{thm:main}; see \Cref{thm:containers} below.
	
	\begin{cor}\label{cor:balancedEdges}
		For all $a\ge 100$ and $b\ge 3$, there exist constants $C,k_0>0$ such that the following holds for all $n\in \N$ and $k\ge k_0$.  If $G$ is an $n$-vertex graph with $kn^{1+1/b}$ edges, then there exists a hypergraph $\cH$ on $E(G)$ whose hyperedges are copies of $\theta_{a,b}$ and is such that $|\cH|\ge C^{-1} k^{ab}n^2$ and such that for every $\sig\sub E(G)$ with $1\le |\sig|\le e(\theta_{a,b})-1$, we have
		
		\[\deg_{\cH}(\sig)\le  \f{Ck^{ab} n^{2}}{kn^{1+1/b} \big(\min\big\{k^{b/(b-1)},kn^{\frac{b-1}{b(ab-1)}}\big\}\big)^{|\sig|-1}}. \]
		
	\end{cor}
	
	\begin{proof}

		Let $\cH'_t$ be the $D'_t$-good $G$-hypergraph on $V(\theta_{a,b})\times V(G)$ guaranteed by \Cref{thm:balancedVertex}. We would like to translate $\cH'_t$ into a hypergraph $\cH$ on $E(G)$ satisfying the codegree bounds above.
		
		This will be conceptually straightforward, but a little tedious. In essence, the hyperedges of $\cH'_t$ correspond to theta graphs in $G$, and we will define $\cH$ to be the hypergraph corresponding to these theta graphs.  However, we must deal with two small issues with this translation: (1) a single theta graph in $G$ may appear isomorphically several times in $\cH'_t$, and (2) the codegree bound $D_t'(\nu)$ depends on the number of edges \textit{induced} by $\nu$, whereas the bound in \Cref{cor:balancedEdges} depends only on $|\sigma|$ for an \textit{arbitrary} set of edges $\sigma$, even if the vertices used by $\sigma$ induce additional edges. Neither of these issues is a real obstacle (in particular, the second can only improve the codegrees), but we will need some additional notation in order to address them.

		For each valid set $\chi$, we will define the corresponding set of edges induced in $G$ (excluding ``extraneous" edges that do not play a role in the isomorphic copy of $\theta_{a,b}$) as follows:
		\[E_\chi=\{zz':(w,z),(w',z')\in \chi,\ ww'\in E(\theta_{a,b})\}.\]
		In particular, $E_h\sub E(G)$  is a copy of $\theta_{a,b}$ in $G$ 
		for every hyperedge $h\in \cH'_t$ (since every hyperedge $h\in \cH'_t$ is a valid set of size $v(\theta_{a,b})$). Define $\cH$ to be the hypergraph with hyperedge set $\{E_h:h\in \cH'_t\}$.  Observe that $|\cH|\ge \rec{v(\theta_{a,b})!} |\cH'_t|=\Om(k^{ab}n^2)$, so it remains to check the codegree conditions -- that is, to bound $\deg_{\cH}(\sig)$ for each set of edges $\sigma$ in $G$.

		Fix a set of edges $\sig\sub E(G)$.  We need to get an understanding of which valid sets ``correspond'' to $\sig$.  To this end, let $\sig_v\sub V(G)$ be the set of vertices used by the edges $\sig$, and let $\mathcal{X}$ be the set of all valid $\chi\in V(\theta_{a,b})\times V(G)$ with $\chi_G=\sig_v$ and $\sig\sub E_\chi$. 
		See \Cref{fig:Echi} for an example.
		
		\begin{figure}[H]
			\centering
			\includegraphics[width=0.8\textwidth]{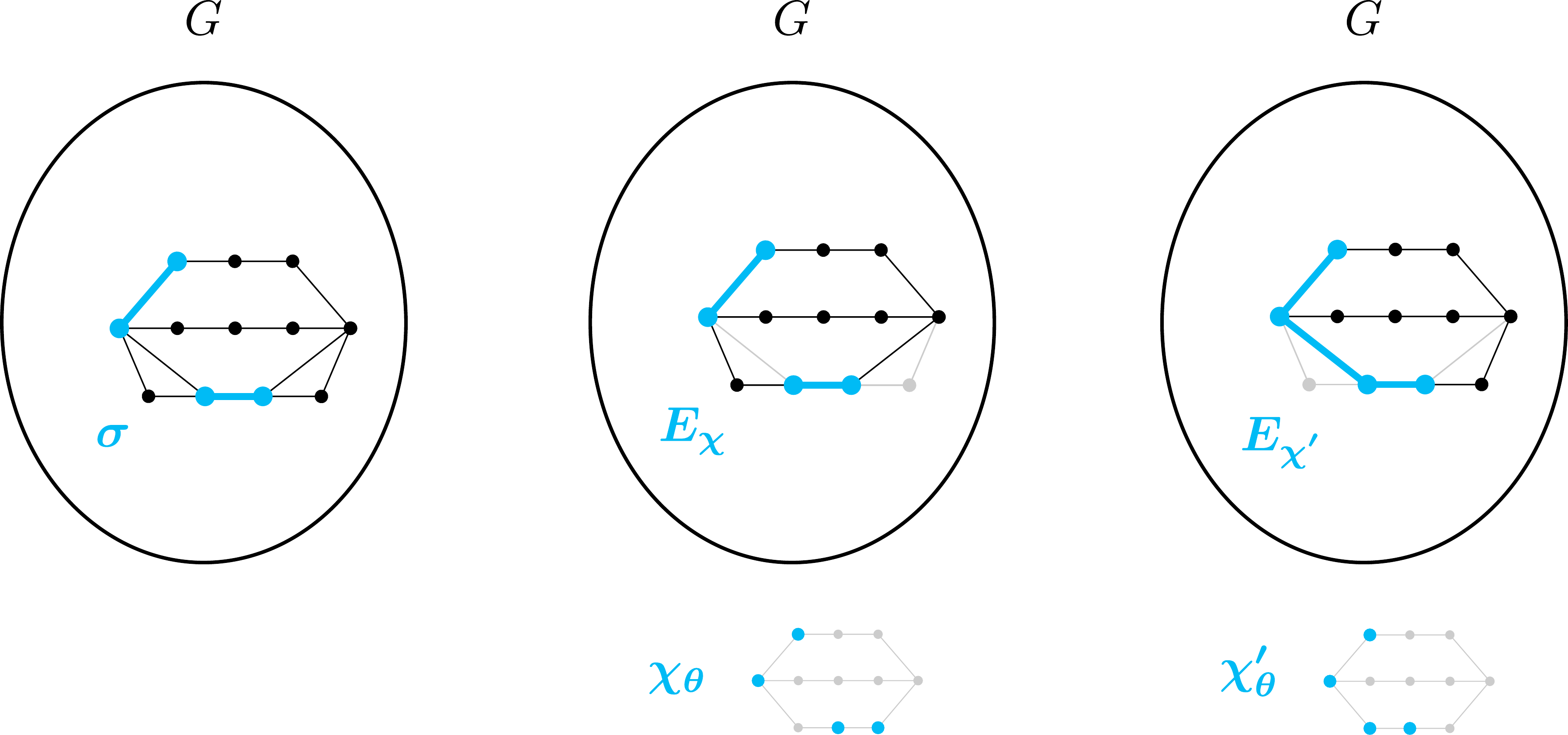}
			\caption{A pair of edges $\sig\sub E(G)$, together with two valid sets $\chi,\chi'$ in $\mathcal{X}$, i.e.\ valid sets having $\chi_G,\chi'_G=\sig_v$ and $\sig\sub E_\chi,E_{\chi'}$.  Note that $E_{\chi'}\ne \sig$ because $\chi'_\theta$ induces three edges in $\theta_{a,b}$.}
			\label{fig:Echi}
		\end{figure}
		
		With this notation, we can convert the codegree bounds in $H_t'$ to a bound on $\deg_{\cH}(\sig)$ as follows.
		\begin{claim}
			We have
			\[\deg_\cH(\sig)\le \sum_{\chi\in \mathcal{X}} \deg_{\cH'_t}(\chi).\] 
		\end{claim}
		\begin{proof}
			
			We would like to show that each hyperedge $h\in\cH$ counted by $\deg_\cH(\sigma)$ corresponds to at least one hyperedge $h' \in \cH_t'$ counted by $\deg_{\cH'_t}(\chi)$ for some $\chi\in\mathcal{X}$. We first observe that if $h\in \cH$, then by definition of $\cH$, there exists some $h'\in \c{H}'_t$ with $h=E_{h'}$. We wish to show that this $h'$ contains a set $\chi\in\mathcal{X}$, so that it will be counted by $\deg_{\cH'_t}(\chi)$.

			To this end, we observe that if the edge set $\sig$ is contained in $ h=E_{h'}$, then the corresponding vertex set $\sig_v$ is contained in $h'_G$; so there exists some $\chi\sub h'$ such that $\chi_G = \sig_v$. We claim that $\sig\sub E_\chi$ as well.  To see this, note that if       
			$zz'\in \sig\sub E_{h'}$, then $(w,z),(w',z')\in h'$ for some $ww'\in E(\theta_{a,b})$ by the definition of $E_{h'}$. Since  $z,z'\in \sig_v=\chi_G$ and $\chi\sub h'$, we have $(w,z),(w',z')\in \chi$ as well, giving $z z'\in E_\chi$. So by definition, $\chi\in \mathcal{X}$ as desired.
		\end{proof}
		To finish, it remains to bound the sum above. First, notice that there are only a constant number of terms: each element of $\c{X}$ is uniquely identified by $\chi_\theta$ (since $\chi_G=\sig_v$ for each $\chi \in \c{X}$), and hence $|\mathcal{X}|\le 2^{v(\theta_{a,b})}$.  
		Since $\c{H}'_t$ is $D'_t$-good, we have $\deg_{\cH'_t}(\chi)\le D'_t(\chi_\theta)$, and hence
		\[\deg_\cH(\sig)\le \sum_{\chi\in \mathcal{X}} \deg_{\cH'_t}(\chi)\le \sum_{\chi \in \mathcal{X}} D'_t(\chi_\theta)\le 2^{v(\theta_{a,b})}\cdot \frac{C'k^{ab}n^2}{kn^{1+1/b}(\min\{k^{b/(b-1)},kn^{(b-1)/(ab-1)}\})^{|\sig|-1}},\]
		where this last step used that each $\chi\in \mathcal{X}$ induces $|E_\chi|\ge |\sig|$ edges, together with the bound on $D_t'$ given by \Cref{prop:codegrees}.  This gives the desired result by taking $C=2^{v(\theta_{a,b})}C'$.		
	\end{proof}
	
	\subsection{Proof of \Cref{prop:codegrees}} The rest of this section is dedicated to proving \Cref{prop:codegrees}, which we emphasize will consist entirely of (moderately involved) arithmetic and case analysis. We will abuse notation slightly by identifying a vertex set $\nu\sub V(\theta_{a,b})$ with its induced subgraph in $\theta_{a,b}$.  For example, we may say $\nu$ contains a cycle to mean its induced subgraph contains a cycle.  Unless stated otherwise, $e$ will refer to the number of edges that $\nu$ induces in $\theta_{a,b}$.  We let $\del>0$ be the constant guaranteed by \Cref{thm:balancedVertex}, and throughout we assume $n^{1-1/b}\ge k$ and $b\geq 3$.  We recall that our goal is to show that there exists a constant $C'>0$ such that for $k$ sufficiently large and for all  $\nu\sub V(\theta_{a,b})$ inducing $e$ edges for $1\le e\le e(\theta_{a,b})-1$, we have
	\[D'_t(\nu)\le  \f{C'k^{ab} n^{2}}{kn^{1+1/b} \big(\min\big\{k^{b/(b-1)},kn^{\frac{b-1}{b(ab-1)}}\big\}\big)^{e-1}}, \]
	where the definition of $D'_t(\nu)$ will be recalled below. We begin with an easy case.
	
	\begin{lem}\label{lem:mOne}
		For any codegree function $D$, if $\nu\sub V(\theta_{a,b})$ is such that $D(\nu)=1$ and $\nu$ induces $e\ge 1$ edges, then \[D(\nu)\le \frac{k^{ab}n^2}{k n^{1+1/b}\big(kn^{\frac{b-1}{b(ab-1)}}\big)^{e-1}}.\]
	\end{lem}
	\begin{proof}
		This follows immediately from $e\le ab$.
	\end{proof}
	
	We remind the reader that
	\[D'_t(\nu):=\begin{cases}
		D_{\forest}(\nu)& \nu \tr{ induces a forest},\\ 
		\min\{D_t(\nu),20(D_{0,t}(\nu)+\ceil{\log n})\} & \tr{otherwise}.
	\end{cases}\]
	For ease of reading, we recall each of the functions mentioned above before they are used.  First, we recall
	\[D_{\forest}(\nu)=\ceil{\f{k^{ab}n^2}{\del kn^{1+1/b}\cdot (\del k^{b/(b-1)})^{e-1}}}\]
	whenever $\nu\sub V(\theta_{a,b})$ induces a forest with $e\ge 1$ edges.
	\begin{lem}\label{lem:mForest}
		If $a\ge 3$, $2\le t\le b$, and $\nu\sub V(\theta_{a,b})$ induces a forest on $e\ge 1$ edges, then 
		\[D'_t(\nu)\le \frac{2 \del^{-ab} k^{ab}n^2}{k n^{1+1/b}\big(\min\big\{k^{b/(b-1)},kn^{\frac{b-1}{b(ab-1)}}\big\}\big)^{e-1}}.\]
	\end{lem}
	\begin{proof}
		If $D'_t(\nu)=D_{\forest}(\nu)=1$ then the result follows from \Cref{lem:mOne}, and otherwise
		\[D'_t(\nu)=D_{\forest}(\nu)\le \f{2k^{ab}n^2}{\del kn^{1+1/b}\cdot (\del k^{b/(b-1)})^{e-1}},\]
		where the factor of 2 comes from the ceiling function and the assumption $D_{\forest}(\nu)>1$.  The result follows since $e-1\le ab-1$.
	\end{proof}

	It remains to prove the result when $\nu$ contains a cycle.  To help with the case analysis, we show that it suffices to prove the result when $\nu$ consists of paths of length $b$, i.e.\ when $\nu$ contains no leaves or isolated vertices.
	
	\begin{lem}\label{lem:cycleReduction}
		Let $D$ be a codegree function such that if $\nu\sub V(\theta_{a,b})$ contains a cycle, then either $D(\nu)=1$ or $D(\nu)\le  2 \del^{-1} k^{-b/(b-1)}D(\nu\sm \{w\})$ for every $w\in \nu$.
		
		If $C\ge 1$ is a constant such that for every $\nu$ which consists of paths of length $b$, we have
		\[D(\nu)\le \frac{C k^{ab}n^2}{k n^{1+1/b}\big(\min\big\{k^{b/(b-1)},kn^{\frac{b-1}{b(ab-1)}}\big\}\big)^{e-1}},\]
		then for every $k\ge 2\del^{-1}$ and $\nu$ which contains a cycle, we have
		
		\[D(\nu)\le \frac{C(2\del^{-1})^{|\nu|} k^{ab}n^2}{k n^{1+1/b}\big(\min\big\{k^{b/(b-1)},kn^{\frac{b-1}{b(ab-1)}}\big\}\big)^{e-1}}.\]
	\end{lem}
	\begin{proof}
		Assume the hypotheses hold for $D$.  We prove by induction on $|\nu|$ that any $\nu\sub V(\theta_{a,b})$ containing a cycle satisfies the desired inequality.   For this proof, we recall that $u,v$ always denote the two high degree vertices of $\theta_{a,b}$.
		
		Say we have proved the result up to some set $\nu$ which induces $e$ edges.  If $D(\nu)=1$ then the inequality follows from Lemma~\ref{lem:mOne}, so we may assume $D(\nu)>1$.   If $\nu$ consists of paths of length $b$ then the result follows by hypothesis. Otherwise, there exists some vertex $w\in \nu\sm \{u,v\}$ which is adjacent to at most one other vertex in $\nu$.  Thus $\nu\sm \{w\}$ induces a graph containing a cycle with at least $e-2$ edges.  By our hypothesis on $D$, we find
		\begin{align*}D(\nu)&\le 2 \del^{-1} k^{-b/(b-1)}\cdot D(\nu\sm \{w\})\\
			&\le 2 \del^{-1} k^{-b/(b-1)}\cdot \frac{C(2\del^{-1})^{|\nu|-1} k^{ab}n^2}{k n^{1+1/b}\big(\min\big\{k^{b/(b-1)},kn^{\frac{b-1}{b(ab-1)}}\big\}\big)^{e-2}}\\&\le \frac{C(2\del^{-1})^{|\nu|} k^{ab}n^2}{k n^{1+1/b}\big(\min\big\{k^{b/(b-1)},kn^{\frac{b-1}{b(ab-1)}}\big\}\big)^{e-1}},\end{align*} 
		where the second inequality used that our inductive hypothesis applies to $\nu\sm\{w\}$ (since $\nu\sm\{w\}$ contains a cycle and $D(\nu\sm\{w\})\ge (\del/2)k^{b/(b-1)} D(\nu)>1$).  This gives the desired result.
	\end{proof}
	We will show that essentially all of our remaining codegree functions are of the form described in \Cref{lem:cycleReduction}.  First, we recall
	\[D_{0,b}(\nu)=\ceil{\f{k^{ab}n^2}{n^2 (k^{b/(b-1)})^{|\nu|-2}\del^{|\nu|}}}\]
	whenever $\nu$ contains the two high degree vertices $u,v$, and $D_{0,b}(\nu)=\infty$ otherwise.  Note that if $\nu$ contains a cycle, then $D_{0,b}(\nu\sm\{w\})=\infty$ if $w\in \{u,v\}$, and otherwise if $D_{0,b}(\nu)> 1$ we have $D_{0,b}(\nu)\le 2\del^{-1} k^{-b/(b-1)}D_{0,b}(\nu\sm \{w\})$, where the factor of 2 comes from the ceiling function.  Thus $D_{0,b}$ satisfies the conditions of Lemma~\ref{lem:cycleReduction}, and using this we prove the following.
	\begin{lem}\label{lem:mb}
		There exists a constant $C>0$ such that if $k$ is sufficiently large in terms of $a,b,\del$ and if $\nu\sub V(\theta_{a,b})$ induces $e$ edges where $1\le e\le e(\theta_{a,b})-1$, then 
		\[D'_b(\nu)\le \frac{C k^{ab}n^2}{k n^{1+1/b}\big(\min\big\{k^{b/(b-1)},kn^{\frac{b-1}{b(ab-1)}}\big\}\big)^{e-1}}.\]
	\end{lem}
	\begin{proof}
		First, if $\nu$ induces a forest then the result follows from \Cref{lem:mForest}, so we may assume $\nu$ contains a cycle. 
		
		Now consider the case $D_b' (\nu) \leq  40\lceil\log n\rceil$. In particular, it suffices here to show that   
		\[40\lceil\log n\rceil \le \f{C k^{ab} n^2}{k n^{1+1/b} \cdot \big(kn^{\frac{b-1}{b(ab-1)}}\big)^{e-1}}.\]	
		And indeed, since $e\leq ab-1$, this inequality is satisfied if 
		\[40\lceil\log n\rceil \le C kn^{\frac{b-1}{b(ab-1)}},\]
		which holds for all $n$ and $k$, provided $C$ is sufficiently large.

		Thus we may assume that $40\lceil\log n\rceil\le D_b' (\nu)$; in particular, since $D'_b (\nu)\leq  20(D_{0,b}(\nu)+\lceil\log n\rceil)$, this implies that   $\lceil\log n\rceil\le D_{0,b}(\nu)$, and as such $D_b'(\nu)\le 40 D_{0,b}(\nu)$.  Possibly by adjusting the constant $C$, it now suffices to show $D_{0,b}(\nu)$ satisfies the inequality of the lemma.  By Lemmas~\ref{lem:mOne} and \ref{lem:cycleReduction}, it suffices to show this holds for $\nu$ consisting of $p\ge 2$ paths of length $b$.  In this case $|\nu|=p(b-1)+2$ and $e=pb$, so it suffices to show
		\[\f{k^{ab}n^2}{n^2 (k^{b/(b-1)})^{p(b-1)}\del^{|\nu|}}\le \frac{C k^{ab}n^2}{kn^{1+1/b}\big(\min\big\{k^{b/(b-1)},kn^{\frac{b-1}{b(ab-1)}}\big\}\big)^{pb-1}}\]
		for some constant $C$, where implicitly we used that the ceiling function in $D_{0,b}(\nu)$ can be ignored by increasing $C$ by a factor of 2.  Using $\min\big\{k^{b/(b-1)},kn^{\frac{b-1}{b(ab-1)}}\big\}\le kn^{\frac{b-1}{b(ab-1)}}$ and rearranging the above gives that it suffices to show
		\[\big(kn^{\frac{b-1}{b(ab-1)}}\big)^{pb-1}\le C \del^{|\nu|} k^{pb-1} n^{1-1/b},\]
		which holds for any $C\ge \del^{-|\nu|}$ since $pb-1\le ab-1$.  We conclude the result.
	\end{proof}
	
	It remains to deal with the case $t<b$.  To start, we recall that we write the paths of $\theta_{a,b}$ as $uw_1^j\cdots w_{b-1}^j v$ for $1\le j\le a$, and that we define $F_t=\{w_i^j:t\le i<b,\ i-t\tr{ is even}\}$.  We recall that if $u,v\in \nu$ then
	\[D_{0,t}(\nu)=\ceil{\f{k^{ab}n^2 }{ k^{(2b-2t+1)/(b-1)}n^{(2t-1)/b}\cdot (kn^{1/b})^f(k^{b/(b-1)})^{|\nu|-f-2}\delta^{|\nu|}}}\]
	where $f=|\nu\cap F_t|$, with $D_{0,t}(\nu)=\infty$ otherwise.  Similarly if $u,v,w_{b-1}^j\in \nu$ for some $j$ then
	\[D_{t}(\nu)=\ceil{\f{k^{ab}n^2 }{kn^{1+1/b}\cdot (kn^{1/b})^g(k^{b/(b-1)})^{|\nu|-g-3}\delta^{|\nu|}}},\]
	where $g=|\nu\cap F_t|$ if $b-t$ is even and $g=|\nu\cap F_t|-1$ otherwise, with $D_t(\nu)=\infty$ otherwise.  Note that both of these codegree functions satisfy the conditions of Lemma~\ref{lem:cycleReduction} since we assumed $n^{1-1/b}\ge k$. From now on we will assume we work with $t<b$ and define $f,g$ as in the above codegree functions. It will be useful to note that if $\nu$ consists of $p$ paths of length $b$, then by definition
	\begin{equation}f=p\ceil{(b-t)/2}\label{eq:f}\end{equation}
	and
	\begin{equation}g=(p-2)\ceil{(b-t)/2}+b-t,\label{eq:g}\end{equation}
	where this last equality follows from $g=f$ if $b-t$ is even and otherwise $g=f-1=(p-1)\ceil{(b-t)/2}+\floor{(b-t)/2}$.

	\begin{lem}\label{lem:D0t}
		Let $\nu$ consist of $p\ge 2$ complete paths and define $h=2t+f-b-2$.  There exists a constant $C>0$ such that 
		\[D_{0,t}(\nu)\le \frac{C k^{ab}n^2}{kn^{1+1/b}\big(\min\big\{k^{b/(b-1)},kn^{\frac{b-1}{b(ab-1)}}\big\}\big)^{e-1}}\]
		provided either
		\[k\le n^{\f{b-1}{b}\cdot \f{h}{(p-1)b+h}},\] or \[k^{b-1-h}\ge n^{\f{b-1}{b}\cdot \f{(b-1)(pb-1)-h(ab-1)}{ab-1}}.\]   
	\end{lem}
	
	\begin{proof}
		If $D_{0,t}(\nu)=1$ then the result holds by Lemma~\ref{lem:mOne}, so from now on we assume $D_{0,t}(\nu)>1$. We can rewrite the denominator of $D_{0,t}(\nu)$ as
		\begin{align*}\del^{|\nu|}k^{(2b-2t+1)/(b-1)}n^{(2t-1)/b}(kn^{1/b})^f(k^{b/(b-1)})^{|\nu|-f-2}=\del^{|\nu|}k^{2b/(b-1)}(k^{-1/(b-1)}n^{1/b})^{2t+f-1}(k^{b/(b-1)})^{|\nu|-2}\\=\del^{|\nu|}kn^{1+1/b}\cdot (k^{-1/(b-1)}n^{1/b})^{2t+f-b-2}(k^{b/(b-1)})^{|\nu|-2}.\end{align*}
		Using this and $D_{0,t}(\nu)>1$, we see that to show the desired result holds with $C=2 \del^{-v(\theta_{a,b})}$, it suffices to show 
		\begin{equation}(k^{-1/(b-1)}n^{1/b})^{h}(k^{b/(b-1)})^{|\nu|-2}\ge \big(\min\big\{k^{b/(b-1)},kn^{\frac{b-1}{b(ab-1)}}\big\}\big)^{e-1}.\label{eq:mt}\end{equation}
		Using that the minimum in \eqref{eq:mt} is at most $k^{b/(b-1)}$ and rearranging, we see that it suffices to have
		\[(k^{-1/(b-1)}n^{1/b})^{h}\ge 
		(k^{b/(b-1)})^{(e-1)-(|\nu|-2)}=(k^{b/(b-1)})^{pb-1-p(b-1)}=k^{(p-1)b/(b-1)},\] i.e. $k^{((p-1)b+h)/(b-1)}\le n^{h/b}$, which gives the first result.
		
		If we instead use that the minimum in \eqref{eq:mt} is at most $kn^{\frac{b-1}{b(ab-1)}}$, then we see that this inequality will be satisfied provided
		\begin{equation*}(k^{-1/(b-1)}n^{1/b})^hk^{pb}\ge k^{pb-1}n^{\f{(b-1)(pb-1)}{b(ab-1)}},\label{eq:D0t}\end{equation*}
		which is equivalent to
		\[k^{\f{b-1-h}{b-1}}\ge n^{\f{(b-1)(pb-1)}{b(ab-1)}-\f{h}{b}}=n^{\f{(b-1)(pb-1)-h(ab-1)}{b(ab-1)}}.\] This gives the last part of the lemma. 
	\end{proof} 
	
	\begin{lem}\label{lem:Dt}
		Let $\nu$ consist of $p\ge 2$ complete paths. There exists a constant $C>0$ such that 
		\[D_{t}(\nu)\le \frac{C k^{ab}n^2}{kn^{1+1/b}(k^{b/(b-1)})^{e-1}}\] provided 
		\[k\le n^{\f{b-1}{b}\cdot \f{g}{pb+g}}.\]
	\end{lem}
	\begin{proof}
		By Lemma~\ref{lem:mOne} we can assume $D_t(\nu)>1$, so the result holds for some $C\ge 2 \del^{-|\nu|}$ provided
		\[(kn^{1/b})^g(k^{b/(b-1)})^{|\nu|-g-3}\ge (k^{b/(b-1)})^{e-1},\]
		and rearranging this gives
		
		\[(k^{-1/(b-1)}n^{1/b})^{g}\ge (k^{b/(b-1)})^{e-1-|\nu|+3}=(k^{b/(b-1)})^{pb-1-p(b-1)+1}=k^{pb/(b-1)},\]
		and rearranging gives the desired result.
	\end{proof}
	
	With these two results we can solve the cycle case for $t<b$ provided $a$ is sufficiently large.
	
	\begin{lem}\label{lem:mt}
		If $b>t\ge 2$ and $a\ge 100,b\ge 3$, then there exists a constant $C>0$ such that if $k$ is sufficiently large in terms of $a,b,\del$ and if $\nu\sub V(\theta_{a,b})$ induces $e$ edges where $1\le e\le e(\theta_{a,b})-1$, then 
		\[D_{t}'(\nu)\le \frac{C k^{ab}n^2}{kn^{1+1/b}(k^{b/(b-1)})^{e-1}}.\]
	\end{lem}
	\begin{proof}
		By using similar reasoning as in Lemma~\ref{lem:mb}, it suffices to prove this upper bounds holds for $\min\{D_{0,t}(\nu),D_t(\nu)\}$ (i.e.\ ignoring the $\log n$ term in $D'_t(\nu)$) whenever $\nu$ consists of $p\ge 2$ paths of length $b$.  With $h$ as in Lemma~\ref{lem:D0t}, \eqref{eq:f} and \eqref{eq:g} give
		\[h=2t+p\ceil{(b-t)/2}-b-2\ge 0,\]
		\[g=(p-2)\ceil{(b-t)/2}+b-t,\]
		where $h\ge 0$ follows from $p,t\ge 2$.
		We also note that both inequalities of Lemma~\ref{lem:D0t} become  easier to satisfy for larger values of $h$  (this holds for the first inequality because the function $\f{h}{c+h}$ is increasing for any $c>0$, and it holds for the second since $n^{\f{b-1}{b}}\ge k$).
		
		We first claim that $h$ is relatively large in most cases; namely, if $p\ge 5$, then $h\ge b-1$.  Indeed, this being false is equivalent to
		\[2t+p\ceil{(b-t)/2}-b-2<b-1.\]
		If $t=b-1$ then this is equivalent to $p<2b+1-2(b-1)=3$, contradicting our assumption on $p$, so we may assume $t<b-1$.  By dropping the ceiling function, the inequality above implies
		\[2t+p(b-t)/2-b-2<b-1,\]
		which is equivalent to 
		\[p<4+\f{2}{b-t}\le 5,\]
		with the last step using $t<b-1$, again giving a contradiction to our assumption on $p$.  
		
		We conclude that $h\ge b-1$ if $p\ge 5$.  Note that the second inequality of Lemma~\ref{lem:D0t} trivially holds at $h=b-1$, and since the lemma is easier to satisfy for larger values of $h$, we conclude the result for $p\ge 5$. From now on\footnote{As an aside, it is not difficult to show that Lemma~\ref{lem:D0t} alone suffices to prove the result for $p\ge 3$ if, say, $a\ge 12$.   However, for $p=2$ it is necessary to use Lemma~\ref{lem:Dt} as well since, in particular, we can have $h=0$ in this case.  Dealing with the case $p=2$ here is the only reason the codegree functions $D_t$ are introduced.} we assume $2\le p\le 4$.

		Note that
		\[g+h=t+(2p-2)\ceil{(b-t)/2}-2,\]
		and in particular,
		\[\max\{g,h\}\ge t/2+(p-1)\ceil{(b-t)/2}-1\ge b/2-1\]
		(where the last step uses $p\geq 2$). First consider the case $h\ge b/2-1$. Note that for $p\le 4$ we have
		\[(b-1)(pb-1)-h(ab-1)\le (b-1)(4b-1)-(b/2-1)(ab-1)=(b/2)(8b-9-a(b-2))\le 0,\]
		where this last step holds for $a\ge 15\ge  \f{8b-9}{b-2}$ (and uses $b\geq 3$).  With this we either have $h\ge b-1$ (in which case we are done by the argument for $p\ge 5$), or
		\[k^{b-1-h}\ge 1\ge n^{\f{b-1}{b}\cdot \f{(b-1)(pb-1)-h(ab-1)}{ab-1}},\]
		in which case the result follows from Lemma~\ref{lem:D0t}.
		
		Now assume $h<b/2-1$, which in particular implies $g\ge b/2-1$.  By Lemma~\ref{lem:Dt}, and using $p\le 4$, we obtain the result if $k\le n^{\f{b-1}{b}\cdot \f{b/2-1}{4b+b/2-1}}=n^{\f{b-1}{b}\cdot \f{b-2}{9b-2}}$. On the other hand, using the second inequality of Lemma~\ref{lem:D0t}, which is harder to satisfy the smaller $h\ge 0$ is, we see that that for $p\le 4$ the result holds if
		\[k^{b-1}\ge n^{\f{b-1}{b}\cdot \f{(b-1)(4b-1)}{ab-1}}.\]
		Thus the result holds for all $k$ provided
		\[\f{4b-1}{ab-1}\le \f{b-2}{9b-2},\]
		or equivalently
		\[a\ge\frac{1}{b}+\f{(4b-1)(9b-2)}{b(b-2)}.\]
		This holds for $a\ge 100$, proving the result.
	\end{proof}
	
	\Cref{prop:codegrees} now follows immediately from Lemmas~ \ref{lem:mb} and \ref{lem:mt}.
	
	We note that sharper arguments can easily be used to reduce the bound $a\ge 100$ of \Cref{prop:codegrees} considerably, though the bound cannot be made arbitrarily small.  In particular, one can work out that the case $b=4$ and $p=t=2$ shows that $a\ge 9$ is needed, as $D_t'(\nu)$ does not satisfy the conclusion of \Cref{prop:codegrees} in this case.
	
	\section{Completing the Proof of \Cref{thm:main}}\label{sec:finish}
	Recall that we wish to show that for all $b\ge 2$, there exists $a_0=a_0(b)$ such that for any fixed $a\ge a_0$, w.h.p.
	\[\ex(G_{n,p},\theta_{a,b})=\begin{cases}
		\Theta\left(p^{\rec{b}}n^{1+\rec{b}}\right) & p\ge n^{-\f{b-1}{ab-1}}(\log n)^{2b},\\ 
		n^{2-\f{a(b-1)}{ab-1}}(\log n)^{O(1)} & n^{-\f{b-1}{ab-1}}(\log n)^{2b}\ge p\ge n^{-\f{a(b-1)}{ab-1}},\\ 
		(1+o(1))p{n\choose 2} &  n^{-\f{a(b-1)}{ab-1}}\gg p\gg n^{-2}.
	\end{cases}\]
	The case $b=2$ follows from Morris and Saxton~\cite{morris2016number}, so from now on we assume $b\ge 3$.  The lower bounds for $\ex(G_{n,p},\theta_{a,b})$ follow\footnote{Specifically, one applies Corollary 5.1 to the rooted tree $(T,R)$ with $T$ the path on $b$ edges and $R$ its set of leaves.  With this one can check $\rho(T)\ge \frac{b}{b-1}$ (which  is also implicitly shown in Conlon~\cite{conlon2019graphs}), and that $\theta_{a,b}\in \c{T}^a$} from \cite[Corollary 5.1]{spiro2022random}, which is proven using random polynomial graphs (similar to how Conlon~\cite{conlon2019graphs} proved $\ex(n,\theta_{a,b})=\Om(n^{1+1/b})$ whenever $a$ is sufficiently large in terms of $b$).  The upper bound for $p$ small follows from the fact that $G_{n,p}$ has at most $(1+o(1))p{n\choose 2}$ edges w.h.p., and the upper bound for $p$ in the middle range will follow from the upper bound for $p$ large due to the monotonicity of $\ex(G_{n,p},F)$ with respect to $p$.
	
	With this all in mind, it only remains to prove $\ex(G_{n,p},\theta_{a,b})=O(p^{1/b}n^{1+1/b})$ when $p\ge n^{-\f{b-1}{ab-1}}(\log n)^{2b}$.  For this we utilize the following general result showing that balanced supersaturation implies upper bounds on $\ex(G_{n,p},F)$.
	
	\begin{thm}\label{thm:containers}
		Let $F$ be a graph and $1<\al<2$ a real number satisfying the following: there exist real numbers $C,k_0>0$ such that for every $n$-vertex graph $G$ with $e(G)= k n^{\al}$ and $k\ge k_0$, there exists a hypergraph $\cH$ on $E(G)$ whose hyperedges are copies of $F$ and is such that $|\cH|\ge C^{-1} k^{e(F)}n^{v(F)-(2-\al)e(F)}$, and such that for every $\sig\sub E(G)$ with $1\le |\sig|\le e(F)-1$, we have
		\[\deg_{\cH}(\sig)\le \frac{C k^{e(F)}n^{v(F)-(2-\al)e(F)}}{k n^\al \left(\min\left\{k^{\frac{1}{2-\al}},kn^{\al-2+\frac{v(F)-2}{e(F)-1}}\right\}\right)^{|\sig|-1}}.\]
		In this case, 
		\[\ex(G_{n,p},F)=O(p^{\al-1}n^\al)\hspace{1.7em} \text{for all \ } p\ge \left(n^{2-\alpha-\frac{v(F)-2}{e(F)-1}}/\log^2n\right)^\frac{1}{\alpha-1}.\]
	\end{thm}
	
	We note that if $F$ is 2-balanced, i.e.\ if it has $m_2(F)=\frac{e(F)-1}{v(F)-2}$, then the conclusion of \Cref{thm:containers} is exactly the upper bound predicted by \Cref{conj:main} provided $\ex(n,F)=\Theta(n^\al)$.
	
	The proof of \Cref{thm:containers} uses what is by now a fairly routine argument involving hypergraph containers, which is a powerful technique developed recently and independently by Balogh, Morris and Samotij~\cite{balogh2015independent} and Saxton and Thomason~\cite{saxton2015hypergraph}.  We defer the details to \Cref{append:Containers}.
	
	In any case, by \Cref{cor:balancedEdges}, we see that $\theta_{a,b}$ satisfies the conditions of \Cref{thm:containers} for $a\ge 100$ and $b\ge 3$, proving the desired upper bound and completing the proof.
	
	\section{Concluding Remarks}\label{sec:conclusion}
	In this paper we established upper bounds for $\ex(G_{n,p},\theta_{a,b})$ which are essentially tight whenever $a$ is sufficiently large in terms of $b$.  It would be of interest if one could extend our ideas to prove effective upper bounds on $\ex(G_{n,p},F)$ for other $F$.  In particular, one might hope to prove upper bounds for powers of rooted trees.
	
	More precisely, given a tree $T$, a set $R\sub V(T)$, and an integer $a$, we define $T_R^a$ to be the graph consisting of $a$ copies of $T$ which agree only on the set $R$.  For example, if $T$ is a path of length $b$ and $R$ consists of its two endpoints, then $T_R^a=\theta_{a,b}$, and if $T=K_{s,1}$ and $R$ is its set of leaves, then $T_R^a=K_{s,a}$.  In particular, the only bipartite graphs for which we know tight bounds for $\ex(G_{n,p},F)$, namely theta graphs and complete bipartite graphs, are examples of powers of trees.  
	\begin{quest}\label{quest:RootedTrees}
		Can one prove essentially tight bounds on $\ex(G_{n,p},T_R^a)$ for other powers of rooted trees?
	\end{quest}
	The best upper bounds for this problem come from the general bounds of Jiang and Longbreak~\cite{jiang2022balanced}, and the best lower bound comes from \cite{spiro2022random}.  We note that analogous to the situation for theta graphs prior to this paper, the lower bound of \cite{spiro2022random} depends only on the tree $T$ while the upper bound of \cite{jiang2022balanced} depends on $a$, and as such the gaps between these bounds grow large as $a$ increases.  Similar to the situation in the present paper, we suspect that the lower bound is closer to the truth, and in particular, Conjecture~\ref{conj:main} claims that in many cases, the lower bound from \cite{spiro2022random} should be the correct answer.
	
	Solving Question~\ref{quest:RootedTrees} for all rooted trees is likely impossible.  Indeed, even the $p=1$ case, namely that of determining the Tur\'an number $\ex(n,T_R^a)$, is an important open problem of Bukh and Conlon~\cite{bukh2018rational} related to the rational exponents conjecture.  That being said, there are a number of special cases where this Tur\'an number is known \cite{conlon2022rational,conlon2021more,janzer2020extremal,jiang2020negligible,kang2021rational}, and it might be possible to generalize our ideas to deal with some of these cases in the random setting.  A more detailed discussion on this problem can be found in the concluding remarks of \cite{spiro2022random}.

	To prove \Cref{thm:main}, we first proved a balanced supersaturation result, \Cref{cor:balancedEdges}, which is essentially optimal for $a\ge 100$.  It would be desirable to do this for all $a$.
	
	\begin{quest}\label{quest:allA}
		Can one extend \Cref{cor:balancedEdges} to hold for all $a\ge 3$?
	\end{quest}
	Note that the $a=2$ case is already dealt with by Morris and Saxton~\cite{morris2016number}.  Solving this question, in addition to being desirable from a philosophical standpoint,  might lead to a simpler proof of \Cref{cor:balancedEdges} which could more easily generalize to solving Question~\ref{quest:RootedTrees}.  The simplest way to resolve this question would be to resolve the following.
	
	\begin{quest}\label{quest:largeX}
		Can one extend Proposition~\ref{prop:expansion} to hold with $|X|\ge \ep m$?
	\end{quest}
	An affirmative answer here would not only give an affirmative answer to Question~\ref{quest:allA}, but also would allow one to avoid many of the messy technical details in our proof.  Namely, with this one can alter the definition of $D_{s,t}$ in such a way that the $D_t$ functions are no longer needed, and such that the computations for proving \Cref{prop:codegrees} are much simpler.
	
	\textbf{Acknowledgments}.  We thank Rob Morris for useful comments about the presentation of this paper.

	\printbibliography
	
	\appendix

	\section{Proof of \Cref{thm:containers}}\label{append:Containers}
	Throughout this section, we say that a graph $F$ is \textit{$\al$-good} with $1<\al<2$ a real number if it satisfies the following balanced supersaturation condition: there exist real numbers $C,k_0>0$ such that for every $n$-vertex graph $G$ with $e(G)= k n^{\al}$ and $k\ge k_0$, there exists a hypergraph $\cH$ on $E(G)$ whose hyperedges are copies of $F$ and is such that $|\cH|\ge C^{-1} k^{e(F)}n^{v(F)-(2-\al)e(F)}$, and such that for every $\sig\sub E(G)$ with $1\le |\sig|\le e(F)-1$, we have
	\[\deg_{\cH}(\sig)\le \frac{C k^{e(F)}n^{v(F)-(2-\al)e(F)}}{k n^\al \left(\min\left\{k^{\frac{1}{2-\al}},kn^{\al-2+\frac{v(F)-2}{e(F)-1}}\right\}\right)^{|\sig|-1}}.\]
	Here we prove \Cref{thm:containers}, i.e.\  that if $F$ is $\al$-good, then  $\ex(G_{n,p},F)=O(p^{\al-1}n^\al)$ w.h.p.\ for all $p\ge \left(n^{\alpha-2+\frac{v(F)-2}{e(F)-1}}/\log^2n\right)^\frac{-1}{\alpha-1}$.  We emphasize that our proof is nearly word-for-word the same as that of Morris and Saxton \cite{morris2016number}.  We make use the following definition from~\cite{saxton2015hypergraph}.
	
	\begin{defn}\label{def:tau}
		Given an $r$-uniform hypergraph $\HH$ and a real number $\tau$, define
		$$\delta(\HH,\tau) \, = \, \frac{1}{e(\HH)} \,\sum_{j=2}^r \,\frac{1}{\tau^{j-1}} \sum_{v \in V(\HH)} d^{(j)}(v),$$
		where
		$$d^{(j)}(v) \, = \, \max\big\{ \deg_\HH(\sigma) \, : \, v \in \sigma \subseteq V(\HH) \textup{ and }|\sigma| = j \big\}$$
		denotes the maximum degree in $\HH$ of a $j$-set containing $v$. 
	\end{defn}
	
	We remark that we have removed some extraneous constants from the definition in~\cite{saxton2015hypergraph}, since these do not affect the formulation of the theorem below.  We also note that $\delta$ is typically called a \textit{codegree function}, but we emphasize that this has no relation to the definition of codegree functions that we used throughout our paper.

	The following container theorem was proved by Balogh, Morris and Samotij~\cite[Proposition~3.1]{balogh2015independent} and by Saxton and Thomason~\cite[Theorem~6.2]{saxton2015hypergraph}\footnote{To be precise, Theorem~6.2 in~\cite{saxton2015hypergraph} is stated where~$T$ is a tuple of vertex sets rather than a single vertex set, but it is straightforward to deduce this form from the methods of~\cite{saxton2015hypergraph}.}, where here the notation $S^{(\le t)}$ denotes the collection of all subsets of $S$ of size at most $t$.
	
	\begin{thm}\label{thm:containers:turan}
		Let $r \ge 2$ and let $0 < \delta < \delta_0(r)$ be sufficiently small. Let $\HH$ be an $r$-graph with~$N$ vertices, and suppose that $\delta(\HH,\tau) \le \delta$ for some $\tau > 0$. Then there exists a collection $\C$ of subsets of $V(\HH)$, and a function $f \colon V(\HH)^{(\le \tau N / \delta)} \to \C$ such that:
		\begin{itemize}
			\item[$(a)$] For every independent set $I$, there exists $T \subset I$ with $|T| \le \tau N / \delta$ and $I \subset f(T)$, and\smallskip
			\item[$(b)$] $e\big( \HH[C] \big) \le \big(1 - \delta \big) e(\HH)$ for every $C \in \C$.
		\end{itemize}
	\end{thm}
	
	We will refer to the collection $\C$ as the \emph{containers} of $\HH$, since, by~$(a)$, every independent set is contained in some member of $\C$. The reader should think of $V(\HH)$ as being the \emph{edge} set of some underlying graph $G$, and $E(\HH)$ as encoding (some subset of) the copies of a graph $F$ in $G$. Thus every $F$-free subgraph of $G$ is an independent set of $\HH$. 
	
	Let us introduce some notation to simplify the statements which follow. Given a graph $F$ and real number $\al$, let $\I = \I(n)$ denote the collection of $F$-free graphs with $n$ vertices, and let $\G = \G(n,k)$ denote the collection of all graphs with $n$ vertices and at most $kn^\alpha$ edges. By a \emph{colored graph}, we mean a graph together with an arbitrary labelled partition of its edge set. 
	
	\begin{thm}\label{thm:cycle:containers:turan:refined}
		If $F$ is $\al$-good, then there exists a constant $C = C(F)$ such that the following holds for all sufficiently large $n,k \in \N$ with $k \le \left(n^{\alpha-2+\frac{v(F)-2}{e(F)-1}}/\log^2n\right)^\frac{2-\alpha}{\alpha-1}$. There exists a collection $\S$ of colored graphs with $n$ vertices and at most $Ck^{-\frac{\alpha-1}{2-\alpha}}\cdot n^\alpha$ edges, as well as functions 
		$$g \colon \I \to \S \qquad \text{and} \qquad h \colon \S \to \G(n,k)$$
		with the following properties:
		\begin{itemize}
			\item[$(a)$] 
			For every $s \ge 0$, the number of colored graphs in $\S$ with $s$ edges is at most
			$$\bigg( \frac{C n^\alpha}{s} \bigg)^{\frac{1}{\alpha-1}\cdot s}  \cdot \exp\Big( Ck^{-\frac{\alpha-1}{2-\alpha}}\cdot n^\alpha \Big).$$
			\item[$(b)$] $g(I) \subset I \subset g(I) \cup h(g(I))$ for every $I \in \I$.
		\end{itemize}
	\end{thm}
	
	We prove Theorem~\ref{thm:cycle:containers:turan:refined} by iterating the following container result.
	
	\begin{prop} \label{prop:refined_containers_for_graph}
		If $F$ is an $\al$-good graph, then there exist $k_0 \in \N$ and $\eps > 0$ such that the following holds for every $k \ge k_0$ and every $n \in \N$. Set
		\begin{equation}\label{def:mu}
			\mu \, = \, \frac{1}{\eps} \cdot \max\Big\{ k^{-\frac{\alpha-1}{2-\alpha}}, \, n^{-\left(\alpha-2+\frac{v(F)-2}{e(F)-1}\right)} \Big\}.
		\end{equation}
		Given a graph~$G$ with~$n$ vertices and $kn^\alpha$ edges, there exists a function $f_G$ that maps subgraphs of $G$ to subgraphs of $G$, such that, for every $F$-free subgraph~$I \subset G$, \begin{itemize}
			\item[$(a)$] There exists a subgraph $T=T(I) \subset I$ with $e(T) \le \mu n^\alpha$ and $I \subset f_G(T)$, and\smallskip
			\item[$(b)$] $e\big( f_G(T(I)) \big) \le (1 - \eps) e(G)$.
		\end{itemize}
	\end{prop}

	\begin{proof}
		By definition of $F$ being $\al$-good, there exist real numbers $C,k_0>0$ and a hypergraph $\cH$ on $E(G)$ whose hyperedges are copies of $F$ and is such that 
		\begin{itemize}
			\item[(i)] $|\cH|\ge C^{-1} k^{e(F)}n^{v(F)-(2-\al)e(F)}$, and
			\item[(ii)] For every $\sig\sub E(G)$ with $1\le |\sig|\le e(F)-1$, we have
			\[\deg_{\cH}(\sig)\le \frac{C k^{e(F)}n^{v(F)-(2-\al)e(F)}}{k n^\al \left(\min\left\{k^{\frac{1}{2-\al}},kn^{\al-2+\frac{v(F)-2}{e(F)-1}}\right\}\right)^{|\sig|-1}}.\]
		\end{itemize} 
		Let $\del:=C^{-1}$, and without loss of generality we can assume $C$ is sufficiently large so that Theorem~\ref{thm:containers:turan} holds with $r = e(F)$ and this choice of $\del$. 
		We will show that if 
		$$\frac{1}{\tau} = \delta^4 k \cdot \min\Big\{ k^{\frac{\alpha-1}{2-\alpha}}, \, n^{\alpha-2+\frac{v(F)-2}{e(F)-1}}\Big\} = \delta^4 \cdot \min\Big\{ k^{\frac{1}{2-\alpha}}, \, kn^{\alpha-2+\frac{v(F)-2}{e(F)-1}}\Big\},$$ 
		then it follows from~$(i)$ and~$(ii)$ that $\delta(\HH,\tau) \le \delta$. Indeed, since $v(\HH) = e(G) = kn^\alpha$, we have
		{\footnotesize \begin{align*}
				\delta(\HH,\tau) & \, = \, \frac{1}{e(\HH)} \,\sum_{j=2}^{e(F)} \,\frac{1}{\tau^{j-1}} \cdot \sum_{v \in V(\HH)} d^{(j)}(v) \\
				& \, \le \, \frac{v(\HH)}{e(\HH)} \Bigg[ \sum_{j=2}^{e(F)} \frac{1}{\tau^{j-1}} \cdot \frac{C k^{e(F)}n^{v(F)-(2-\al)e(F)}}{k n^\al \left(\min\left\{k^{\frac{1}{2-\al}},kn^{\al-2+\frac{v(F)-2}{e(F)-1}}\right\}\right)^{j-1}} \Bigg]\\
				& \, \le \, \frac{kn^\al}{C^{-1} k^{e(F)}n^{v(F)-(2-\al)e(F)}} \Bigg[ \sum_{j=2}^{e(F)} \left(\delta^4 \cdot \min\Big\{ k^{\frac{1}{2-\alpha}}, \, kn^{\alpha-2+\frac{v(F)-2}{e(F)-1}}\Big\}\right)^{j-1} \cdot \frac{C k^{e(F)}n^{v(F)-(2-\al)e(F)}}{k n^\al \left(\min\left\{k^{\frac{1}{2-\al}},kn^{\al-2+\frac{v(F)-2}{e(F)-1}}\right\}\right)^{j-1}} \Bigg]\\				
				& \, \le \, \sum_{j=2}^{e(F)} C^2\delta^{4(j-1)} \, \le \, \delta,
		\end{align*}}
		where this last bound holds provided $\del=C^{-1}$ is sufficiently small, which we can assume to be the case without loss of generality.
		Thus, applying Theorem~\ref{thm:containers:turan} and setting $\eps = \delta^5$, we obtain a collection~$\C$ of subgraphs of $G$ and a function $f_G$ mapping subgraphs of $G$ to elements of $\C$ so that for every $F$-free subgraph $I\subset G$, there exists a subgraph $T = T(I)\subset I$ with
		
		\[
		e(T)\ \leq\ \tau N /\delta \ \leq  \frac{n^\al}{\eps}  \cdot \max\Big\{ k^{-\frac{\alpha-1}{2-\alpha}}, \, n^{-\left(\alpha-2+\frac{v(F)-2}{e(F)-1}\right)} \Big\}\ =\ \mu n^\alpha
		\]
		
		and $I\subset f_G(T)$, and also 
		
		\begin{equation}\label{eq:yourAdHere1}
			e\big( \HH[C] \big) \le \big(1 - \delta \big) e(\HH) \text{ for all } C\in\C.
		\end{equation}

		It only remains to show that this second condition implies $e(C) \le (1-\eps) e(G)$ for every $C \in \C$ (notice that the first inequality is about hyperedges and the second is about graph edges). To prove this, for each $C \in \C$ set 
		\[
		\D(C) \, = \, E(\HH) \setminus E(\HH[C]) \, = \, \big\{ e \in E(\HH) \,:\, v \in e \mbox{ for some } v \in V(\HH) \setminus C \big\},
		\]
		and recall that $\deg_\HH(v) \le e(\HH) / \big( \delta^2 kn^\alpha \big)$ for every $v \in V(\HH)$, by~$(i)$, ~$(ii)$, and $\del=C^{-1}$. Therefore,
		\[
		|\D(C)| \, \le \, \frac{e(\HH)}{\delta^2 kn^\alpha} \cdot |E(G) \setminus C|.
		\]
		On the other hand, we have $|\D(C)| = e(\HH) - e(\HH[C]) \ge \delta e(\HH)$ by condition~\eqref{eq:yourAdHere1}, and so
		$$ |E(G) \setminus C| \, \ge \, \delta^3 kn^\alpha\ge \ep e(G),$$
		as required. Hence the proposition follows with $\eps = \delta^5$.
	\end{proof}

	With this in hand, we continue on towards the proof of  Theorem~\ref{thm:cycle:containers:turan:refined}. We will need the following straightforward lemma (see, for example,~\cite[Lemma~4.3]{CM}).

	\begin{lem}\label{obs:optimize}
		Let $M > 0$, $s > 0$ and $0 < \delta < 1$. If $a_1,\ldots,a_m \in \RR$ satisfy $s = \sum_j a_j$ and $1 \le a_j \le (1-\delta)^j M$ for each $j \in [m]$, then
		$$s \log s \, \le \, \sum_{j=1}^m a_j \log a_j + O(M).$$
	\end{lem} 
	
	We can now deduce Theorem~\ref{thm:cycle:containers:turan:refined}. 
	
	\begin{proof}[Proof of Theorem~\ref{thm:cycle:containers:turan:refined}]
		We construct the functions $g$ and $h$ and the family $\S$ as follows. Given an $F$-free graph $I \in \I$, we repeatedly apply Proposition~\ref{prop:refined_containers_for_graph}, first to the complete graph $G_0 = K_n$, then to the graph $G_1 = f_{G_0}(T_1) \setminus T_1$, where $T_1 \subset I$ is the set guaranteed to exist by part~$(a)$, then to the graph $G_2 = f_{G_1}(T_2) \setminus T_2$, where $T_2 \subset I \cap G_1 = I \setminus T_1$, and so on. We continue until we arrive at a graph $G_m$ with at most $kn^\al$ edges, and set 
		$$g(I) = (T_1, \ldots, T_m) \qquad \text{and} \qquad h\big( g(I) \big) = G_m.$$ 
		Since $G_m$ depends only on the sequence $(T_1,\ldots,T_m)$, the function $h$ is well-defined.
		
		It remains to bound the number of colored graphs in $\S$ with $s$ edges. To do so, it suffices to count the number of choices for the sequence of graphs $(T_1,\ldots,T_m)$ with $\sum_j e(T_j) = s$. For each $j \ge 1$, define $k(j)$ and $\mu(j)$ as follows:
		$$e\big( G_{m-j} \big) = k(j) n^\al \quad \text{and} \quad \mu(j) = \frac{1}{\eps} \cdot \max\Big\{ k(j)^{-\frac{\alpha-1}{2-\alpha}}, \, n^{-\left(\alpha-2+\frac{v(F)-2}{e(F)-1}\right)} \Big\},$$
		and note that
		$$k(j) \ge (1 - \eps)^{-j+1} k, \qquad T_{j+1} \subset G_j \qquad \text{and} \qquad e(T_{m-j}) \le \mu(j) n^\al.$$
		Thus, fixing $k$, $\eps$ and $s$ as above, and writing
		$$\K(m) \, = \, \Big\{ \k = (k(1),\ldots,k(m)) \, : \, (1 - \eps)^{-j+1} k \le k(j) \le n^{2-\al} \Big\}$$
		for each $m \in \N$, and
		$$\A(\k) \, = \, \Big\{ \a = (a(1),\ldots,a(m)) \, : \, a(j) \le \mu(j) n^\al \text{ and } \sum_j a(j) = s \Big\},$$
		for each $\k \in \K(m)$, it follows that the number of colored graphs in $\S$ with $s$ edges is at most 
		$$\sum_{m = 1}^\infty \sum_{\k \in \K(m)} \sum_{\a \in \A(\k)} \prod_{j=1}^m { k(j) n^\al \choose a(j)}.$$
		Given $m \in \N$, $\k \in \K(m)$ and $\a \in \A(\k)$, let us partition the product over $j$ according to whether or not $\mu(j) = \frac{1}{\eps} \cdot n^{-\left(\alpha-2+\frac{v(F)-2}{e(F)-1}\right)}$. Since $\K(m) = \emptyset$ if $m$ is at least some large constant times $\log n$,  the product of the terms for which this is the case is at most
		$$\big( n^2 \big)^{\sum_j a(j)} \, \le \, \exp\Big( O(1) \cdot n^{2-\frac{v(F)-2}{e(F)-1}} (\log n)^2 \Big) \, \le \, \exp\Big( O(1) \cdot k^{-\frac{\alpha-1}{2-\alpha}} n^\al \Big),$$
		where in the last step we used the fact that $k \le \left(n^{\alpha-2+\frac{v(F)-2}{e(F)-1}}/\log^2n\right)^\frac{2-\alpha}{\alpha-1}$. On the other hand, if $a(j) \le k(j)^{-\frac{\alpha-1}{2-\alpha}} n^\al$, i.e.\ if $k(j)\le (n^\al/a(j))^{\f{2-\al}{\al-1}}$, then  
		$${ k(j) n^\al \choose a(j)} \, \le \, \bigg( \frac{e k(j) n^\al}{a(j)} \bigg)^{a(j)} \, \le \, \bigg( \frac{en^\al}{  a(j)} \bigg)^{\frac{1}{\al-1} \cdot a(j)},$$
		and hence, by Lemma~\ref{obs:optimize}, the product over the remaining $j$ is at most 
		$$\bigg( \frac{C' n^\al}{s} \bigg)^{\frac{1}{\al-1} \cdot s} \cdot \exp\Big( C'k^{-\frac{\alpha-1}{2-\alpha}} n^\al \Big)$$
		for some $C' = C'(F)$. Noting that $\sum_{m = 1}^\infty \sum_{\k \in \K(m)} |\A(\k)| = n^{O(\log n)}$ since $\K(m)=\emptyset$ for $m$ at least some large constant $\log n$, the theorem follows.
	\end{proof}
	
	We can now easily deduce Theorem~\ref{thm:containers}.
	
	\begin{proof}[Proof of Theorem~\ref{thm:containers}] 
		Let $F$ be a graph satisfying the hypotheses of the theorem, i.e.\ a graph which is $\al$-good for some $1<\al<2$.  Recall that we wish to show that for $p \ge  \left(n^{2-\al-\frac{v(F)-2}{e(F)-1}}/\log^2n\right)^\frac{1}{\alpha-1}$, we have $\ex(G_{n,p},F)=O(p^{\al-1}n^\al)$ w.h.p.  Given such a function $p = p(n)$, define $k = p^{-(2-\al)}$.  Since $k \le \left(Cn^{\alpha-2+\frac{v(F)-2}{e(F)-1}}/\log^2n\right)^\frac{2-\alpha}{\alpha-1}$, we can apply \Cref{thm:cycle:containers:turan:refined} to get functions $g,h$. Suppose that there exists an $F$-free subgraph $I \subset G(n,p)$ with $m$ edges, and observe that $g(I) \subset G(n,p)$, and that $G(n,p)$ contains at least $m - e\big( g(I) \big)$ elements of $h(g(I))$. The probability of this event is therefore at most 
		\begin{align*}
			\sum_{S \in \S} {kn^\al \choose m - e(S)} p^m & \, \le \sum_{s = 0}^{C k^{-\frac{\alpha-1}{2-\alpha}} n^\al} \bigg( \frac{C p^{\al-1} n^\al}{s} \bigg)^{\frac{1}{\al-1} \cdot s} \exp\Big( C k^{-\frac{\alpha-1}{2-\alpha}} n^\al \Big) \bigg( \frac{3pkn^\al}{m - s} \bigg)^{m - s} \nonumber\\
			& \, \le \, \exp\bigg[ O(1) \cdot \Big( p^{\al-1} n^\al + k^{-\frac{\alpha-1}{2-\alpha}} n^\al \Big) \bigg] \bigg( \frac{4pkn^\al}{m} \bigg)^{m/2} \to 0\label{eq:finalcounttozero}
		\end{align*}
		as $n \to \infty$, as long as $m$ is a sufficiently large constant times 
		$$\max\Big\{ pk  n^\al , \, k^{-\frac{\alpha-1}{2-\alpha}} n^\al \Big\}=p^{\al-1}n^\al.$$
		We conclude that $\ex(G_{n,p},F)=O(p^{\al-1}n^\al)$ w.h.p., giving the result.
	\end{proof}

	\section{Proof of Proposition~\ref{prop:expansion}}\label{append:Expansion}

	We emphasize that almost everything in this section will be nearly identical to Morris and Saxton~\cite{morris2016number}.  We first recall the definitions and conventions introduced in Section~\ref{sec:expansion}:
	
	\begin{itemize}
		\item We fixed a sequence of rapidly decreasing constants \[1\ge \ep_b\ge \cdots \ge \ep_2\ge \ep_1>0\] which depend only on $b$.  We also fixed some $m$-vertex graph $G$ with minimum degree $\ell m^{1/b}$ with $\ell$ (and hence $m$) sufficiently large in terms of the $\ep_t$ constants.
		
		\item For $x\in V(G)$, we say that a tuple $\c{A}=(A_0,A_1,\ldots,A_t)$ of (not necessarily disjoint) subsets of $V(G)$ is a \textit{concentrated $t$-neighborhood of $x$} if $A_0=\{x\}$, $|N(v)\cap A_i|\ge \ep_t \ell m^{1/b}$ for all $v\in A_{i-1}$, and
		\[|A_t|\le \ell^{(b-t)/(b-1)}m^{t/b}.\]
		We define $t(x)$ to be the minimal $t\ge 2$ such that there exists a concentrated $t$-neighborhood of $x$ in $G$.
		
		\item Lemma~\ref{lem:XSet} says that for some $2\le t\le b$, there exists $X\sub V(G)$ of size at least $\half (4b)^{t-b}\ell^{(b-t)/(b-1)}m^{t/b}$ such that $t(x)=t$ for every $x\in X$, and such that for every $x\in X$ there exists a tuple of sets $\c{A}=(A_0,\ldots,A_t)$ such that $A_0=\{x\}$, $|A_t|\le \ell^{(b-t)/(b-1)}m^{t/b}$, $|N(y)\cap A_i|\ge \half \ep_t\ell m^{1/b}$ for all $y\in A_{i-1}$,  and every $y\in \bigcup A_i$ has $t(y)\ge t$. 
	\end{itemize}
	For the rest of this section we fix $t,X$ as in Lemma~\ref{lem:XSet}.  We also fix some $\ep>0$ sufficiently small compared to the $\ep_s$ constants, as well as a set of forests $\c{F}$ such that for every path $x_1\cdots x_r$ of $G$ which does not contain an element of $\c{F}$ as a subgraph, there are at most $\ep \ell m^{1/b}$ vertices $x_{r+1}\in N_G(x_r)$ such that the path $x_1\cdots x_{r+1}$ contains an element of $\c{F}$ as a subgraph.  As much as possible we use the notation of Morris and Saxton's original proof, and in particular, we drop our convention from the main part of the text that $u,v,w$ are used only as vertices of $\theta_{a,b}$.

	We introduce some notation and definitions that will be used for the rest of the proof.  Given a set of paths $\c{P}$, we define the \textit{$(r,v)$-branching factor} of $\c{P}$ is the maximum number $d$ such that there exist $d$ paths in $\c{P}$ with $i$th vertex $v$ and pairwise distinct $(i+1)$st vertices.  The \textit{branching factor} of $\c{P}$ is defined to be the maximum $(i,v)$-branching factor amongst all choices of $i,v$.  We define $\c{P}_{i,j}=\{u_i\cdots u_j:u_0\cdots u_t\in \c{P}\}$ and $\c{P}[u\to v]:=\{x_1\cdots x_s\in \c{P}:x_1=u,\ x_s=v\}$,  and also define $\c{P}[u\to S]=\bigcup_{v\in S} \c{P}[u\to v]$.
	
	One lemma that we will need in several places is the following.
	\begin{lem}\label{lem:RPaths}
		Let $\c{R}$ be a collection of paths of length $s\ge 2$ in $G$ from a vertex $x\in V(G)$ to a set $B\sub V(G)$.  Assume that $|B|\le \ell^{(b-s)/(b-1)}m^{s/b}$, $|\c{R}|> s \ep_s (\ell m^{1/b})^s$, and that $\c{R}$ has branching factor at most $\ell m^{1/b}$.  Then $t(x)\le s$,
	\end{lem}
	\begin{proof}
		Form a subset $\c{R}'\sub \c{R}$ by starting with $\c{R}'=\c{R}$ and then iteratively choosing $i,v$ such that the $(i,v)$-branching factor of $\c{R}'$ is less than $\ep_s\ell m^{1/b}$ and then deleting any paths which contain $v$ as their $i$th vertex.  Let $A_i$ be the set of vertices used as the $i$th vertex of some path of $\c{R}'$.  If $\c{R}'$ is non-empty, then $(A_0,\ldots,A_s)$ is a concentrated $s$-neighborhood of $x$ by construction, which shows $t(x)\le s$.
		
		Thus it suffices to show that $\c{R}'$ is non-empty.  We claim that the number of paths that were destroyed is at most $s\cdot \ep_s(\ell m^{1/b})^s$.  Indeed, because $\c{R}$ has branching factor at most $\ell m^{1/b}$, every destroyed path can be identified by choosing some index $0\le i<s$, starting a path at $u_0=x$, and then iteratively choosing the next vertex of the path $u_{j+1}$ in at most $\ell m^{1/b}$ ways for each $j\ne i$ and in at most $\ep_s\ell m^{1/b}$ ways when $j=i$, proving the claim.  Since $|\c{R}|$ is strictly greater than the number of destroyed paths, $\c{R}'$ is non-empty and the result follows.
	\end{proof}
	
	The following definition will almost be strong enough to prove Proposition~\ref{prop:expansion}.
	\begin{defn}\label{def:balanced}
		Let $\c{A}=(A_0,\ldots,A_t)$ be a collection of (not necessarily disjoint) Sets of vertices of $G$ with $A_0=\{x\}$ and let $\c{P}$ be a collection of paths of the form $xu_1\cdots u_t$ with $u_i\in A_i$ for all $i$.  We say that $(\c{A},\c{P})$ is a \textit{balanced $t$-neighborhood} of $x$ if the following conditions hold:
		\begin{enumerate}
			\item[(i)] We have $|A_1|\le \ell m^{1/b}$ and $|A_t|\le \ell^{(b-t)/(b-1)}m^{t/b}$.
			\item[(ii)] For every $0\le i<j\le t$ with $(i,j)\ne (0,t)$ and every $u\in A_i,v\in A_j$, we have $|\c{P}_{i,j}[u,v]|\le \ell^{(j-i-1)b/(b-1)}$.
			\item[(iii)] The branching factor of $\c{P}$ is at most $\ep_t\ell m^{1/b}$.
		\end{enumerate}
	\end{defn}
	For the next lemma we recall that $\c{F}$ is a set of forests satisfying a property that depends on $\ep$.
	
	\begin{lem}\label{lem:balanced}
		If $x\in X$, then there exists a balanced $t$-neighborhood $(\c{A},\c{P})$ of $x$ with $|\c{P}|\ge \half(\quart \ep_t \ell m^{1/b})^t$ such that every subgraph of each $P\in \c{P}$ does not lie in $\c{F}$ provided $\ep$ is sufficiently small.
	\end{lem}
	
	\begin{proof}

		Let $\c{A}$ be the tuple of sets guaranteed by Lemma~\ref{lem:XSet}.  We may assume that $|A_1|\le \ell m^{1/b}$, as otherwise we can just remove vertices from $A_1$ while maintaining all the properties guaranteed by Lemma~\ref{lem:XSet}.  For each $v\in A_{i-1}$, let $Q_i(v)$ be an arbitrary subset of $N(v)\cap A_i$ of size $\half \ep_t\ell m^{1/b}$.  Let $\c{Q}$ be the set of paths $xu_1\cdots u_t$ generated as follows.  Given $xu_1\cdots u_{i-1}$, select any $u_i\in \c{Q}_i(u_{i-1})$ such that $u_i\notin \{x,u_1,\ldots,u_{i-1}\}$ and such no subgraph of $xu_1\cdots u_i$ is contained in $\c{F}$.  Note that the number of choices at each step is at least \[\half \ep_t\ell m^{1/b}-t-\ep \ell m^{1/b}\ge \quart \ep_t \ell m^{1/b},\]
		with the last step holding if $\ep$ is sufficiently small (which also implies $\ell m^{1/b}$ is sufficiently large compared to $\ep_t^{-1} t$).  This means \begin{equation}|\c{Q}|\ge \left(\quart \ep_t \ell m^{1/b}\right)^t.\label{eq:QPaths}\end{equation}  Note that by construction, every path in $\c{Q}$ avoids $\c{F}$.
		
		We now remove some paths from $\c{Q}$ to produce $\c{P}$.  If there exists $0\le i<j\le t$ with $(i,j)\ne (0,t)$ and vertices $u\in A_i,v\in A_j$ and $|\c{Q}_{i,j}[u\to v]|>\ell^{(j-i-1)b/(b-1)}$, then choose a path $xu_1\cdots u_t\in \c{Q}$ with $u_i=u$ and $u_j=v$ and delete this path from $\c{Q}$.  Repeat this until no such paths remain in $\c{Q}$, and let $\c{P}$ be the resulting set of paths.  By construction $(\c{A},\c{P})$ is a balanced neighborhood, so it suffices to show $|\c{P}|$ is large.
		
		We say that a pair of vertices $(u,v)$ is $(i,j)$-unbalanced if $u\in A_i,v\in A_j$ and $|\c{Q}_{i,j}[u\to v]|>\ell^{(j-i-1)b/(b-1)}$ (we emphasize that this condition involves  the original family $\c{Q}$ before any paths are deleted).  Let $\c{R}(i,j)=\{xu_1\cdots u_j\in \c{Q}_{0,j}:(u_i,u_j)\tr{ is }(i,j)\tr{-unbalanced}\}$.  We claim that \begin{equation}|\c{R}(i,j)|\le t^{-2}4^{-b-1}(\ep_t\ell m^{1/b})^j\label{eq:RPaths}\end{equation} for all $0\le i<j\le t$ with $(i,j)\ne (0,t)$.  Assuming this is true, this fact together with the branching factor of $\c{Q}$ implies that the number of paths removed is at most
		\[\sum_{i,j}|\c{R}(i,j)|(\ep_t\ell m^{1/b})^{t-j}\le \half  \left(\quart \ep_t\ell m^{1/b}\right)^t,\]
		with this last step holding if $\ell m^{1/b}$ is sufficiently large (which holds if $\ep$ is sufficiently small).
		From this and $\eqref{eq:QPaths}$, we conclude that the remaining set of paths $\c{P}$ has the desired properties and is as large as claimed.  It thus remains to prove \eqref{eq:RPaths}.
		
		Fix $(i,j)\ne (0,t)$ and let $s:=j-i$.  If $s=1$ then $\c{R}(i,j)=\emptyset$ (since no pair of vertices can be $(i,i+1)$-unbalanced), so we may assume $s\ge 2$.  Observe that
		\[|\c{R}(i,j)|\le \sum_{u\in A_i} |\c{R}(i,j)_{0,i}[x\to u]|\cdot |\c{R}(i,j)_{i,j}[u\to A_j]|\le (\ep_t \ell m^{1/b})^i\cdot \max_{u\in A_i}|\c{R}(i,j)_{i,j}[u\to A_j]|,\]
		where the second inequality used $\c{R}(i,j)_{0,i}\sub \c{Q}_{0,i}$ which has branching factor at most $\ep_t \ell m^{1/b}$.  Thus if we assume for contradiction that \eqref{eq:RPaths} does not hold, then there must exist some $u\in A_i$ such that
		\begin{equation}|\c{R}(i,j)_{i,j}[u\to A_j]|> t^{-2}4^{-b-1}(\ep_t\ell m^{1/b})^s \ge  s \ep_s(\ell m^{1/b})^s,\label{eq:Rij}\end{equation}
		with this last step holding if the $\ep_{s'}$ constants decrease sufficiently quickly.  Let \[B:=\{u_j\in A_j:\exists xu_1\cdots u_j\in \c{R}(i,j),\ u_i=u\}.\]  
		Note that $|\c{R}(i,j)[u,A_j]|\le (\ep_t\ell m^{1/b})^{j-i}$ because $\c{Q}$ has branching factor at most $\ep_t\ell m^{1/b}$, and that each $v\in B$ is the last vertex of more than $\ell^{(j-i-1)b/(b-1)}$ paths of $\c{R}(i,j)[u,A_j]$ (since by definition of $\c{R}(i,j)$, such a pair $(u,v)$ must be $(i,j)$-unbalanced).  Using $s=j-i$ gives
		\[|B|\le \f{(\ep_t\ell m^{1/b})^{s}}{\ell^{(s-1)b/(b-1)}}= \ep_t^s \ell^{(b-s)/(b-1)}m^{s/b}\le \ell^{(b-s)/(b-1)}m^{s/b}.\]
		With this and \eqref{eq:Rij}, we can apply Lemma~\ref{lem:RPaths} to $\c{R}(i,j)_{i,j}[u\to A_j]$ to conclude $t(u)\le s<t(x)$.  This gives a contradiction to $u\in A_i$ and the properties of $\c{A}$ guaranteed by Lemma~\ref{lem:XSet}.  
		
	\end{proof}

	A key fact about balanced neighorhoods is the following.
	\begin{lem}\label{lem:pathsAvoid}
		Let $(\c{A},\c{P})$ be a balanced $t$-neighborhood of $x\in V(G)$.  If $\ep$ is sufficiently small, then for any  $y\in A_t$ and non-empty set of vertices $S\sub V(G)\sm \{x,y\}$, there are at most $\ep^{-1}\ell^{(t-1-|S|)b/(b-1)}$ paths in $\c{P}[x\to y]$ containing $S$.
	\end{lem}
	\begin{proof}
		Let $S=\{z_1,\ldots,z_r\}$, and for ease of notation let $z_0=x$ and $z_{r+1}=y$.  Given a sequence $0=i_0<i_1<\cdots<i_r<i_{r+1}=t$, the number of paths $xu_1\cdots u_{t-1}y\in \c{P}[x\to y]$ with $u_{i_j}=z_j$ is at most
		\[\prod_{j=0}^{r}|\c{P}[z_j\to z_{j+1}]|\le \prod_{i=0}^r \ell^{(i_{j+1}-i_j-1)b/(b-1)}= \ell^{(i_{r+1}-i_0-(r+1))b/(b-1)}=\ell^{(t-1-|S|)b/(b-1)}.\]
		Every path containing $S$ can be formed in this way, possibly by reordering the elements of $S$ and by choosing different indices $i_j$.  As the number of ways of doing this is some finite number depending only on $b$, we conclude the result.
	\end{proof}
	
	We now move onto the last notion of neighborhoods that we need for this proof.
	\begin{defn}\label{def:refined}
		Let $(\c{B},\c{Q})$ be a balanced $t$-neighborhood of $x$.  We say that $(\c{B},\c{Q})$ is a \textit{refined $t$-neighborhood} of $x$ if the following conditions also hold: 
		\begin{enumerate}
			\item For every $i\in \{0,1,\ldots,t-1\}$ and every $u\in B_i$,
			\[|N(u)\cap B_{i+1}|\ge t^{-1} 4^{-2t} \ep_t \ell m^{1/b}.\]
			\item For every $v\in B_t$,
			\[|N(v)\cap B_{t-1}|\ge 4^{-2t}\ep_t^2 \ell^{(t-1)b/(b-1)}.\]
			\item For every $v\in B_t$,
			\[|\c{Q}[x\to v]|\ge 4^{-2t}\ep_t^t \ell^{(t-1)b/(b-1)}.\]
		\end{enumerate}
	\end{defn}
	\begin{lem}\label{lem:refined}
		If $(\c{A},\c{P})$ is a balanced $t$-neighborhood of a vertex $x\in X$ with $|\c{P}|\ge \half(\quart \ep_t \ell m^{1/b})^t$, then there exists a refined $t$-neighborhood $(\c{B},\c{Q})$ of $x$ with $B_i\sub A_i$ for all $i$ and $\c{Q}\sub \c{P}$ such that
		\[|\c{Q}|\ge \quart (\quart \ep_t \ell m^{1/b})^t.\]
	\end{lem}
	\begin{proof}
		Repeatedly delete vertices using the following three steps until no further vertices can be removed:
		\begin{enumerate}
			\item[Step 1] If there exists $i\in \{1,\ldots,t-1\}$ and $v\in A_i$ with 
			\[|N(v)\cap A_{i+1}|<t^{-1} 4^{-2t} \ep_t \ell m^{1/b},\]
			then remove $v$ from $A_i$ and remove all paths $P=xu_1\cdots u_t\in \c{P}$ with $u_i=v$.
			\item[Step 2] If there exists $v\in A_t$ with 
			\[|N(v)\cap A_{t-1}|<4^{-2t}\ep_t^2 \ell^{b/(b-1)},\]
			then remove $v$ from $A_t$ and remove all paths $P=xu_1\cdots u_t\in \c{P}$ with $u_t=v$.
			\item[Step 3] If there exists $v\in A_t$ with 
			\[|\c{P}[x\to v]|<4^{-2t}\ep_t^t \ell^{(t-1)b/(b-1)},\]
			then remove $v$ from $A_t$ and remove all paths $P=xu_1\cdots u_t\in \c{P}$ with $u_t=v$.
		\end{enumerate}
		Let $B_i\sub A_i$ and $\c{Q}\sub \c{P}$ be the set of vertices and paths that remain at the end of this process and let $\c{B}=(B_0,\ldots,B_t)$.   Note that with this, $(\c{B},\c{Q})$ automatically satisfies every condition for a refined $t$-neighborhood except possibly $|N(x)\cap B_1|\ge t^{-1} 4^{-2t} \ep_t \ell m^{1/b}$.  This will follow from having $\c{Q}$ large, which we prove below by arguing that few paths are destroyed in the process above.
		
		Because $\c{P}$ is a balanced $t$-neighborhood, its branching factor is at most $\ep_t \ell m^{1/b}$.  As such the number of paths removed in Step 1 is at most
		\begin{equation}t\cdot t^{-1} 4^{-2t}\cdot\ep_t \ell m^{1/b}\cdot (\ep_t \ell m^{1/b})^{t-1}= 4^{-t} \left(\rec{4}\ep_t \ell m^{1/b}\right)^t,\label{eq:Step1}\end{equation}
		and in Step 3 we remove at most
		\begin{equation}4^{-2t}\ep_t^t \ell^{(t-1)b/(b-1)}|A_t|\le 4^{-t}\left(\quart \ep_t \ell m^{1/b}\right)^t,\label{eq:Step3}\end{equation}
		where this last step uses the definition of balanced $t$-neighborhoods.
		
		For Step 2, we aim to show that the number of destroyed paths is at most
		\begin{equation}\label{eq:Step2}
			\rec{8} \left(\quart \ep_t \ell m^{1/b}\right)^t.
		\end{equation}
		Let $Z\sub A_t$ and $\c{P}(Z)$ denote the collection of vertices and paths removed in Step 2, and let
		\[Y=\{u\in A_{t-1}:\c{P}(Z)\tr{ has }(t-1,u)\textrm{-branching factor at least } 4^{-t-2}\ep_t \ell m^{1/b}\}.\]
		Note that by the definition of $Y$ and the bound on the branching factor of $\c{P}$,
		\begin{equation}|\{xu_1\cdots u_t\in \c{P}(Z):u_{t-1}\in A_{t-1}\sm Y\}|\le 4^{-t-2}(\ep_t \ell m^{1/b})^{t},\label{eq:notY}\end{equation}
		so it remains to show that there are few paths which use a vertex of $Y$ as the second to last vertex.  For this, let
		\[W=\{(u,v):u\in Y,\ v\in Z,\exists xw_1\cdots w_t\in \c{P}(Z),w_{t-1}=u,w_t=v\}.\]
		Note first that by definition of $Y$, $|W|\ge |Y| 4^{-t-2}\ep_t \ell m^{1/b}$.  On the other hand, we have that $|W|\le |Z| 4^{-2t}\ep_t^2 \ell^{b/(b-1)}$ since at the time each vertex $v\in Z$ is deleted, $v$ has at most $\ep_t^2 \ell^{b/(b-1)}$ neighbors in $A_{t-1}\supseteq Y$.  In total then we find
		\begin{equation}|Y|\le \f{|Z|4^{-2t}\ep_t^2 \ell^{b/(b-1)}}{4^{-t-2}\ep_t \ell m^{1/b}}\le 4^{-t+2}\ep_t \ell^{(b-t+1)/(b-1)}m^{(t-1)/b},\label{eq:YUpper}\end{equation}
		where this last step used $|Z|\le |A_t|\le \ell^{(b-t)/(b-1)}m^{t/b}$ by definition of balanced neighborhoods. 
		
		Observe that if the number of paths in $\c{P}(Z)$ using a vertex of $Y$ as the second to last vertex is at most $4^{-t-2} \left(\ep_t \ell m^{1/b}\right)^t$ then \eqref{eq:notY} implies that the number of paths removed is at most \eqref{eq:Step2}, so we may assume this is not the case. Letting $S:=\c{P}(Z)[x\to Y]$, this assumption together with the branching factor of $\c{P}$ implies $|S|\cdot \ep_t\ell m^{1/b}\le 4^{-t-2} \left(\ep_t \ell m^{1/b}\right)^t$, or equivalently
		\begin{equation}|S|\ge \f{4^{-t-2}(\ep_t \ell m^{1/b})^{t}}{\ep_t \ell m^{1/b}}> \ep_{t-1} (\ell m^{1/b})^{t-1},\label{eq:pathLower}\end{equation}
		With the last step holding if the $\ep_{s}$ constants decrease sufficiently quickly.  If $t-1\ge 2$, then \eqref{eq:YUpper} and \eqref{eq:pathLower} together with Lemma~\ref{lem:RPaths} imply $t(x)\le t-1$, contradicting $x\in X$.  If $t=2$ then \eqref{eq:YUpper} gives $|Y|\le 4^{-t-2} \ep_t \ell m^{1/b}$, so the fact that $\c{P}$ has branching factor at most $\ep_t \ell m^{1/b}$ implies that there are at most $4^{-t-2}(\ep_t\ell m^{1/b})^2$ paths in $\c{P}$ whose second to last vertex is in $Y$.  In either case, this bound together with \eqref{eq:notY} implies the number of paths removed is at most \eqref{eq:Step2}.
		
		As $t\ge 2$, \eqref{eq:Step1}, \eqref{eq:Step2}, \eqref{eq:Step3} imply that the total number of paths destroyed is at most $\quart (\quart \ep_t \ell m^{1/b})^t$, so $\c{Q}$ has the desired size.  To prove that $(\c{B},\c{Q})$ is a refined $t$-neighborhood, it remains to show $|N(x)\cap B_1|\ge t^{-1}4^{-2t}\ep_t \ell m^{1/b}$.  Since $\c{Q}\sub \c{P}$ has branching factor at most $\ep_t \ell m^{1/b}$, we have that $|\c{Q}|\le |N(x)\cap B_1|\cdot (\ep_t \ell m^{1/b})^{t-1}$.  Our bound on $|\c{Q}|$ then implies
		\[|N(x)\cap B_1|\ge \quart (\quart \ep_t \ell m^{1/b})^t\cdot  (\ep_t \ell m^{1/b})^{1-t}=4^{-t-1}\ep_t\ell m^{1/b}.\]
		This gives the desired bound, proving the result.
	\end{proof}
	
	\begin{proof}[Proof of Proposition~\ref{prop:expansion}]
		Let $t,X$ be as in Lemma~\ref{lem:XSet}.  For each $x\in X$, let $(\c{A},\c{P})$ be the balanced $t$-neighborhood guaranteed by Lemma~\ref{lem:balanced} and $(\c{B},\c{Q})$ the refined $t$-neighborhood guaranteed by Lemma~\ref{lem:refined} from $(\c{A},\c{P})$.  Most of the properties of Proposition~\ref{prop:expansion} follow immediately from Definitions~\ref{def:balanced} and \ref{def:refined}, as well as Lemmas~\ref{lem:balanced} and \ref{lem:pathsAvoid} (with the last lemma using that $(\c{B},\c{Q})$ is in particular a balanced $t$-neighborhood).  The only conditions which are not immediate are the bounds $|B_{t-1}|,|B_t|\ge \ep \ell^{(b-t+1)/(b-1)}m^{(t-1)/b}$.  If this bound did not hold for $B_{t-1}$, then the tuple $(B_0,B_1,\ldots,B_{t-1})$ would be a concentrated $(t-1)$-neighborhood of $x$ (assuming $t^{-1} 4^{-2t} \ep_t\ge \ep_{t-1}$), contradicting every $x\in X$ having $t(x)=t$.
		
		To prove the bound on $B_t$, we use Lemma~\ref{lem:pathsAvoid} to find
		\[|\c{Q}|=\sum_{u\in B_1,v\in B_t} |\c{Q}[u\to v]|\le \ep^{-1} \ell^{(t-2)b/(b-1)}\cdot |B_1|\cdot |B_t|\le \ep^{-1} \ell^{(t-2)b/(b-1)}\cdot \ell m^{1/b}\cdot |B_t|,\]
		where this last step used Definition~\ref{def:balanced}(i).  As $|\c{Q}|\ge \ep \ell^t m^{t/b}$, this gives $|B_t|\ge \ep^2  \ell^{(b-t+1)/(b-1)}m^{(t-1)/b}$. This gives the desired result after replacing $\ep$ in the proposition statement with $\ep^2$ (which easily implies the result after replacing $\ep$ with $\sqrt{\ep}$ throughout).
	\end{proof}
	
\end{document}